\documentclass{amsart}
\usepackage{amsmath}
\usepackage{amssymb}
\usepackage{eucal}
\usepackage{amscd}
\usepackage{delarray}
\usepackage{txfonts}

\DeclareMathOperator{\im}{im}

\DeclareMathOperator{\supp}{supp}

\DeclareMathOperator{\dom}{dom}

\newcommand{\norm}[1]{\lVert#1\rVert}

\newcommand{\fhi}{\varphi}

\renewcommand{\o}{\circ}

\renewcommand{\epsilon}{\varepsilon}
\newcommand{\FF}{\mathcal{F}}

\newcommand{\GG}{\mathcal{G}}

\newcommand{\MM}{\mathcal{M}}

\newcommand{\HH}{\mathcal{H}}
\newcommand{\VV}{\mathcal{V}}
\newcommand{\UU}{\mathcal{U}}

\newcommand{\OO}{\mathcal{O}}

\newcommand{\WW}{\mathcal{W}}
\newcommand{\NN}{\mathcal{N}}

\newcommand{\Z}{\mathbb{Z}}
\newcommand{\C}{\mathbb{C}}

\newcommand{\E}{\mathbb{E}}
\newcommand{\R}{\mathbb{R}}
\newcommand{\N}{\mathbb{N}}

\newtheorem{theorem}{Theorem}[section]
\newtheorem{lemma}[theorem]{Lemma}
\newtheorem{cor}[theorem]{Corollary}
\newtheorem{prop}[theorem]{Proposition}
\newtheorem{claim}{Claim}

\newtheorem{problem}{Problem}

\theoremstyle{definition}
\newtheorem{defn}[theorem]{Definition}
\newtheorem*{rem}{Remark}
\newtheorem*{rems}{Remarks}

\newtheorem{example}[theorem]{Example}
\numberwithin{equation}{section}

\title[equicontinuous foliated spaces]{equicontinuous foliated spaces}
\begin{document}

\author[Jes\'us A. \'Alvarez L\'opez]{J. A. \'Alvarez L\'opez*}
\address{$^*$Departamento de Xeometr\'\i a e Topolox\'\i a \\
Facultade de Matem\'aticas \\
Universidade de Santiago de Compostela \\
Campus Universitario Sur \\
15706 Santiago de Compostela \\Spain }
\thanks{*Research of the first author supported by DGICYT Grant
PB95-0850} 
\email{jesus.alvarez@usc.es}
\author[Alberto Candel]{A. Candel$^\dagger$}
\address{$^\dagger$ Department of Mathematics \\ CSUN \\ Northridge, CA
91330 \\ U.S.A.}

\thanks{$^\dagger$Research of the second author supported by NSF Grants}
\email{alberto.candel@csun.edu}



\maketitle

\tableofcontents

\section*{Introduction}

The theme of this paper originates in the study of the generic
geometry of leaves of a foliated space. In \cite{jes}, we analyze the
problem of when all (or almost all) the leaves of a foliated space are
quasi-isometric. In this paper, a dynamical condition on a foliated
space guaranteeing that all the leaves without holonomy are
quasi-isometric will be discussed. The condition is on the structure
of the action of the holonomy pseudogroup, and is called
equicontinuity. That such condition may imply that all leaves without
holonomy are quasi-isometric comes from the structure theorem of
Riemannian foliations. These foliations are the models of
equicontinuous foliated spaces, and P.~Molino's work \cite{Mol88}
describing their structure has as a consequence that the holonomy
covers of all the leaves are quasi-isometric via a
diffeomorphism. Indeed, given a compact, connected manifold $M$
endowed with a Riemannian foliation, Molino shows that there is a
fiber bundle over $M$ with compact fiber (the transverse orthonormal
frame bundle), and with a foliation transverse to the fibers whose
leaves are the holonomy covers of the leaves of $M$. Furthermore,
there is a group of automorphisms of this bundle which permutes the
leaves.

The concept of Riemannian foliation is easily formulated by saying
that the holonomy pseudogroup is a pseudogroup of local isometries of
a Riemannian manifold.  This generalizes to equicontinuous
pseudogroups of local transformations of topological spaces. As such,
the concept of equicontinuity appears in R.~Sacksteder's paper
\cite{Sak}. The parallelism between Riemannian foliations and
equicontinuous pseudogroups is developed by E.~Ghys in \cite[Appendix
E]{Mol88}, see also M.~Kellum's paper \cite{kellum} for this connection.

Thus, in certain measure, what is done in this paper may be seen as a
generalization of that aspect of Molino's theory pointed out above.
First we show that all leaves without holonomy have the same coarse
quasi-isometry type. For general equicontinuous foliated spaces, it
seems not to be possible to obtain this result, the main obstruction
being the very general structure of the transverse models. This
obstruction can be overcome by imposing certain regularity to the
transverse structure. The proof of this result requires certain amount
of work on pseudogroups and on the geometric structure of their
orbits, and on how to pass from coarse quasi-isometries between orbits
to leaves.

When the foliated space is smooth, then it is possible to introduce a
metric tensor on the leaves that varies continuously from leaf to
leaf. In this case the above results can be improved by using
quasi-isometries via diffeomorphisms between leaves. Moreover, for
general equicontinuous foliated spaces, it can be shown that the
universal covers of all leaves are quasi-isometric to each other via
diffeomorphisms. These results are obtained with the help of
normal bundle theory.

From our study of pseudogroups, it also follows that equicontinuous
foliated spaces (with some mild conditions) satisfy two other typical
properties of Riemannian foliations. First, the leaf closures are
homogeneous spaces and form a partition. Second, the holonomy
pseudogroup is indeed given by local isometries with respect to some
``local metric,'' and has a closure in certain sense. The existence of
this closure of the holonomy pseudogroup is an important ingredient of
our topological description of Riemannian foliations with dense leaves
given in \cite{Alv-Candel:Riemannian}.

The main results of this paper and \cite{Alv-Candel:Riemannian} were
conjectured and greatly justified by E.~Ghys \cite[Appendix~E]{Mol88}.
The concept of equicontinuity for foliations was also studied by
M.~Kellum \cite{kellum} and C.~Tarquini \cite{Tarquini}.  Kellum dealt
with the more restrictive setting of transversely quasi-isometric
foliations, and Tarquini showed that equicontinuous transversely
conformal foliations are Riemannian, which also follows from our main
result of \cite{Alv-Candel:Riemannian} in the case of dense leaves.

\medskip

{\bf Acknowledgments.} The work of the first author was supported by
DGICYT Grant PB95-0850, and that of the second by NSF Grants No.
DMS-9973086 and DMS-0049077. The authors were at the Universidade de
Santiago de Compostela and at CSUN while this paper was
being written, and thank these institutions for their support.
We also thank F.~Alcalde Cuesta for helpful conversations.
%
%

\section{Foliated spaces}\label{fol.space}

This section collects and develops some information on foliated spaces
which will be used later.  General references on foliated spaces are
\cite{Hector-Hirsch}, \cite{MS}, \cite{CC2K}.

The definition of the concept of foliated space (or lamination)
requires that of smooth function. Let $Z$ be a Polish space (
{\it i.e.}, a completely metrizable separable space) and let $U$ be an open
set in $\R^n\times Z$. A map $f:U\to \R^p$ is smooth (of class $C^k$)
at the point $(x_0,z_0)$ if there is a neighborhood of $(x_0,z_0)$ in
$U$ of the form $D\times Y$ (where $D$ is an open ball in $\R^n$ and
$Y$ is an open subset of $Z$) such that $f$ is continuous on $D\times
Y$ and the partial derivatives up to order $\le k$ of all coordinate
functions $f_i$ of $f$ exist and are continuous at all points
$(x,z)\in D\times Y$.

A locally compact Polish space $X$ 
is said to have the structure of a
\textit{foliated space} (of class $C^k$) if there is a collection of
charts (flow boxes) $(U_i,\phi_i)$, where $\{U_i\}$ is a covering of $X$
by open sets, the maps $\phi_i$ are homeomorphisms $\phi_i:U_i\to
B_i\times Z_i$, for Polish spaces $Z_i$ and open balls $B_i$ of finite
radius in $\R^n$, such that the coordinate changes are of the form 
$$
\phi_i\circ \phi_j^{-1}(x_j,z_j) =
(x_i(x_j,z_j),z_i(z_j))\;,
$$ 
where $x_i:\phi_j(U_i\cap U_j) \to \R^n$ is
of class $C^k$ and $z_i$ is continuous. The sets
$\phi_i^{-1}(B_i\times \{z\})$ are called plaques.

Since the foliated space $X$ is locally compact, given an open cover
$\UU=\{U_i,\phi_i\}$ by flow boxes of $X$, it is always possible to
find a locally finite covering $\{V_\alpha,\fhi_\alpha\}$ by flow
boxes which is a refinement of $\UU$ in the sense that each $V_\alpha$
has compact closure on some $U_{i(\alpha)}$ and the corresponding
chart $\phi_{i(\alpha)}$ extends $\fhi_\alpha$. Such cover
$\{V_\alpha\}$ is called regular \cite{Hector-Hirsch}, \cite{CC2K}.

A synonym of the term foliated space is lamination, which is
sometimes reserved for a foliated space which is embedded as a closed
subset of a manifold.  It may also be convenient to denote the
foliated space by $(X,\FF)$, with $\FF$ referring to the particular
foliated structure on $X$.

The foliated structure of a space $X$ induces a locally euclidean
topology on $X$, the basic open sets being the open subsets of the
plaques, which is finer than the original. The connected components of
$X$ in this topology are called leaves. The smooth structure implies
that each leaf is a connected manifold of dimension $n$ and of class
$C^k$. If $x$ is a point of $X$, the leaf which contains $x$ will
usually be denoted by $L_x$.

Concepts of manifold theory readily extend to foliated spaces.  In
particular, if $\FF$ is at least of class $C^1$, there is a continuous
vector bundle $TX$ over $X$ whose fiber at each point $x$ of $X$ is
the tangent space of the manifold $L_x$ at $x$.

Some very basic properties of foliated spaces will now be listed. They
easily follow from the definition and standard techniques of manifold
theory extended to ``manifolds with parameters.'' A basic observation
is the following \cite{MS}, which is immediate from the
paracompactness of Polish spaces and the local structure of foliated
spaces.

\begin{prop}\label{partun}
Every open cover of a foliated space of class $C^k$ admits a
subordinate partition of unity of class $C^k$.
\end{prop}

The important consequence is that a foliated space of class $C^k$,
with $k\ge1$, always admits a metric tensor;  {\it i.e.}, and inner
product on $TX$ inducing a Riemannian metric on each leaf, and varying
continuously from leaf to leaf. The induced distance on a leaf $L$
will be denoted by $d_L$. If $X$ is compact, then each leaf is a
complete Riemannian manifold of bounded geometry. (This may not be the
case in general.)

Another consequence of the existence of partitions of unity is the
following. The proof is like that for manifolds~\cite{CC2K}.

\begin{prop}
  If $X$ is a compact foliated space of class $C^k$, $k\ge 1$, then
  there is a $C^k$ embedding $\varphi $ of $X$ in the separable real
  Hilbert space $\E$. Moreover, a given metric tensor along the leaves
  can be extended to a metric tensor on $\E$.
\end{prop}

A given metric tensor along the leaves of $X$ admits an extension to a
metric tensor on $\E$. This is so because a metric tensor may be
viewed as a section of some bundle over $\E$ with contractible fibers,
and  $X$ being closed in $\E$ implies that a section over $X$ extends
to one over $\E$. It may only be continuous, of course, but that
is sufficient  to define length.

If $X$ is compact, then any two metric tensors on the leaves of $X$
are quasi-isometric, and so, for the purposes of this paper, the
metric tensor induced from the standard metric of $\E$ will be taken
by default.

This embedding of a foliated space $X$ in $\E$ also gives rise to a
normal vector bundle, the fibers of which are isomorphic to $\E$. This
structure permits to formalize concepts and results like ``local
projection of leaves onto leaves'' or ``Reeb's stability.'' This
structure will be described presently.

Let $X$ be a compact foliated space, of class $C^k$ with $k\ge 1$,
embedded in the Hilbert space $\E$. The restriction of the embedding
to each leaf is not an embedding, but only an injective immersion. The
smoothness of $X$ being at least $C^1$ implies that the map which
assigns to a point $x\in X$ the subspace $T_xX$ of $\E$ is continuous
(as a map of $X$ into the space of $n$-dimensional subspaces of
$\E$). It follows that if $F$ is a subspace complementary to one
$T_xX$ in $\E$, then it is also complementary to $T_yX$ for $y$ close
to $x$.

The key point is that each tangent space $T_xX$ is a finite
dimensional subspace of $\E$, hence is closed and has an orthogonal
complement.  If $i:L\to \E$ is the inclusion of a leaf of $X$ in $\E$,
then there are charts about $x$ in $L$ and $\E$ such that the
corresponding local representation of $i$ is of the form
$y\mapsto(y,0)$, provided by the implicit function theorem (flat
coordinates). In these flat coordinates, $T_yX^\perp = \{y\}\times \E$
for $y$ in this plaque containing $x$, and by continuity, the affine
subspaces $y+T_yX^\perp$ meet nearby leaves transversely. Since $X$ is
compact, it follows that this holds for all $y$ in a neighborhood of
radius $r$ about $x$ in $X$, the radius $r$ being independent of the
point $x$.

\begin{theorem}\label{quasibundle}
Let $X$ be a compact foliated space embedded in $\E$ as above, of
class $C^2$.
Let $i:L\to \E$ denote the inclusion of a leaf $L$ in
$X\subset \E$. Then there is a vector bundle $p:N\to L$ and a
neighborhood $W$ of the zero section of $N$ such that the following
hold:
\begin{itemize}
\item[(1)] The map $i:L\to \E$ extends to a local diffeomorphism
$\varphi:W\to \E$;
\item[(2)] there is a foliated space $Y\subset W$, of the same
dimension as 
$X$, having $L$ as a leaf and transverse to the fibers of $N$; and
\item[(3)] as foliated spaces, $Y=\varphi^{-1}(X\cap \varphi(W))$ and
the restriction of $p$ to each leaf of $Y$ is a local diffeomorphism
into $L$.
\end{itemize}
\end{theorem}

Let $N(X)=\{ (x,v)\mid v\perp T_xX\} \subset X\times \E$. Then the
exponential map $\lambda:N(X)\to \E$ is $\lambda(x,v)= x+v$. Let
$N(X,\epsilon)$ denote the set of pairs $(x,v)\in N(X)$ with
$\norm{v}<\epsilon$. The normal bundle $N(L)$ to each leaf $L$ is
contained in $N(X)$.

The differential $\lambda_*$ at each point $(x,0)\in N(X)$ is the
identity (under the canonical identification $T_{(x,0)}N(X) =
T_xX\times T_xX^\perp$). Therefore, by continuity of $\lambda_*$ on
$N(X)$ and compactness of $X$, there exists $\epsilon>0$ such that
$\lambda_*$ is an isomorphism at $(x,v)$ for $x\in X$ and
$\norm{v}<\epsilon$.  This does not mean that $\lambda$ is locally a
homeomorphism. What it means is the following:

\begin{lemma} \label{loc.dif}
For each point $x\in X$ there exists a neighborhood
$U_x$ in the leaf through $x$ such that $\lambda$ is a diffeomorphism
of $N(U_x,\epsilon)$ into $\E$.
\end{lemma}

It will be shown that there exists $\epsilon >0$ and $r>0$ such that
$\lambda$ is a diffeomorphism of $N(B(x,r),\epsilon)$ into $\E$, for
every $x\in X$.  Otherwise, by compactness of $X$, there exist
sequences $(x_n,v_n)\ne (y_n,w_n)$ in $N(X)$, with
$\norm{v_n},\norm{w_n}\to 0$, $x_n,y_n\to x$, and such that
$\lambda(x_n,v_n)\ne \lambda (y_n,w_n)$ for each $n$.
Working on a local flow box around $x$ in $X$, choosing flat
coordinates around the plaque through $x$, and taking into account
that the normal subspace to points of $X$ varies continuously, it is
then obvious that it is possible to find new sequences $(x'_n,v'_n)$
and $(y'_n,w'_n)$ as above with the added property that all the points
$x'_n$ and $y'_n$ are in the same plaque as $x$. This contradicts
Lemma~\ref{loc.dif}.

The vector bundle $N(X)$ has a metric on the fibers so that the
continuous map $\phi(x,v) = \lambda(x,v) - x$ is a linear isometry on
each fiber.  By compactness of $X$, there
 exists $\epsilon_0>0$ such that for each leaf $L$ of $X$ the
restricted map $\lambda:N(L,\epsilon_0) \to \E$ is a local
diffeomorphism. Moreover, there exists $r>0$ such that for every leaf
$L$ and every point $x\in L$, the map $\lambda:
N(B(x,r),\epsilon_0)\to \E$ is a diffeomorphism.

It follows that if $L$ is a leaf of $X$, then the neighborhood
$N(L,\epsilon_0)$ contains a foliated space $Y$, which is lifted from
$X$ via the local diffeomorphism $\lambda$.  This completes the proof
of Theorem~\ref{quasibundle}.

\begin{rem}
  The hypothesis that $X$ be of class $C^2$ is required so that the
  normal bundle and exponential map be of class $C^1$. On the other
  hand, a manifold of class $C^1$ admits a compatible structure of
  class $C^r$, for any $1\le r\le \infty$. The proof
  (see~\cite[Chapter 2]{hirsch}) is based on approximation results
  that can be adapted to `manifolds with parameters', that is, to
  foliated spaces. Since it is out of place to do this here, the $C^2$
  hypothesis is kept here and in Theorems~\ref{t:equicontinuous
    foliated spaces},~\ref{t:equicontinuous quasi-effective foliated
    spaces}.
\end{rem}

\begin{rem}
  Let $x\in L$, and let $T$ be a transversal through $x$, which may be
  taken to lie in the affine subspace $x+ T_x X^\perp$ of $\E$. If
  $\GG(x)$ denotes the group of germs of local homeomorphisms of $T$
  which fix $x$, then there is the germinal holonomy representation
  $$\pi_1(L,x)\to \GG(x).$$
  The construction above shows that this
  representation is equivalent to that of $L$ as a leaf of the
  foliated space $Y$ of $N(L,\epsilon_0)$.
\end{rem}

The following two propositions complement Theorem~\ref{quasibundle}.

\begin{prop}\label{nor.1}
  Let $L$ be a leaf of the compact foliated space of
  Theorem~\ref{quasibundle}. Then there
  exists $\epsilon>0$ and an $\epsilon$-disc bundle $W\to L$
  which carries a foliated space $Y$ and there exists a constant $K>0$
  such that the projection $p:S\to L$ of each leaf $S$ of $Y$ into
  $L$ is a local diffeomorphism \upn{(}of class $C^1$\upn{)} of
  distortion in the interval $[1/K,K]$.
\end{prop}

The disk bundle $W\to L$ and foliated space $Y\subset W$ with the
stated properties have been constructed in the previous proof. Because
the space $X$ is compact, it admits a finite regular cover by flow
boxes. The projection $Y\to L$ amounts to finitely many projections
between plaques of these flow boxes.  As the cover is finite and
regular, there are global bounds for the distortion of these
projections within flow boxes. The claim made in the proposition
follows from these observations.

A similar useful fact is the following.

\begin{prop}
Let $L$ be as in Proposition~\ref{nor.1}. Given $R>0$, there
exists $\delta>0$ such that, if $x\in L$ and $y$ is a point of $W$
with $p(y)=x$, and such that the distance in the fiber of $W$
between $x$ and $y$ is $<\delta$, then the ball $B(y,R)$ in the leaf
of $Y$ through $y$ is contained in $W$.
\end{prop}
\begin{proof}
The proof of this fact uses again a finite regular cover by flow
boxes. The local coordinate changes in the transverse direction are
uniformly continuous maps, and only a finite number of coordinate
changes are required to run through a ball of radius $R$, because
plaques have a definite size.
\end{proof}

\begin{rem}
  The construction of Theorem~\ref{quasibundle} and the discussion
  that follows it (especially Proposition~\ref{nor.1}), extend to a
  covering $\pi:L'\to L$ of a leaf $L$, simply by considering the
  normal bundle of the immersion $L'\to L\hookrightarrow \E$. The
  output is a disc bundle $W\to L'$ containing a foliated space $Y$
  transverse to the fibers and having $L'$ as a leaf. The projection
  of leaves of $Y$ into $L$ is a local diffeomorphism of bounded
  distortion.
\end{rem}

Continuing with the notation so far introduced, let $L'\to L$ be a
covering of a leaf $L$ of $X$, let $x\in L'$ and let $D$ be a compact
manifold with boundary. Let $y\in D$ and let $f:D\to L'$ be a
continuous map with $f(y)=x$. The following is a Reeb stability type
of result.
  
\begin{prop}\label{reeb}
Suppose that the induced homomorphism $\pi_1(D,y) \to \pi_1(L',x)$
takes $\pi_1(D,x)$ into the kernel of the holonomy representation of
$L'$ \upn{(}as a leaf of the foliated space $Y$\upn{)}. Then there is
a transversal $Z\subset Y$ through $x$ and a smooth map $F$ from the
product foliation $D\times Z$ into $X$ such that $F|D\times \{x\}=f$ and
$F(D\times \{y\})$ is contained in the leaf through $y$, for all $y\in Z$.
\end{prop}

If $f:D\hookrightarrow L$ is an embedding, $Z$ can be chosen so that
$F$ embeds $D\times Z$ in $Y$.  In particular, if $\pi_1(L',x) \to
\pi_1(L)$ has image contained in the kernel of the holonomy
representation of $L$ as a leaf of $X$, then, given $x\in L'$ and
$R>0$ arbitrarily large, there exists $\delta>0$ such that if $y$ is a
point in the fiber $p^{-1}(x)$ at distance $<\delta$ from $x$, then
the ball $B(y,KR)$ contains $p^{-1}(B(x,R))\cap Y$, and the component
of this last set which contains $y$, contains the ball $B(y, R/K)$.

Moreover, the absence of holonomy permits to choose $\delta$ so that
there is a transversal $Z$ in the fiber through $x$ such that the
union of the leaves of $p^{-1}(B(x,R))\cap Y$ through points of $Z$ is
parametrized as a product $B(x,R)\times Z$, and the induced metric on
the leaves is at bounded distance from that of $B(x,R)$.

\section{Pseudogroups of local transformations}

A {\em pseudogroup of local transformations\/} of a topological space
$Z$ is a collection $\HH$ of homeomorphisms between open subsets of
$Z$ that contains the identity on $Z$ and is closed under composition
(wherever defined), inversion, restriction and combination of
maps. Such pseudogroup $\HH$ is {\em generated\/} by a set
$E\subset\HH$ if every element of $\HH$ can be obtained from $E$ by
using the above pseudogroup operations; to simplify arguments, the
sets of generators to be considered will be assumed to be symmetric
($h^{-1}\in E$ if $h\in E$).

The {\em orbit\/} of an element $x\in Z$ is the set $\HH(x)$ of
elements $h(x)$, for all $h\in\HH$ whose domain contains $x$. These
orbits are the equivalence classes of an equivalence relation on $Z$.
Note that an arbitrary equivalence relation $R\subset Z\times Z$ is
defined by a pseudogroup on $Z$ if and only if $R$ is a union of sets
$R_i$, $i\in I$, such that the restriction to each $R_i$ of both
factor projections $Z\times Z\to Z$ are homeomorphisms onto open
subsets. Indeed, take the sets $R_i$ to be the graphs of all local
transformations in the pseudogroup. Moreover $R$ is defined by a
countably generated pseudogroup on $Z$ if and only if $R$ is a
countable union of sets $R_i$ satisfying the above condition. This
follows because a countable set of local transformations of $Z$ gives
rise to a countable family of composites with maximal domain.

The set of germs of all transformations in the pseudogroup $\HH$ at
all points of their domains, endowed with the \'etale topology, is a
topological groupoid, product and inversion being induced by
composition and inversion of maps, respectively. Thus, for each $x\in
Z$, the set of germs at $x$ of transformations $h\in\HH$ with
$x\in\dom h$ and $h(x)=x$ is a group called the group of germs at
$x$. If $x,y\in Z$ are in the same $\HH$-orbit, then the groups of
germs at $x$ and $y$ are isomorphic: an isomorphism is given by
conjugation with the germ at $x$ of any transformation $g\in\HH$ with
$x\in\dom g$ and $g(x)=y$.  The group of germs of an orbit is
therefore well defined, up to isomorphisms, as the group of germs at
any point of that orbit. In particular, a distinguished type of orbits
are those with trivial group of germs.

Pseudogroups of local transformations must be thought of as natural
generalizations of group actions on topological spaces (each group
action generates a pseudogroup). But the main example to keep in mind
is the holonomy pseudogroup of a foliated space $(X,\FF)$ associated
to a regular covering by flow boxes $(U_i,\phi_i)$, whose construction
is now recalled. If $\phi_i:U_i\to B_i \times Z_i$ for Polish spaces 
$Z_i$ and open balls $B_i$ of finite radius in
$\R^n$, 
%
%
%
%
let $p_i:U_i\to Z_i$ denote the composite of $\phi_i$ with the
factor projection $\R^n \times Z_i\to Z_i$; the fibers of these $p_i$
are the plaques.  If $U_i$ meets $U_j$, let $Z_{i,j}=p_i(U_i\cap
U_j)$, and regularity of the cover permits to define a homeomorphism
$h_{i,j}:Z_{i,j}\to Z_{j,i}$ such that $p_j=h_{i,j}\circ p_i$ on
$U_i\cap U_j$ \cite{CC2K, Hector-Hirsch}.
(If the covering by flow boxes is not regular, one can define
generators of a pseudogroup via local sections of the projections
$p_i$.)
Such a collection $(U_i,p_i,h_{i,j})$ is called a defining cocycle
for $\FF$ \cite{Haefliger85,Haefliger88}. These $h_{i,j}$ generate a
pseudogroup $\HH$ of local transformations of $Z=\bigsqcup_iZ_i$,
which is called the \textit{holonomy pseudogroup} of $(X,\FF)$ (with
respect to the covering $(U_i,\phi_i)$).

There is a canonical bijection between the set of leaves and the set
of $\HH$-orbits, which is given by $L\mapsto\HH(p_i(x))$ if $x\in
L\cap U_i$. Each $Z_i$ can be considered as a local transversal of
$\FF$ via $\phi_i$ and the identification $Z_i\equiv\{0\}\times
Z_i\subset B_i\times Z_i$. It may be assumed that all of these local
transversals are disjoint from
each other, and thus that $Z$ is embedded
in $X$ as a complete transversal. Each $\HH$-orbit injects into the
corresponding leaf in this way.

The holonomy groups of the leaves can be defined as the groups of
germs of the corresponding orbits. Thus leaves with trivial holonomy
group correspond to orbits with trivial group of germs. Moreover, with
the same arguments of \cite{Hector-Hirsch}, it follows that, for a
general pseudogroup $\HH$ of local transformations of a topological
space $Z$, if $\HH$ has a countable set of generators, then the union
of orbits with trivial group of germs is a residual subset of $Z$; in
particular, this union is dense in $Z$ if $Z$ is a Polish space.

It is well known that all defining cocycles of a foliated space induce
holonomy pseudogroups that are equivalent in the sense given by the
following definition; thus the relevant properties of pseudogroups of
local transformations of a topological space are those that are
invariant by these equivalences.

\begin{defn}[{Haefliger \cite{Haefliger85,Haefliger88}}]
Let $\HH,\HH'$ be pseudogroups of local transformations of topological
spaces $Z,Z'$, respectively. An {\em \'etale morphism\/}
$\Phi:\HH\to\HH'$ is a maximal collection $\Phi$ of homeomorphisms of
open subsets of $Z$ to open subsets of $Z'$ such that:
\begin{itemize}

\item If $\phi\in\Phi$, $h\in\HH$ and $h'\in\HH'$, then
$h'\circ\phi\circ h\in\Phi$;

\item the sources of elements of $\Phi$ form a covering of $Z$;
and

\item if $\phi,\psi\in\Phi$, then $\psi\circ\phi^{-1}\in\HH'$.

\end{itemize}
An \'etale morphism $\Phi:\HH\to\HH'$ is called an {\em equivalence\/}
if the collection $\Phi^{-1}=\{\phi^{-1}\ |\ \phi\in\Phi\}$ is also an
\'etale morphism $\HH'\to\HH$, which is called the {\em inverse\/} of $\Phi$.
The {\em composite\/} of two \'etale morphisms $\Phi:\HH\to\HH'$ and
$\Psi:\HH'\to\HH''$ is the collection $\Psi\circ\Phi=\{\psi\circ\phi\ |\
\phi\in\Phi,\ \psi\in\Psi\}$, which is an \'etale morphism $\HH\to\HH''$.
Finally, an \'etale morphism $\Phi:\HH\to\HH'$ is {\em generated\/} by a subset
$\Phi_0\subset\Phi$ if all the elements of $\Phi$ can be obtained by
restriction and combination of composites $h'\circ\phi\circ h$ with
$h\in\HH$, $\phi\in\Phi_0$ and $h'\in\HH'$.
\end{defn}
 
An \'etale morphism $\Phi:\HH\to\HH'$ clearly induces a continuous map
between the corresponding spaces of orbits, $\bar\Phi:Z/\HH\to
Z'/\HH'$, which is a homeomorphism if $\Phi$ is an equivalence.

A basic example of a pseudogroup equivalence is the following. Let $\HH$ be a
pseudogroup of local transformations of a space $Z$, let
$U\subset Z$ be an open subset that meets every
$\HH$-orbit, and let $\GG$ denote the restriction of $\HH$ to $U$. Then
the inclusion map $U\hookrightarrow Z$ generates an equivalence
$\GG\to\HH$. In fact, this example can be used to describe any pseudogroup
equivalence in the following way. Let $\HH,\HH'$ be pseudogroups
of local transformations of spaces $Z,Z'$, and $\Phi:\HH\to\HH'$ an
equivalence. Then let
$\HH''$ be the pseudogroup of local transformations of $Z''=Z\sqcup Z'$ generated
by $\HH\cup\HH'\cup\Phi$. Then the inclusions of $Z,Z'$ in $Z''$ generate
equivalences $\Psi_1:\HH\to\HH''$ and $\Psi_2:\HH'\to\HH''$ so that
$\Phi=\Psi_2^{-1}\circ\Psi_1$.

For pseudogroups of local transformations of locally compact spaces, the
following result characterizes the existence of relatively compact open subsets
that meet all orbits.

\begin{lemma}\label{l:compact space of orbits}
Let $\HH$ be a pseudogroup of local transformations of a locally compact $Z$.
The orbit space $Z/\HH$ is compact if and only if there exists a
relatively compact open subset that meets every $\HH$-orbit.
\end{lemma}

\begin{proof}
If an open subset $U\subset Z$ meets every $\HH$-orbit, then the restriction $U\to
Z/\HH$ of the quotient map $Z\to Z/\HH$ is onto.  Thus $Z/\HH$ is
compact because it is a continuous image of the compact space
$\overline{U}$.

Assume that $Z/\HH$ is compact. Since $Z$ is locally compact, each
$x\in Z$ has relatively compact open neighborhood $U_x$. Let $Q_x$
denote the image of $U_x$ by the quotient map $Z\to Z/\HH$. Since
$Z/\HH$ is compact, its open covering $\{ Q_x\mid x\in Z\}$ has a
finite sub-covering $Q_{x_1},\dots,Q_{x_m}$.  Then the open set
$U=U_{x_1}\cup\dots\cup U_{x_n}$ of $Z$ is relatively compact and
meets all orbits of $\HH$.
\end{proof}

Examples of pseudogroups with compact space of orbits are the holonomy
pseudogroups of compact foliated spaces, as orbit and leaf spaces can be
identified. But such pseudogroups satisfy a stronger compactness
condition that is defined as follows.

\begin{defn}[{Haefliger \cite{Haefliger85}}]\label{d:compactly generated}
Let $\HH$ be a pseudogroup of local transformations of a locally
compact space $Z$.  Then $\HH$ is {\em compactly generated\/} if there
is a relatively compact open set $U$ in $Z$ meeting each orbit of
$\HH$, and such that the restriction $\GG$ of $\HH$ to $U$ is
generated by a finite symmetric collection $E\subset\GG$ so that each
$g\in E$ is the restriction of an element $\bar g$ of $\HH$ defined on
some neighborhood of the closure of the source of $g$.
\end{defn}

It was observed in \cite{Haefliger85} that this notion is invariant by
equivalences and that the relatively compact open set $U$ meeting each
orbit can be chosen arbitrarily. If $E$ satisfies the conditions of
Definition~\ref{d:compactly generated}, it will be called a {\em
system of compact generation\/} of $\HH$ on $U$.

\section{Coarse quasi-isometries}

The concept of coarse quasi-isometry was introduced by M.~Gromov in
\cite{Gromov93}; it is also called {\em rough isometry\/} in the
context of potential theory \cite{Kanai85}.  A {\em net\/} in a metric
space $M$, with metric $d$, is a subset $A\subset M$ such that
$d(x,A)<C$ for some $C>0$ and all $x\in M$; the term $C$-net is also
used.  A {\em coarse quasi-isometry\/} between $M$ and another
metric space $M'$ is the choice of a bi-Lipschitz bijection between
nets of $M$ and $M'$; in this case, $M$ and $M'$ are said to have the
same {\em coarse quasi-isometry type\/} or to be {\em coarsely
quasi-isometric\/}. This definition involves two constants that will
be called {\em coarse distortions\/}: one for the nets and another one
for the bi-Lipschitz equivalence. A collection of coarse
quasi-isometries is said to be {\em uniform\/} when the same coarse
distortions are valid for all of them.

Recall that the {\em Hausdorff distance\/} between subspaces $X,Y$ of
some metric space with metric $d$, is defined as
$$
d_H(X,Y)=\max\left\{\sup_{x\in X}d(x,Y),\sup_{y\in Y}d(y,X)\right\}\;.
$$
Now let $M,M'$ be arbitrary metric spaces with metrics $d,d'$.  The
{\em Gromov-Hausdorff distance\/} (also called {\em abstract Hausdorff
  distance\/}) between $M,M'$, denoted by $d_{GH}(M,M')$, is the
infimum of the Hausdorff distances $d_H(M,M')$ over all metrics on
$M\sqcup M'$ that restrict to $d,d'$ on $M,M'$. Note that
$d_{GH}(M,M')$ may be equal to $\infty$. If $d_{GH}(M,M')<\infty$,
then $M,M'$ are called {\em Hausdorff equivalent\/}. On the other
hand, the metric spaces $M,M'$ are called {\em Lipschitz equivalent\/}
when there exists a bi-Lipschitz bijection $M\to M'$. Then $M,M'$ are
coarsely quasi-isometric if and only if there are some metric spaces
$N,N'$ such that the pairs $M,N$ and $M',N'$ are Hausdorff equivalent,
and $N,N'$ are Lipschitz equivalent.

There is also a categorical description of coarse quasi-isometries.
Two maps $f,g:M\to M'$ are called {\em parallel\/} \cite{Gromov93}, or
{\em bornotopic\/} \cite{Roe93} or {\em uniformly close\/}
\cite{Block-Weinberger} if there is some $R>0$ such that
$d'(f(x),g(x))<R$ for all $x\in M$.  A map $f:M\to M'$ is said to be
{\em large scale Lipschitz\/} \cite{Gromov93} if there are constants
$\lambda,c>0$ such that
$$
d'(f(x),f(y))\le\lambda\,d(x,y)+c
$$
for all $x,y\in M$; note that $f$ need not be continuous. Then
coarse quasi-isometries can be considered as isomorphisms in the
category of metric spaces and parallel classes of large scale
Lipschitz maps.

The above description of coarse quasi-isometry is similar to the
definition of another type of ``coarse'' equivalence. A map $f:M\to
M'$ is called {\em effectively proper\/} \cite{Block-Weinberger} if
for all $r>0$ there is some $s>0$ so that
$$
d'(f(x),f(y))<r\Longrightarrow d(x,y)<s
$$
for all $x,y\in M$. The map $f$ is called {\em uniformly
  bornologous\/} \cite{Roe93} or {\em \upn{(}coarsely\upn{)}
  Lipschitz\/} \cite{Block-Weinberger} if for all $r>0$ there is some
$s>0$ so that
$$
d(x,y)<r\Longrightarrow d'(f(x),f(y))<s
$$
for all $x,y\in M$. Then two metric spaces are called {\em
  uniformly close\/} \cite{Block-Weinberger} if there is an
isomorphism between them in the category of metric spaces and
uniformly close classes of effectively proper coarsely Lipschitz maps.
Note that every large scale Lipschitz map is coarsely Lipschitz, and
it is also effectively proper if it has a large scale Lipschitz
inverse up to the uniform closeness of maps. Therefore coarsely
quasi-isometric metric spaces are uniformly close.

Uniform closeness of metric spaces is a slight modification of the
concept of {\em bornotopy equivalence\/} introduced in \cite{Roe93},
which is an isomorphism in the category of proper metric spaces and
bornotopy classes of effectively proper uniformly bornologous Borel
maps. Here, a metric space is called {\em proper\/} when its closed
bounded subsets are compact. Thus bornotopy equivalence is the same as
uniform closeness for all spaces that will be considered in this
paper.

\section{Coarse quasi-isometry type of orbits}

Let $\HH$ be a pseudogroup of local transformations of a space $Z$,
and $E$ a symmetric set of generators of $\HH$. For each $h\in\HH$ and
each $x\in\dom h$, let $|h|_{E,x}$ be defined as follows. If $h$ is
the identity around $x$, set $|h|_{E,x}=0$. Otherwise, $|h|_{E,x}$ is
the minimal positive integer $k$ such that $h=h_k\circ\dots\circ h_1$
around $x$ for some $h_1,\dots,h_k\in E$. Let $R\subset Z\times Z$
denote the equivalence relation induced by $\HH$ (whose equivalence
classes are the orbits). Then, for $(x,y)\in R$, let
$$ 
d_E(x,y)=\min\{|h|_{E,x}\ |\ h\in\HH,\ x\in\dom
h,\ h(x)=y\}\;.  
$$ 
In this way, $E$ induces a map $d_E:R\to\N$ whose
restriction to each orbit is a metric.  This is a well known
construction of a metric on the orbits; especially, for group actions.

Unlike the case of group actions, for a pseudogroup $\HH$ of local
transformations of a space $Z$ with a symmetric set $E$ of generators,
the coarse quasi-isometry type of the induced metric $d_E$ on the
orbits may depend on the choice of $E$, even if $E$ is finite. This is
due to the fact that not only composition of maps is used to generate a
group action, but restriction and combination of maps are also used to
generate $\HH$. Moreover, an equivalence of pseudogroups may not
preserve the coarse quasi-isometry type of the orbits for any choice
of generators, as can be seen in the following simple example.

\begin{example}
Let $\HH$ be the pseudogroup on $\R$
generated by the action of $\Z$ by translations, and let $\GG$ be the
restriction of $\HH$ to some open interval $U\subset\R$. If $U$ is of
length $>1$, then it meets every $\HH$-orbit, and thus $\HH$ is
equivalent to $\GG$. But the $\HH$-orbits are infinite, while each
$\GG$-orbit is finite if $U$ is of bounded length. So, for the
metrics induced by any choice of symmetric families of generators of
$\HH,\GG$, the $\HH$-orbits have infinite diameter and the
$\GG$-orbits finite diameter, and thus cannot be coarsely
quasi-isometric.
\end{example}

In the measure theoretic setting, this problem is solved by
considering {\em Kakutani equivalences\/} \cite{HK1}, which are kind
of measure theoretic counterparts of \'etale equivalences with the
additional requirement that the coarse quasi-isometry type of the
orbits is preserved. In the present topological context, the above
problem will be addressed without adding more conditions to \'etale
equivalences. Instead, appropriate representatives of pseudogroups and
sets of generators will be chosen to determine a coarse quasi-isometry
type on the orbits. The choice of appropriate pseudogroup
representatives is easy, while the choice of appropriate generators is
rather delicate.

Let $\HH$ be a pseudogroup of local transformations of a locally
compact space $Z$ with compact orbit space. By Lemma~\ref{l:compact
  space of orbits}, there is a relatively compact open subset $U$ of
$Z$ that meets all $\HH$ orbits. If $\HH$ is indeed compactly
generated, the restriction $\GG$ of $\HH$ to $U$ is a representative
of $\HH$ whose orbits will be shown to have a canonical coarse
quasi-isometry type, which is determined by any symmetric set $E$ of
generators of $\GG$ satisfying certain conditions. The first condition
on $E$ is that it must be a system of compact generation of $\HH$ on
$U$.  But this is not enough because there are systems of compact
generation on the same open set inducing different coarse
quasi-isometry types in the same orbit; such an example will be given
in Section~\ref{sec:example}. So a second new condition is introduced
as follows.

\begin{defn}\label{d:recurrent finite symmetric family of generators}
A finite symmetric family $E$ of generators of a pseudogroup $\HH$ of
local transformations of a locally compact space $Z$ is said to be
{\em recurrent\/} if there exists a relatively compact open subset
$U\subset Z$ and some $R>0$ such that any $d_E$-ball of radius $R$ in
any $\HH$-orbit meets $U$; {\it i.e.}, for any $x\in Z$ there exists
$h\in\HH$ with $x\in\dom h$, $|h|_{E,x}<R$ and $h(x)\in U$.
\end{defn}

The role played by $U$ in Definition~\ref{d:recurrent finite symmetric
family of generators} can actually be played by any relatively compact
open subset that meets all orbits, as shown by the following result.

\begin{lemma}\label{l:recurrent finite symmetric family of generators}
Let $\HH$ be a pseudogroup of local transformations of a locally
compact space $Z$, and let $E$ be a recurrent finite symmetric family
of generators of $\HH$. If $V\subset Z$ is an open set that meets
every orbit, then there exists $S>0$ such that any $d_E$-ball of
radius $S$ in any $\HH$-orbit meets $V$.
\end{lemma}

\begin{proof}
By definition, there exist a relatively compact open subset $U\subset
Z$ and a positive number $R$ such that any $d_E$-ball of radius $R$ in
any $\HH$-orbit meets $U$.  Since $V$ also meets every orbit, there
exists a finite family $F\subset\HH$ such that:
\begin{itemize}
\item the sources of elements of $F$ cover the compact closure
$\overline{U}$;
\item the targets of elements of $F$ are contained in $V$; and
\item each element of $F$ is a composite of elements of $E$.
\end{itemize}
Let $r>0$ be an integer so that every $g\in F$ can be written as a
composition of at most $r$ elements of $E$.

Fix any $x\in Z$. On the one hand, there is some $h\in\HH$ with
$x\in\dom h$, $|h|_{E,x}<R$ and $h(x)\in U$. On the other hand, there
is some $g\in F$ whose domain contains $h(x)$. So $x\in\dom(gh)$,
$gh(x)\in V$, and $$ |gh|_{E,x}\leq|g|_{E,h(x)}+|h|_{E,x}\leq r+R\;.
$$ Thus the result follows with $S=r+R$.
\end{proof}

Let $\HH$ be a compactly generated pseudogroup of local
transformations of a locally compact space $Z$. A system of compact
generation of $\HH$ on a relatively compact open subset $U\subset Z$ that
meets every orbit is called {\em recurrent\/} if it is recurrent when
considered as finite symmetric set of generators of the restriction of
$\HH$ to $U$. An example of a non-recurrent system of compact
generation will be given in Section~\ref{sec:example}. On the other
hand, the existence of recurrent systems of compact generation will be
a consequence of the following result.

\begin{lemma}\label{l:recurrent systems of compact generation}
With the above notation, let $E$ be a system of compact generation of
$\HH$ on $U$. For each $g\in E$, fix an extension $\bar g\in\HH$ of $g$
with $\overline{\dom g}\subset\dom\bar g$. Suppose that every
$x\in\overline{U}$ has an open neighborhood $V_x$ in $Z$ such that 
$$
V_x\subset\dom\bar g_x\;,\quad V_x\cap U\subset\dom g_x\;,\quad
\bar g_x(V_x)\subset U
$$
for some $g_x\in E$. Then $E$ is recurrent.
\end{lemma}

\begin{proof}
For each $x\in\overline{U}$, let $W_x=\bar g_x(V_x)$; its closure can
be assumed to be contained in $U$. Compactness of $\overline{U}$
implies that $\overline{U}\subset V_{x_1}\cup\dots\cup V_{x_n}$, for
some finite set of points $x_1,\dots,x_n\in\overline{U}$. Let
$V_k=V_{x_k}$, $W_k=W_{x_k}$ and $g_k=g_{x_k}$ for $k=1,\dots,n$, so
$W=W_1\cup\dots\cup W_n$ is a relatively compact open set in $U$.
Moreover, each $y\in U$ belongs to some $V_k$; thus $y\in V_k\cap
U\subset\dom g_k$ and $g_k(y)\in W$, yielding
$d_E(y,W\cap\HH(y))\leq1$.
\end{proof}

\begin{cor}\label{c:recurrent systems of compact generation}
Let $\HH$ be a compactly generated pseudogroup of local
transformations of a locally compact space $Z$, and let $U$ be a
relatively compact open subset of $Z$ that meets all
$\HH$-orbits. Then there exists a recurrent system $E$ of compact
generation of $\HH$ on $U$ satisfying the following property. The
extension $\bar g\in\HH$ of each $g\in E$ with $\overline{\dom
g}\subset\dom\bar g$ can be chosen so that $\overline{E}=\{\bar g\ |\
g\in E\}$ is also a recurrent system of compact generation on some
relatively compact open subset $U'\subset Z$ with $\overline{U}\subset
U'$.
\end{cor}

\begin{proof}
Since $U$ meets every $\HH$-orbit and $\overline{U}$ is compact, there
exists a finite family $F\subset\HH$ satisfying the following properties:
\begin{itemize}

\item each $f\in F$ is the restriction of some $\bar f\in\HH$ whose
domain is relatively compact and contains $\overline{\dom f}$;

\item  each $\bar f$ is the restriction of some $\tilde
f\in\HH$ with $\overline{\dom\bar f}\subset\dom\tilde f$;

\item the sources of elements of $F$ cover $\overline{U}$; and

\item $\im\bar f\subset U$ for every $f\in F$.

\end{itemize}
For each $f\in F$, let $f'$ denote its restriction
$$U\cap\dom f\to f(U\cap\dom f)\; ,$$ let $F'=\{f'\ |\ f\in F\}$, and
set $F^{\prime-1}=\{f^{\prime-1}\ |\ f'\in F'\}$.

There exists a system $G$ of compact generation of $\HH$ on $U$, and
$E=G\cup F'\cup F^{\prime-1}$ is also a system of compact generation
of $\HH$ on $U$. Moreover, $E$ satisfies the condition of
Lemma~\ref{l:recurrent systems of compact generation} because
$$
\overline{U}\subset\bigcup_{f\in F}\dom f\;,
$$ 
and $\im\bar f\subset U$
for every $f\in F$. Thus $E$ is recurrent.

Now let
$$
U'=\bigcup_{f\in F}\dom\bar f\;,
$$
which is a relatively compact open subset of $Z$ containing
$\overline{U}$.  Let
$\overline{F}=\left\{\bar f\ |\ f\in F\right\}$ and
$\overline{F}^{-1}=\left\{\bar f^{-1}\ |\ f\in F\right\}$, which are
subsets of the restriction $\GG'$ of $\HH$ to $U'$. The extensions
$\bar g$ of the maps $g\in G$ satisfying $\overline{\dom
g}\subset\dom\bar g$ can obviously be chosen so that:
\begin{itemize}

\item each $\bar g$ has source and range in $U'$;

\item the set $\overline{G}=\{\bar g\ |\ g\in G\}$ is symmetric; and

\item each $\bar g$ is the restriction of some $\tilde g\in\HH$ with
$\overline{\dom\bar g}\subset\dom\tilde g$.

\end{itemize}
Then $\overline{E}=\overline{G}\cup\overline{F}\cup\overline{F}^{-1}$
is a finite symmetric subset of $\GG'$ which
generates $\GG'$ because $G$ generates $\GG$ and $\im\bar f\subset U$
for all $f\in F$. The above properties guarantee
that $\overline{E}$ is a system of compact generation of $\HH$ on
$U'$. Finally, $\overline{E}$ is recurrent by Lemma~\ref{l:recurrent
systems of compact generation} since $\im\bar f\subset U$ for all
$f\in F$.
\end{proof}

The following is the promised result that shows the invariance of the
coarse quasi-isometry type of the orbits by equivalences when
appropriate representatives of pseudogroups and generators are chosen.

\begin{theorem}\label{t:coarsely quasi-isometric orbits}
Let $\HH,\HH'$ be compactly generated pseudogroups of local
transformations of locally compact spaces $Z,Z'$, and let $U,U'$ be
relatively compact open subsets of $Z,Z'$ that meet all orbits of
$\HH,\HH'$, respectively. Let $\GG,\GG'$ denote the restrictions of
$\HH,\HH'$ to $U,U'$, and let $E,E'$ be recurrent symmetric systems 
of compact
generation of $\HH,\HH'$ on $U,U'$, respectively. Suppose that there
exists an equivalence $\HH\to\HH'$, and consider the induced
equivalence $\GG\to\GG'$ and homeomorphism $U/\GG\to U'/\GG'$. Then the
$\GG$-orbits with $d_E$ are uniformly coarsely quasi-isometric to the
corresponding $\GG'$-orbits with $d_{E'}$.
\end{theorem}

In other words, Theorem~\ref{t:coarsely quasi-isometric orbits}
asserts that, for pseudogroups $\HH$ of local transformations of
locally compact spaces $Z$ with given sets of generators $E$,
the coarse quasi-isometry type of the orbits is uniformly invariant by
equivalences when the following conditions are satisfied. First, $\HH$
must be the restriction of a pseudogroup $\HH'$ acting on a larger
locally compact space where $Z$ is open, relatively compact and meets
all $\HH'$-orbits.  Second, $E$ must be a recurrent system of compact
generation of $\HH'$ on $Z$.

In order to prove Theorem~\ref{t:coarsely quasi-isometric orbits}, the
following preliminary results will be required.

\begin{lemma}\label{l:coarsely quasi-isometric orbits}
Let $\HH$ be a compactly generated pseudogroup of local
transformations of a locally compact space $Z$, let $U,U'$ be
relatively compact open subsets of $Z$ such that $U\cap U'$ meets all
$\HH$-orbits, and let $E,E'$ be recurrent systems of compact generation
of $\HH$ over $U,U'$, respectively. Then, for any open set $V$ that
meets all $\HH$-orbits and with $\overline{V}\subset U\cap U'$, there
exists some $C>0$ such that
$$
\frac{1}{C}\,d_{E'}(x,y)\le d_E(x,y)\le C\,d_{E'}(x,y)
$$
for all $x,y\in V$ lying in the same $\HH$-orbit.
\end{lemma}

\begin{proof}
Let $\GG$ denote the restriction of $\HH$ to $U$.
By Lemma~\ref{l:recurrent finite symmetric family of generators}, there
exists some $R>0$ such that any $d_E$-ball of radius $R$ in any
$\GG$-orbit meets $V$. Let $\Phi\subset\GG$ denote the finite set of
restrictions of the form
$$ g:V\cap\dom g\to g(V\cap\dom g)\;, $$ where $g$ runs over the
composites of at most $R$ elements of $E$, wherever defined.  It is
noted that the images of elements of $\Phi$ cover $U$. Moreover, it
may be assumed that $R\ge2$, and thus that the identity map of
$V\cap\dom g$ belongs to $\Phi$ for all $g\in E$ with $V\cap\dom
g\neq\emptyset$.

Let $F$ denote the finite set of composites $\psi^{-1}\circ
g\circ\phi$, wherever defined, where $\phi,\psi\in\Phi$ and $g\in
E$. Observe that each $f\in F$ is the restriction of some $\bar
f\in\GG$ with $\overline{\dom f}\subset\dom\bar f$.  Furthermore, for
each $x\in\dom\bar f$, it holds that $\left|\bar
f\right|_{E',y}\leq\left|\bar f\right|_{E',x}$, for all $y$ in some
neighborhood of $x$ in $\dom\bar f$. Hence, since $F$ is finite and
the domain of each $f\in F$ is relatively compact in $U'$, there
exists an integer $S>0$ such that $|f|_{E',x}\leq S$ for all $f\in F$
and all $x\in\dom f$.

Let $x,y\in V$ be points in the same $\HH$-orbit.  If $x=y$, then
$d_E(x,y)=d_{E'}(x,y)=0$, and the statement holds trivially with any
$C>0$. If
$x\neq y$, then
$d_E(x,y)=k\ge1$. Let $g\in\GG$ be such that $x\in\dom g$, $g(x)=y$
and $|g|_{E,x}=k$. This element $g$ may be assumed to be of the form
$g=g_k\circ\dots\circ g_1$ for some $g_1,\dots,g_k\in E$. Then, for
$i=1,\dots,k-1$, there exists $\phi_i\in\Phi$ whose image contains
$x_i=g_i\circ\dots\circ g_1(x)$ and such that $z_i=\phi_i^{-1}(x_i)\in
V$. Such $g$ can be written as $g=f_k\circ\dots\circ f_1$ around $x$,
where $f_1,\dots,f_k$ are the elements of $F$ given by $$
f_1=\phi_1^{-1}\circ g_1\;,\quad f_k=g_k\circ\phi_{k-1}\;,\quad
f_i=\phi_i^{-1}\circ g_i\circ\phi_{i-1}\;, $$ for
$i=2,\dots,k-1$. Observe that $x\in\dom f_1$ and $z_i\in\dom f_{i+1}$
for all $i=1,\dots,k-1$. Therefore, $$
|g|_{E',x}\leq|f_k|_{E',z_{k-1}}+\dots+|f_2|_{E',z_1}+|f_1|_{E',x}\leq
kS\;, $$ yielding $$ d_{E'}(x,y)\leq S\,d_E(x,y)\;.  $$ Similarly, $$
d_E(x,y)\leq S'\,d_{E'}(x,y) $$ for some integer $S'>0$, and
result follows with
$C=\max\{S,S'\}$.
\end{proof}

\begin{cor}\label{c:coarsely quasi-isometric orbits}
Let $\HH$ be a compactly generated pseudogroup of local
transformations of a locally compact space $Z$ and let $U,U'$ be
relatively compact open subsets of $Z$ such that $U\cap U'$ meets all
$\HH$-orbits. Let $\GG,\GG'$ denote the restrictions of $\HH$ to
$U,U'$, and let $E,E'$ be recurrent systems of compact generation of
$\HH$ over $U,U'$, respectively. Then the $\GG$-orbits with $d_E$ are
uniformly coarsely quasi-isometric to the corresponding $\GG'$-orbits
with $d_{E'}$.
\end{cor}

\begin{proof}
There exists an open set $V$ meeting every $\HH$-orbit and  such that
$\overline{V}\subset U\cap U'$. By Lemma~\ref{l:recurrent finite
symmetric family of generators}, there also exist $R,R'>0$ such that any
$d_E$-ball of radius $R$ in any $\GG$-orbit meets $V$, and any
$d_{E'}$--ball of radius $R'$ in any $\GG'$-orbit  meets
$V$. That is, for every $\GG$-orbit $\OO$ and for every $\GG'$-orbit
$\OO'$, the intersection $\OO\cap V$ is an $R$-net in $\OO$ and
$\OO'\cap V$ is an $R'$-net in $\OO'$. So the result follows from
Lemma~\ref{l:coarsely quasi-isometric orbits}.
\end{proof}

\begin{proof}[Proof of Theorem~\ref{t:coarsely quasi-isometric orbits}]
Let $\Phi:\HH\to\HH'$ be an equivalence, and let $\HH''$ be the
pseudogroup of local transformations of $Z''=Z\sqcup Z'$ generated by
$\HH\cup\HH'\cup\Phi$, considered as sets of local transformations
of $Z''$ in the obvious way. Then $U$, $U'$ and $U''=U\sqcup U'$ are
relatively compact open subsets of $Z''$ that meet all
$\HH''$-orbits. The restrictions of $\HH''$ to $U$ and $U'$ are $\GG$
and $\GG'$, thus $E$ and $E'$ are recurrent systems of compact generation
of $\HH''$ on $U$ and $U'$, respectively.  Let $\GG''$ denote the
restriction of $\HH''$ to $U''$, and select a recurrent system $E''$
of compact generation of $\HH''$ on $U''$. By Corollary~\ref{c:coarsely
quasi-isometric orbits}, the $\GG''$-orbits with $d_{E''}$ are
uniformly coarsely quasi-isometric to the corresponding $\GG$-orbits
with $d_E$, and also to the corresponding $\GG'$-orbits with
$d_{E'}$.
\end{proof}

Let $\HH$ be a compactly generated pseudogroup of local
transformations of a locally compact space $Z$, $U$ a relatively
compact open subset of $Z$ that meets all $\HH$-orbits, and $\GG$ the
restriction of $\HH$ to $U$.  By considering the identity map on the
$\GG$-orbits and inclusions of systems of compact generation on $U$,
we get an inductive system of metric spaces. Note also that distances
between points in the same $\GG$-orbit do not increase by
considering larger systems of compact generation. By
Theorem~\ref{t:coarsely quasi-isometric orbits}, the corresponding
``inductive system of coarse quasi-isometry types'' has a limit, which
is uniformly reached just when a recurrent system of compact
generation is considered. The following consequence of
Lemma~\ref{l:coarsely quasi-isometric orbits} shows that it happens so
with the corresponding ``inductive system of Lipschitz types.''

\begin{cor}\label{c:recurrent}
With the above notation, let $E$ be a recurrent symmetric system of compact
generation of $\HH$ on $U$. Any other symmetric system $E'$ of compact
generation of $\HH$ on $U$ is recurrent
if and only if there exists some $C>0$ such that
\begin{equation}\label{e:C}
\frac{1}{C}\,d_{E'}(x,y)\le d_E(x,y)\le C\,d_{E'}(x,y)
\end{equation}
for all $x,y\in U$ lying in the same $\GG$-orbit.
\end{cor}

\begin{proof}
Fix any open set $V$ that meets all
$\GG$-orbits and with $\overline{V}\subset U$.

Suppose that $E'$ is recurrent. Then, by Lemma~\ref{l:recurrent finite
  symmetric family of generators}, there is some $R>0$ such that
$\OO\cap V$ is an $R$-net in $(\OO,d_{E'})$ for every $\GG$-orbit
$\OO$. To show that~\eqref{e:C} holds for some $C>0$, assume first
that $E'\subset E$. Hence the first inequality of~\eqref{e:C} holds
for any $C\ge1$. Take arbitrary points $x,y\in U$ lying in the same
$\GG$-orbit. We can assume that $x\neq y$, otherwise~\eqref{e:C} holds
trivially for any $C>0$.  There are points $x',y'\in V$ with
$d_{E'}(x,x'),d_{E'}(y,y')\le R$. So $d_E(x,x'),d_E(y,y')\le R$ as
well because $E'\subset E$. By Lemma~\ref{l:coarsely quasi-isometric
  orbits}, there is some $C'>0$, independent of $x',y'$, such that
$$
\frac{1}{C'}\,d_{E'}(x',y')\le d_E(x',y')\le C'\,d_{E'}(x',y')\;.
$$
Therefore
\begin{align*}
d_E(x,y)&\le d_E(x,x')+d_E(x',y')+d_E(y,y')\\
&\le d_E(x',y')+2R\\
&\le C'\,d_{E'}(x',y')+2R\\
&\le C'\,(d_{E'}(x',x)+d_{E'}(x,y)+d_{E'}(y,y'))+2R\\
&\le C'\,(d_{E'}(x,y)+2R)+2R\\
&\le (C'+2C'R+2R)\,d_{E'}(x,y)\\
\end{align*}
since $d_{E'}(x,y)\ge1$, yielding the second inequality of~\eqref{e:C}
with $C=C'+2C'R+2R$.

When $E'\not\subset E$, the union $E''=E\cup E'$ is a recurrent
symmetric system of compact generation of $\HH$ on $U$. We have shown
that there are some $C_1,C_2>0$ such that
\begin{gather*}
\frac{1}{C_1}\,d_{E''}(x,y)\le d_E(x,y)\le C_1\,d_{E''}(x,y)\;,\\
\frac{1}{C_2}\,d_{E''}(x,y)\le d_{E'}(x,y)\le C_2\,d_{E''}(x,y)
\end{gather*}
for all $x,y\in U$ lying in the same $\GG$-orbit, and
therefore~\eqref{e:C} holds with $C=C_1C_2$.

Now assume that~\eqref{e:C} holds for some $C>0$ and all $x,y\in U$
lying in the same $\GG$-orbit. By Lemma~\ref{l:recurrent finite
  symmetric family of generators}, there is some $R>0$ such that
$\OO\cap V$ is an $R$-net in $(\OO,d_E)$ for every $\GG$-orbit $\OO$.
Then it easily follows that $\OO\cap V$ is an $CR$-net in
$(\OO,d_{E'})$ for every $\GG$-orbit $\OO$, and thus $E'$ is recurrent
by Lemma~\ref{l:recurrent finite symmetric family of generators}.
\end{proof}

\section{F\"olner orbits}

The F\"olner condition will be used in the next section to distinguish
coarse quasi-isometry types of orbits. The property that F\"olner
orbits give rise to invariant measures will be needed also. This was
shown by S.~Goodman and J.~Plante~\cite{Goodman-Plante} for
pseudogroups acting on compact metric spaces, and is partially
improved in this section by using recurrent compact generation instead
of a compact space. For compact foliated spaces, it is well known that
F\"olner leaves induce invariant transverse probability measures. So
recurrence can be useful to show that compactly generated pseudogroups
behave like compact foliated spaces, which is in the spirit of a
famous project of A.~Haefliger \cite{Haefliger00}.

Let $M$ be a metric space with metric $d$. A {\em quasi-lattice\/}
$\Gamma$ of $M$ is a net of $M$ such that for every $r\ge0$ there is
some $K_r\ge0$ such that $\#(\Gamma\cap B(x,r))\le K_r$ for every
$x\in M$. Not every metric space has a quasi-lattice, but metric
spaces with bounded complexity in a reasonable sense do; see {\it
  e.g.} \cite{Block-Weinberger}) for examples. The metric space $M$ is
said to be of {\em coarse bounded geometry\/} if it has a
quasi-lattice.

For any $r>0$, the {\em $r$-boundary\/} of each subset $S\subset M$ is the set
$$
\partial_rS=\{x\in S\ |\ d(x,S)<r\ \text{and}\ d(x,M\setminus
S)<r\}\;.
$$
The notation $\partial^M_rS$ will be also used for $\partial_rS$.
Then $M$ is called {\em amenable\/} \cite{Block-Weinberger} if it has
a quasi-lattice $\Gamma$ and a sequence of finite subsets
$S_n\subset\Gamma$ such that
\begin{equation}\label{e:Folner}
\lim_{n\to\infty}\frac{\#\partial^\Gamma_rS_n}{\#S_n}=0
\end{equation} 
for each $r>0$. Such a sequence $S_n$ will be called a {\em F\"olner
sequence\/} in $\Gamma$. Since 
$$
\partial^\Gamma_rS\setminus S\subset\bigcup_{x\in
S\cap\partial^\Gamma_rS}(\Gamma\cap B(x,r))
$$
for every $S\subset\Gamma$, it follows that
\begin{equation}\label{e:partial r S setminus S}
\#(\partial_rS\setminus S)\le K_r\cdot\#(S\cap\partial_rS)
\end{equation}
if $\#(\Gamma\cap B(x,r))\le K_r$ for any $x\in\Gamma$. Therefore the
amenability condition~\eqref{e:Folner} is equivalent to
$$
\lim_{n\to\infty}\frac{\#(S_n\cap\partial^\Gamma_rS_n)}{\#S_n}=0
$$
for each $r>0$.

\begin{theorem}[Block-Weinberger \cite{Block-Weinberger}]\label{t:Folner}
  Let $M,M'$ be uniformly close metric spaces of coarse bounded
  geometry. Then $M$ is amenable if and only if so is $M'$.
\end{theorem}
 
This result was proved in \cite{Block-Weinberger} in the following
way. First, the uniformly finite homology $H^{\text{\rm
    uf}}_\bullet(M)$ is introduced for any metric space $M$.  Second,
it is shown that, if two metric spaces $M,M'$ are uniformly close,
then $H^{\text{\rm uf}}_\bullet(M)\cong H^{\text{\rm
    uf}}_\bullet(M')$. Finally, it is shown that a metric space $M$ of
coarse bounded geometry is amenable if and only if $H^{\text{\rm
    uf}}_0(M)\neq0$, and Theorem~\ref{t:Folner} follows. In
particular, amenability is a coarse quasi-isometry invariant for
metric spaces of coarse bounded geometry, which can be also proved
directly without too much difficulty.

The following lemma will be useful in the in the proof of the
main result of this section.

\begin{lemma}\label{l:A cap S}
  Let $\Gamma$ be a quasi-lattice in some metric space, and let $A$ be
  a $C$-net in $\Gamma$ for some $C>1$.  Fix any $K>0$ such that every
  ball of radius $C$ in $\Gamma$ has at most $K$ points.  Then
$$
\#S\le\#\left(S\cap\partial^\Gamma_CS\right)+K\cdot\#(A\cap S)
$$
for any $S\subset\Gamma$.
\end{lemma}

\begin{proof}
For every $x\in S$, there is some $a\in A$ so that $d(x,a)<C$, and thus
$x\in S\cap\partial^\Gamma_CS$ if $a\not\in S$. Therefore
$$
S\subset\left(S\cap\partial^\Gamma_CS\right)\cup
\bigcup_{a\in A\cap S}(\Gamma\cap B(a,C))\;,
$$
yielding the inequality of the statement.
\end{proof}

We are interested in the case of a metric space $M$ whose points are
the vertices of some connected graph, and where the distance between
two points is the minimum number of contiguous edges needed to join
them. In this case, besides the amenability condition of
\cite{Block-Weinberger}, $M$ may be also F\"olner in the usual graph
sense, which is defined as follows. The {\em boundary\/} $\partial S$
of any $S\subset M$ is the set of points $x\in S$ such that there is
some edge joining $x$ with some point in $M\setminus S$;  {\it i.e.},
$\partial S=S\cap\partial_2S$ with the notation
of~\cite{Block-Weinberger}. Then $M$ is {\em F\"olner\/} (as a graph)
when there is a sequence of finite subsets $S_n\subset M$ such that
$$
\lim_{n\to\infty}\frac{\#\partial S_n}{\# S_n}=0\;.
$$
Note that $M$ is a quasi-lattice in itself just when there is a
uniform upper bound $K$ on the number of edges that meet at every
vertex; indeed, $\#B(x,r)\le K^r$ for all $r\ge0$ when there is such a
$K$. In this case, since
$$
\partial_rS\subset\bigcup_{x\in\partial S}B(x,r)\;, 
$$
it follows that
$$
\#\partial_rS\le K^r\cdot\#\partial S
$$
for any $r>0$. Hence, when there is such a uniform upper bound $K$,
$M$ is amenable (as metric space) if and only if it is
F\"olner (in the graph sense).

Consider again a pseudogroup $\HH$ of local
transformations of a space $Z$ with the metric $d_E$ on the orbits
induced by a finite symmetric set $E$ of generators of $\HH$. Then we
get a graph by introducing an edge between two points $x,y\in Z$
whenever there is some $g\in E$ with $g(x)=y$. Thus $\#E$ is an upper
bound for the number of edges that meet at every vertex. Observe that
each orbit of $\HH$ is given by the vertices of a connected component
of this graph, and $d_E$ is the metric induced by this graph on its
connected components.  The following notation and terminology will be
used in this setting:
\begin{itemize}

\item Let $\partial^ES$ denote the boundary of any finite subset $S$
of an orbit with respect to the graph structure induced by $E$;

\item for $r>0$, let $\partial^E_rS$ denote the $r$-boundary of any
  finite subset $S$ of an orbit with respect to the metric $d_E$;
  
\item a F\"olner sequence of an orbit with the metric $d_E$ (or with
  the graph structure induced by $E$) will be called an {\em
    $E$-F\"olner sequence\/}; and
  
\item an orbit with an $E$-F\"olner sequence will be called {\em
    $E$-F\"olner\/} or {\em $E$-amenable\/}.

\end{itemize}

\begin{theorem}\label{t:Folner and recurrent compact generation}
  Let $\HH$ be a compactly generated pseudogroup of local
  transformations of a locally compact metric space $Z$, let $U$ be a
  relatively compact open subset of $Z$ that meets all $\HH$-orbits,
  and let $\GG$ be the restriction of $\HH$ to $U$. Consider the
  metric on the $\GG$-orbits induced by a recurrent symmetric system
  $E$ of compact generation of $\HH$ on $U$. If some $\GG$-orbit is
  $E$-F\"olner, then there is a non-trivial non-negative
  $\GG$-invariant Borel measure on $U$ of finite mass.
\end{theorem}

\begin{proof}
Let $S_n$ be an $E$-F\"olner sequence of some orbit $\OO$ of $\GG$.
As in \cite{Goodman-Plante}, a measure $\mu$ is
constructed on $U$ as a limit of averaging measures on the finite sets $S_n$.
Let $C_0(U)$ be the Banach space of continuous functions $f:U\to\R$ that
vanish at infinity, endowed with the supremum norm $\|\ \|$ given by
$$
\|f\|=\sup_{x\in U}|f(x)|
$$
 For each $n\in\N$, let $\mu_n:C_0(U)\to\R$ be defined by
$$
\mu_n(f)=\frac{1}{\#S_n}\sum_{x\in S_n}f(x)
$$
for $f\in C_0(U)$. Each $\mu_n$ is obviously linear and continuous;
{\it i.e.}, it is an element of the (algebraic-topological) dual space
$C_0(U)'$. Moreover it is easy to check that
$$
|\mu_n(f)|\le\|f\|
$$
for all $n\in\N$ and $f\in C_0(U)$. Therefore, by the Banach-Alaoglu
theorem, the set $\{\mu_n\ |\ n\in\N\}$ is relatively compact in
$C_0(U)'$ with the weak$^*$ topology; {\it i.e.}, the topology of
pointwise convergence. Then, by passing to a subsequence if necessary,
we can assume that the sequence $\mu_n$ converges pointwise to some
$\mu$ in $C_0(U)'$, which can be considered as a Borel measure of
finite mass on $U$ by the Riesz representation theorem. This $\mu$ is
non-negative since all the $\mu_n$ are probability measures. The
$\GG$-invariance of $\mu$ follows from the $E$-F\"olner condition of
the sequence $S_n$ since, as shown in \cite{Goodman-Plante},
$$
|\mu(f\circ g)-\mu(f)|\le2\,\|f\|\,\lim_n\frac{\#\partial^ES_n}{\#S_n}
$$
for all $g\in E$ and $f\in C_0(U)$ with $\supp f\subset\im g$.

Finally, we show that $\mu$ is not trivial. Take any open set $U'$
that meets all $\GG$-orbits and with $\overline{U'}\subset U$, and
consider any non-negative function $f\in C_0(U)$ with $f(x)=1$ for all
$x\in U'$. By Lemma~\ref{l:recurrent finite symmetric family of
  generators}, there is some $C>0$ such that $\OO\cap U'$ is a $C$-net
in $\OO$ with $d_E$. Then
$$
\mu(f)\ge\lim_n\frac{\#(S_n\cap U')}{\#S_n}
\ge(\#E)^{-C}\,\lim_n\frac{\#S_n-\#\left(S_n\cap\partial^E_CS_n\right)}{\#S_n}
=(\#E)^{-C}>0
$$
by Lemma~\ref{l:A cap S} and since the
sequence $S_n$ is $E$-F\"olner.
\end{proof}

\begin{rems}
  (i) In the proof of Theorem~\ref{t:Folner and recurrent compact
    generation}, the measure $\mu$ could be trivial if $E$ were not
  recurrent. This is different from the arguments of
  \cite{Goodman-Plante} because $U$ is not compact. For instance, for
  the pseudogroup on $\R$ generated by the translation $g(x)=x+1$, the
  sets $\{n,n+1,\dots,2n\}$ ($n\in\N$) form a $\{g,g^{-1}\}$-F\"olner
  sequence in an orbit, and the limit of corresponding averaging
  measures is trivial.  
\newline 
(ii) The statement of
  Theorem~\ref{t:Folner and recurrent compact generation} could be
  more general. As in \cite{Goodman-Plante}, the definitions and
  arguments could be modified to remove the condition that all sets of
  the F\"olner sequence lie in the same orbit: it would be enough to
  have what is called an {\em averaging sequence\/} in
  \cite{Goodman-Plante}. But our result is simpler and general enough
  for our purposes in the next section.
\end{rems}

\section{An example of non-recurrent compact generation}\label{sec:example}

In this section, we give an example showing that there exist
non-recurrent systems of compact generation, and that the coarse
quasi-isometry type of the orbits may depend on the system of compact
generation if recurrence is not considered.

Fix real numbers 
$$
a<a'<a''<b''<b'<b\;,
$$ 
and choose homeomorphisms $\phi:\R\to(a,b)$ and
$\tilde g_1:\R\to\R$ satisfying the following properties:
\begin{itemize}

\item $\phi(x)=x$ for all $x\in[a'',b'']$;

\item $\phi(x)>x$ for all $x\in(-\infty,a'')$;

\item $\phi(x)<x$ for all $x\in(b'',\infty)$;

\item $\phi(a)=a'$ and $\phi(b)=b'$;

\item $\tilde g_1(x)=x$ for all $x\in(-\infty,a]\cup[b,\infty)$;

\item $\tilde g_1(x)>x$ for all $x\in(a,b)$; and

\item $\tilde g_1(a'')<b''$.

\end{itemize}
Let $\HH$ be the pseudogroup of local transformations of $\R$
generated by $\phi$ and $\tilde g_1$. The bounded open interval
$U=(a,b)$ meets all $\HH$-orbits because $\phi(\R)=U$, and let $\GG$
denote the restriction of $\HH$ to $U$.

Now define a map $\tilde g_2:\R\to\R$ by setting
$$
\tilde g_2(x)=
\begin{cases}
\phi\circ\tilde g_1\circ\phi^{-1}(x)&\text{if $a<x<b$,}\\
x&\text{otherwise.}
\end{cases}
$$
Such a $\tilde g_2$ is a homeomorphism and satisfies the following properties:
\begin{itemize}

\item $\tilde g_2(x)=x$ for all $x\in(-\infty,a']\cup[b',\infty)$;

\item $\tilde g_2(x)>x$ for all $x\in(a',b')$; and

\item $\phi\circ\tilde g_1(x)=\tilde g_2\circ\phi(x)$ for all $x\in U$.

\end{itemize}

We now prove that $\GG$ is generated by the restrictions $\tilde
g_1,\tilde g_2:U\to U$, which will be denoted by $g_1,g_2$.  It is
enough to prove that the restriction $\phi:U\to\phi(U)$ is in the
pseudogroup $\GG'$ generated by $g_1,g_2$. Since $\tilde
g_1(a'')<b''$, the collection of sets $U_n=g_1^n(a'',b'')$, $n\in\Z$,
is an open covering of $U$, and thus it suffices to prove that each
restriction $\phi:U_n\to\phi(U_n)$ is in $\GG'$. But $\phi$ is the
identity on $(a'',b'')=g_1^{-n}(U_n)$, yielding
$$
g_1^{-n}=\phi\circ g_1^{-n}=g_2^{-n}\circ\phi
$$
on $U_n$. So $\phi=g_2^n\circ g_1^{-n}$ on $U_n$, which belongs 
to $\GG'$, as desired.

Since $\overline{\dom g_i}\subset\dom\tilde g_i$ for $i=1,2$, it follows that
$E=\{g_1,g_2,g_1^{-1},g_2^{-1}\}$ is a system of compact generation of
$\HH$ on $U$. For the open subset $V=(a',b)\subset U$, it will be shown that
\begin{equation}\label{e:recurrent}
\lim_{x\to a}d_E(x,V\cap\GG(x))\to\infty\;.
\end{equation}
Since $V$ meets every $\GG$-orbit, it follows from
Lemma~\ref{l:recurrent finite  symmetric family of generators}
and~\eqref{e:recurrent} that $E$ is not recurrent. 

To prove~\eqref{e:recurrent},
let
$$
\nu(x)=\min\{n\in\N\ |\ g_1^n(x)\in V\}
$$
for each $x\in U$. Clearly, 
$$
x\in V\Longleftrightarrow\nu(x)=0\;,\quad
x<y\Longrightarrow\nu(x)\ge\nu(y)\;,\quad\lim_{x\to a}\nu(x)=\infty\;.
$$
Take any $x\in U$ and some $h\in\GG$ with
$$
x\in\dom h\;,\quad h(x)\in V\;,\quad |h|_{E,x}=d_E(x,V\cap\GG(x))\;.
$$
For $n=|h|_{E,x}$, we have
$h=g_{i_n}^{\epsilon_n}\circ\dots\circ g_{i_1}^{\epsilon_1}$ around
$x$ for some 
$i_1,\dots,i_n\in\{1,2\}$ and
$\epsilon_1,\dots,\epsilon_n\in\{\pm1\}$. Let 
$x_k=g_{i_k}^{\epsilon_k}\circ\dots\circ g_{i_1}^{\epsilon_1}(x)$ for every
$k=1,\dots,n$. Since $x_k\not\in V$ for each $k<n$,
either $x_{k+1}\le x_k$ (yielding $\nu(x_{k+1})\ge\nu(x_k)$), or
$g_{i_{k+1}}^{\epsilon_{k+1}}=g_1$ (yielding
$\nu(x_{k+1})=\nu(x_k)-1$). Therefore  
$$
\nu(x)\le n=d_E(x,V\cap\GG(x))
$$ 
because $\nu(x_n)=0$,
and~\eqref{e:recurrent} follows.

Finally, let $F$ be a recurrent symmetric system of compact generation
of $\HH$ on $U$. We will show that no $\GG$-orbit with the metric
$d_E$ is coarsely quasi-isometric to itself with the metric $d_F$.
Suppose that this is not true for some $\GG$-orbit $\OO$. Since the
open interval $I=(a,a')$ meets every orbit and since $g_2$ is the
identity on $I$, there is some point $x_0\in\OO\cap I$ such that the
set
$$
X=\{g_1^{-m}(x_0)\ |\ m\in\N\}\subset\OO\cap I
$$
satisfies $\partial^EX=\{x_0\}$.
Hence $\OO$ is $E$-F\"olner; for instance, an $E$-F\"olner sequence for
$\OO$ is given by the sets
$$
S_n=\{g_1^{-m}(x_0)\ |\ 0\le m\le n\}\;.
$$
Then $\OO$ is also $F$-F\"olner by Theorem~\ref{t:Folner} since we
are assuming that the metrics $d_E,d_F$ on $\OO$ are coarsely
quasi-isometric. So, by Theorem~\ref{t:Folner and recurrent compact
  generation} and because $F$ is recurrent, there is a non-trivial
non-negative $\GG$-invariant Borel measure $\mu$ on $U$ of finite
mass. Fix any $t\in U$, and let $I_n=(a,g_1^n(t))$ for each $n\in\Z$.
Since $\mu$ is $\GG$-invariant and $g(I_n)=I_{n+1}$, all sets $I_n$
have the same $\mu$-measure, which is finite since $\mu$ has finite
mass. Then $\mu(I_{n+1}\setminus I_n)=0$ for all $n$, whence
$\mu(U)=0$ because $U=\bigcup_n(I_{n+1}\setminus I_n)$. This is a
contradiction because $\mu$ is non-trivial and non-negative.

\section{Quasi-local metric spaces}

The concept of equicontinuity can be defined for pseudogroups of local
transformations of uniform spaces, but we are mainly concerned with
the case of metric spaces in this paper. Nevertheless, it is enough to
consider only certain part of the local geometry of metric spaces,
which is extracted in the following definition. Moreover it is easier
to work with pseudogroups and their equivalences when all other
geometric information is removed from metric spaces.

\begin{defn}\label{d:quasi-local metric space}
Let $\{(Z_i,d_i)\}_{i\in I}$ be a family of
metric spaces such that $\{Z_i\}_{i\in I}$ is a covering of a set $Z$, each
intersection $Z_i\cap Z_j$ is open in $(Z_i,d_i)$ and $(Z_j,d_j)$, and for all
$\epsilon>0$ there is some
$\delta(\epsilon)>0$ so that the following property holds: for all
$i,j\in I$ and $z\in Z_i\cap Z_j$, there is
some open neighborhood $U_{i,j,z}$ of $z$ in $Z_i\cap Z_j$ (with
respect to the topology induced by $d_i$ and $d_j$) such that
\begin{equation}\label{e:quasi-local metric space}
d_i(x,y)<\delta(\epsilon)\Longrightarrow d_j(x,y)<\epsilon
\end{equation}
for all $\epsilon>0$ and all $x,y\in U_{i,j,z}$. Such a family will be
called a {\em cover of $Z$ by quasi-locally equal metric spaces\/}.
Two such families are called {\em quasi-locally equal\/} when their
union also is a cover of $Z$ by quasi-locally equal metric spaces.
This is an equivalence relation whose equivalence classes are called
{\em quasi-local metrics\/} on $Z$. For each quasi-local metric
${\mathfrak Q}$ on $Z$, the pair $(Z,{\mathfrak Q})$ is called a {\em
  quasi-local metric space\/}.
\end{defn}

Any quasi-local metric $\mathfrak Q$ on $Z$ induces a uniformity so
that, for any $\{(Z_i,d_i)\}_{i\in I}\in{\mathfrak Q}$, the families
$$
\UU_r=\left\{\left.\bigcup_{i\in I,\ x\in Z_i}B_i(x,r)\ \right|\ x\in Z\right\}\;,\quad
r>0\;,
$$
form a base of uniform covers, where $B_i(x,r)$ denotes the open
ball in $(Z_i,d_i)$ of center $x$ and radius $r$. The open balls of
all metric spaces $(Z_i,d_i)$ form a base of the corresponding
topology on $Z$. Any topological concept or property of $(Z,{\mathfrak
  Q})$ refers to this underlying topology.  Any quasi-local metric
space $(Z,{\mathfrak Q})$ is locally metrizable, and thus first
countable and completely regular. If $(Z,{\mathfrak Q})$ is Hausdorff
and paracompact, then it is metrizable \cite{Smirnov} and normal
\cite[Theorem~20.10]{Willard}, and every point finite open cover of
$(Z,{\mathfrak Q})$ is shrinkable \cite[Theorem~15.10]{Willard}.
Moreover $(Z,{\mathfrak Q})$ is a locally compact Polish space if and
only if it is Hausdorff, paracompact, separable and locally compact;
this is the type of quasi-local metric spaces that will be mainly
considered in this paper.

\begin{rem}
  A quasi-local metric is a ``local structure'' in the sense that it is
  determined by its ``restriction'' to the sets of any open covering.
  This property is specially useful to deal with pseudogroup
  equivalences, and is not satisfied by general uniformities. This is
  another reason to consider quasi-local metric spaces instead of
  general uniform spaces.
\end{rem}

If a quasi-local metric space $(Z,{\mathfrak Q})$ is paracompact, then
there is some $\{(Z_i,d_i)\}_{i\in I}\in\mathfrak Q$ so that the
covering $\{Z_i\}_{i\in I}$ is locally finite. In this case,
$\{(Z_i,d_i)\}_{i\in I}$ satisfies the following slightly stronger
condition.

\begin{lemma}\label{l:locally finite quasi-locally equal covers}
  Let $(Z,{\mathfrak Q})$ be a quasi-local metric space.  If
  $\{Z_i\}_{i\in I}$ is locally finite for some $\{(Z_i,d_i)\}_{i\in
    I}\in{\mathfrak Q}$, then there is some open neighborhood $U_x$ of
  each $x\in Z$ and some assignment $\epsilon\mapsto\delta(\epsilon)$
  such that~\eqref{e:quasi-local metric space} holds for all
  $\epsilon>0$, $i,j\in I$ and $y\in U_x\cap Z_i\cap Z_j$.
\end{lemma}

\begin{proof}
With the notation of Definition~\ref{d:quasi-local metric space}, the set
$$
U_z=\bigcap_{i,j\in I,\ z\in Z_i\cap Z_j}U_{i,j,z}
$$
is an open neighborhood of each $z\in Z$ and satisfies the stated
property.
\end{proof}

\begin{example}
Any metric $d$ on a set $Z$ induces a unique quasi-local metric
$\mathfrak Q$ on $Z$ so that $\{(Z,d)\}\in{\mathfrak Q}$. It will be
shown in Section~\ref{sec:local isometries} that every Hausdorff
paracompact quasi-local metric space $(Z,{\mathfrak Q})$ is indeed induced
by some metric $d$ on $Z$ (Theorem~\ref{t:isometrization of
quasi-local metrics}), but the information of $(Z,d)$ contained
in $\mathfrak Q$ is just what is relevant for our study of equicontinuous
pseudogroups.
\end{example}

\begin{example}\label{ex:quasi-local metric on R2}
For each $t>0$, consider the metric $d_t$ on $\R^2$ defined by
$$
d_t(x,y)=\sqrt{t(x_1-y_1)^2+(x_2-y_2)^2}
$$
for $x=(x_1,x_2)$ and $y=(y_1,y_2)$. Then, for $T\subset\R_+$, the
family $\{(\R^2,d_t)\ |\ t\in T\}$ is a cover of $\R^2$ by quasi-locally equal metric
spaces if and only if $T$ is relatively compact in $\R_+$. Hence there are no
maximal covers by quasi-locally equal metric spaces in general.
\end{example}

\section{Equicontinuous pseudogroups}

This section develops the concept of equicontinuity for pseudogroups,
as suggested by E.~Ghys in \cite[Appendix~E]{Mol88}. To motivate our
definitions, consider a group $G$ of homeomorphisms of a space $Z$. On
the one hand, some uniformity is needed on $Z$ for the usual
definition of equicontinuity of $G$ (see, {\it e.g.}, A.~Weil~\cite{weil}).
But, on the other hand, equicontinuity of $G$ does not imply that each
map in $G$ is uniformly continuous; {\it i.e.}, these homeomorphisms
may not preserve the uniformity of $Z$.  This gives a difficulty when
trying to generalize equicontinuity to pseudogroups in a way
compatible with pseudogroup equivalences. More precisely, let
$\HH,\HH'$ be pseudogroups on spaces $Z,Z'$, and $\Phi:\HH\to\HH'$ an
equivalence.  Suppose that $\HH$ is equicontinuous in some reasonable
way, which should use some uniformity on $Z$. Then one has to use
$\Phi$ to construct a uniformity on $Z'$ so that $\HH'$ is
equicontinuous too.  The following is a standard way to do this kind
of construction: the uniformity of $Z$ must be restricted to domains
of homeomorphisms in $\Phi$, which are used to define uniformities on
the sets of some open covering of $Z'$, and then these local
uniformities must be combined to yield a uniformity on the whole of
$Z'$. Some conditions must be satisfied to achieve such a combination.
First, we need some type of uniformity that is determined by its
restriction to the sets of any open covering, which holds for
quasi-local metrics as indicated in the remark of
Definition~\ref{d:quasi-local metric space}. Secondly, these
uniformities on open sets of $Z'$ must be compatible on the overlaps,
which means that the local transformations of $\HH$ must preserve the
uniformity of $Z$; {\it i.e.}, they must be uniformly continuous!
Therefore the type of equicontinuity needed for pseudogroups seems to
be ``equi-uniform continuity;'' {\it i.e.}, the transformations of a
pseudogroup are not only required to be ``simultaneously'' continuous
at every point, but also required to be ``simultaneously'' uniformly
continuous. Moreover, surprisingly, there are some unsolved
difficulties to show that reasonable definitions of ``equicontinuity
at every point'' and ``equi-uniform continuity'' are equivalent for
compactly generated pseudogroups. So we define equicontinuity for
pseudogroups by requiring that the ``transformations with small
domain'' are ``simultaneously'' uniformly continuous. Indeed, what may
be understood as ``transformations with small domain'' gives rise to
two versions of equicontinuity. The first one is weaker and looks more
natural, but the second one fits our needs.

\begin{defn}\label{d:weakly equicontinuous}
Let $(Z,{\mathfrak Q})$ be a quasi-local metric space.
A pseudogroup $\HH$ of local homeomorphisms of $(Z,{\mathfrak Q})$ is
called {\em weakly equicontinuous\/} if, for some
$\{(Z_i,d_i)\}_{i\in I}\in{\mathfrak Q}$ and every
$\epsilon>0$, there is some
$\delta(\epsilon)>0$ so that the following property holds: for every
$h\in\HH$, $i,j\in I$ and $z\in Z_i\cap h^{-1}(Z_j\cap\im h)$, there is
some neighborhood $U_{h,i,j,z}$ of $z$ in $Z_i\cap h^{-1}(Z_j\cap\im
h)$ such that
\begin{equation}\label{e:equicontinuous}
d_i(x,y)<\delta(\epsilon)\Longrightarrow d_j(h(x),h(y))<\epsilon
\end{equation}
for all $\epsilon>0$ and $x,y\in U_{h,i,j,z}$.
\end{defn}

A pseudogroup $\HH$ acting on a space $Z$ will be called {\em weakly
equicontinuous\/} when it is weakly equicontinuous with respect to some
quasi-local metric inducing the topology of $Z$.

\begin{rems}
  (i) Note that weak equicontinuity is a local property on
  $(Z,{\mathfrak Q})$ to a large extent; the only global aspect is the
  assignment $\epsilon\mapsto\delta(\epsilon)$, which is valid for all
  possible $h,i,j,z$.  \newline (ii) The
  condition~\eqref{e:quasi-local metric space} of
  Definition~\ref{d:quasi-local metric space} is the particular case
  of~\eqref{e:equicontinuous} for $h$ equal to the identity map on
  $Z$. So the whole structure of quasi-local metrics is needed to
  define weakly equicontinuous pseudogroups.  \newline (iii) The
  condition of weak equicontinuity given in Definition~\ref{d:weakly
    equicontinuous} can be described as certain compatibility of $\HH$
  with $\mathfrak Q$: $\HH$ is equicontinuous on $(Z,{\mathfrak Q})$
  if and only if $\mathfrak Q$ can be realized as a combination of
  $h^*({\mathfrak Q}_{|\im h})$ for $h$ running through $\HH$, where the
  restrictions, pull-backs and combinations of quasi-local metrics are
  defined in an obvious way (when appropriate conditions are
  satisfied).  \newline (iv) In Definition~\ref{d:weakly
    equicontinuous}, the assignment $\epsilon\mapsto\delta(\epsilon)$
  depends on $\{(Z_i,d_i)\}_{i\in I}\in{\mathfrak Q}$, but this
  definition is of course independent of the choice of
  $\{(Z_i,d_i)\}_{i\in I}\in{\mathfrak Q}$; {\it i.e.}, any other
  choice of $\{(Z_i,d_i)\}_{i\in I}$ satisfies the definition with
  some other assignment $\epsilon\mapsto\delta(\epsilon)$.
\end{rems}

The following result shows that weak equicontinuity is a property of
equivalence classes of pseudogroups. The definition was worded in
precisely such way for this property to hold true; in fact, this is
rather evident by the above remarks~(i) and~(iii). 

\begin{lemma}\label{l:weakly equicontinuous}
Let $\HH,\HH'$ be equivalent pseudogroups.
Then $\HH$ is weakly equicontinuous if and only if 
$\HH'$ is equicontinuous.
\end{lemma}

\begin{proof}
Let $Z,Z'$ be the acted on by $\HH,\HH'$.
Assuming  that $\HH$ is weakly equicontinuous with respect to some
quasi-local metric $\mathfrak Q$ inducing the topology of $Z$,
we will show that so is $\HH'$. Thus there is some
$\{(Z_i,d_i)\}_{i\in I}\in{\mathfrak Q}$ and some assignment
$\epsilon\mapsto\delta(\epsilon)$ such that,
for all $h\in\HH$ and $i,j\in I$, the 
condition~\eqref{e:equicontinuous} holds on some neighborhood
$U_{h,i,j,z}$ of each $z\in Z_i\cap
h^{-1}(Z_j\cap\im h)$. 

Let $\Phi:\HH\to\HH'$ be a pseudogroup equivalence. There is an open
covering $\{Z'_a\}_{a\in A}$ of $Z'$ such that, for each $a\in A$, there
is some $\phi_a\in\Phi_0$ and some $i_a\in I$ with
$Z'_a\subset\im\phi_a$ and $\dom\phi_a\subset Z_{i_a}$. Let $d'_a$
denote the restriction to $Z'_a$ of the metric on
$\im\phi_a$ that corresponds via $\phi_a$ to the restriction of 
$d_{i_a}$ to $\dom\phi_a$. For 
$$
h'\in\HH'\;,\quad a,b\in A\;,\quad z'\in Z'_a\cap
{h'}^{-1}(Z'_b\cap\im h')\;,
$$
let $U'_{h',a,b,z'}=\phi_a(U_{h,i_a,i_b,z})$, where
$$
h=\phi_b^{-1}\circ h'\circ\phi_a\in\HH\;,\quad z=\phi_a^{-1}(z')\in
Z_{i_a}\cap h^{-1}(Z_{i_b}\cap\im h)\;.
$$

Now, given any $\epsilon>0$, suppose
$d'_a(x',y')<\delta(\epsilon)$ for
$x',y'\in U'_{h',a,b,z'}$. Then the
points $x=\phi_a^{-1}(x')$ and $y=\phi_a^{-1}(y')$ lie in
$U_{h,i_a,i_b,z}$ and satisfy $d_{i_a}(x,y)<\delta$, yielding
$d_{i_b}(h(x),h(y))<\epsilon$ by~\eqref{e:equicontinuous}, and thus
$d'_b(h'(x'),h'(y'))<\epsilon$. Therefore~\eqref{e:equicontinuous} is
satisfied by $\HH'$, $\{(Z'_a,d'_a)\}_{a\in A}$, the same assignment
$\epsilon\mapsto\delta(\epsilon)$, and the above choice of
neighborhoods $U'_{h',a,b,z'}$. It follows that $\{(Z'_a,d'_a)\}_{a\in
A}$ is a cover of $Z'$ by quasi-locally equal metric spaces (remark~(ii) of
Definition~\ref{d:weakly equicontinuous}), and $\HH'$ is weakly
equicontinuous with respect to the corresponding quasi-local metric,
which obviously induces the given topology of $Z'$.
\end{proof}

On paracompact spaces, the following slightly different
description of weak equicontinuity will be useful to understand the stronger
version.

\begin{lemma}\label{l:weakly equicontinuous'}
  Let $\HH$ a pseudogroups acting on a paracompact quasi-local metric
  space $(Z,{\mathfrak Q})$. Then $\HH$ is weakly equicontinuous if
  and only if there is some $\{(Z_i,d_i)\}_{i\in I}\in{\mathfrak Q}$
  and some symmetric subset $S\subset\HH$ such that any germ of any
  map in $\HH$ is the germ of some map in $S$ and, for every
  $\epsilon>0$, there is some $\delta(\epsilon)>0$ so
  that~\eqref{e:equicontinuous} holds for all $h\in S$, $i,j\in I$ and
  $x,y\in Z_i\cap h^{-1}(Z_j\cap\im h)$.
\end{lemma}

\begin{proof}
  The ``only if'' part follows because, with the notation of
  Lemma~\ref{l:locally finite quasi-locally equal covers} and
  Definition~\ref{d:weakly equicontinuous}, if
$$
U_{h,z}=U_z\cap\bigcap_{i,j\in I,\ z\in Z_i\cap Z_j}U_{h,i,j,z}
$$
for $h\in\HH$, $i,j\in I$ and $z\in\dom h$,
then~\eqref{e:equicontinuous} is satisfied by the family $S$ of
all possible restrictions $h:U_{h,z}\to h(U_{h,z})$.

Reciprocally, for all $h,i,j,z$ as in Definition~\ref{d:weakly equicontinuous}, take $U_{h,i,j,z}$ equal to some open neighborhood of $z$ in $Z_i\cap
h^{-1}(Z_j\cap\im h)$ where $h$ is equal some map in $S$. Then~\eqref{e:equicontinuous} obviously holds for all
$x,y\in U_{h,i,j,z}$.
\end{proof}

The stronger version of equicontinuity is defined by requiring that there is
a set $S$ as in Lemma~\ref{l:weakly equicontinuous'} that is also closed
under compositions, which is some kind
of a non-local condition .

\begin{defn}\label{d:strongly equicontinuous}
Let $\HH$ be a pseudogroup of local homeomorphisms
of a quasi-local metric space $(Z,{\mathfrak Q})$. Then $\HH$ is called
{\em strongly equicontinuous\/} if there exists some
$\{(Z_i,d_i)\}_{i\in I}\in{\mathfrak Q}$ and some symmetric set $S$ of
generators of $\HH$ that is closed under compositions such that, for
every $\epsilon>0$, there is some
$\delta(\epsilon)>0$ so that~\eqref{e:equicontinuous} holds for all
$h\in S$, $i,j\in I$ and $x,y\in Z_i\cap h^{-1}(Z_j\cap\im h)$.
\end{defn}

A pseudogroup $\HH$ acting on a space $Z$ will be called {\em strongly
equicontinuous\/} when it is strongly equicontinuous with respect to some
quasi-local metric inducing the topology of $Z$.

\begin{rems}
(i) A typical choice of $S$ in
Definition~\ref{d:strongly equicontinuous} is the set of all possible
composites of some symmetric set of generators. In fact, given any $S$
satisfying the condition of strong equicontinuity, it is obviously possible to
find a symmetric set of generators $E$ given by restrictions of
elements of $S$, and then the set of all composites of elements of $E$
also satisfies the condition of strong equicontinuity. 
\newline
(ii) If $h\in\HH$ satisfies~\eqref{e:equicontinuous} for all $i,j\in
I$ and $x,y\in Z_i\cap h^{-1}(Z_j\cap\im h)$, then so does its
restriction to any open set. Hence 
$S$ can be assumed to be also closed under restrictions to open sets. 
Nevertheless, Definition~\ref{d:strongly equicontinuous} is not
satisfactory with $S=\HH$ because then the basic test of
Example~\ref{ex:equicontinuous} below is not satisfied (see
Example~\ref{ex:S neq HH}). 
\newline
(iii) The definition of strong
equicontinuity is independent of the choice of $\{(Z_i,d_i)\}_{i\in
I}\in{\mathfrak Q}$. Hence it is possible to assume that $\{Z_i\}_{i\in
I}$ locally finite in Definition~\ref{d:strongly equicontinuous} when
$(Z,{\mathfrak Q})$ is paracompact.
\end{rems}

\begin{example}\label{ex:identity}
The pseudogroup $\HH$ generated by the identity map on any
quasi-local metric space $(Z,{\mathfrak Q})$ is obviously weakly
equicontinuous by the remark~(ii) of
Definition~\ref{d:weakly equicontinuous}. If $(Z,{\mathfrak Q})$ is
paracompact, then $\HH$ is also strongly equicontinuous by
Lemma~\ref{l:locally finite quasi-locally equal covers}; in fact, the
definition of strong equicontinuity is satisfied with $S$ equal to the
family of the identity maps on all finite intersections 
of the sets
$U_z$ given by  Lemma~\ref{l:locally finite quasi-locally equal covers}.
\end{example}

\begin{example}\label{ex:equicontinuous}
Recall that a group $G$ of homeomorphisms of a metric space $(Z,d)$ is
equicontinuous, or better ``equi-uniformly continuous,'' if for every $\epsilon>0$
there is some
$\delta(\epsilon)>0$ such that
$$
d(x,y)<\delta(\epsilon)\Longrightarrow d(g(x),g(y))<\epsilon
$$
for all $g\in G$ and $x,y\in Z$. The pseudogroup $\HH$ generated by such
a $G$ is strongly equicontinuous because
Definition~\ref{d:strongly equicontinuous} is satisfied with $S=G$ and
$\{(Z_i,d_i)\}_{i\in I}=\{(Z,d)\}$. Of course, if $(Z,d)$ is compact, $G$ is
``equicontinuous at every point'' if and only if it is ``equi-uniformly continuous.''
\end{example}

\begin{example}\label{ex:S neq HH}
The group of translations on $\R$ is strongly equicontinuous with
respect to the euclidean metric,
and thus generates a strongly equicontinuous pseudogroup
$\HH$ with respect to the corresponding quasi-local metric. The
following simple argument shows that some proper subset
$S\subset\HH$ must be taken to verify the definition of strong
equicontinuous. Let
$\{Z_i\}_{i\in I}$ be any open covering of $\R$, and $d_i$ a metric
on each $Z_i$ inducing its topology.
For any fixed index $i$, take real numbers $a<b$ such that 
$[a,b]\subset Z_i$.
Let $r>0$ so that
$0<r<\frac{b-a}{3}$. Let
$$
U=(a,a+r)\cup(a+r,a+2r)\;,\quad V=(a,a+r)\cup(b-r,b)\;,
$$
which are contained in $Z_i$, and let
$h:U\to V$ be the map in $\HH$ defined by
$$
h(x)=
\begin{cases}
x&\text{if $a<x<a+r$,}\\
x+b-a-2r&\text{if $a+r<x<a+2r$.}
\end{cases}
$$
The points
$$
x_n=a+\frac{n-1}{n}r\;,\quad y_n=a+\frac{n+1}{n}r
$$
satisfy the following properties: 
\begin{itemize}

\item $x_n,y_n\in U$; 

\item $d_i(x_n,y_n)\to0$ because
$x_n,y_n\to a$ in $Z_i$; and

\item $d_i(h(x_n),h(y_n))>C$ for some $C>0$ since 
$h(x_n)\to a$ and $h(y_n)\to b$ in $Z_i$.

\end{itemize}
Then $h$ does not
satisfy~\eqref{e:equicontinuous} for any $\{(Z_i,d_i)\}_{i\in I}$ as above.
\end{example}

Even though an apparently non-local condition was added to define
strong equicontinuity, the following result shows that this 
property is invariant by equivalences of pseudogroups
acting on locally compact Polish spaces. 

\begin{lemma}\label{l:strongly equicontinuous} If $\HH,\HH'$ are
  equivalent pseudogroups acting on locally compact Polish spaces.
  Then $\HH$ is strongly equicontinuous if and only if $\HH'$ strongly
  equicontinuous.
\end{lemma}

\begin{proof}
Let $\Phi:\HH\to\HH'$ be a pseudogroup equivalence.
Let $Z,Z'$ be the locally compact Polish spaces acted on by $\HH,\HH'$.
Assuming that
$\HH$ is strongly equicontinuous with respect to some
quasi-local metric $\mathfrak Q$ inducing the topology of $Z$,
we will show that so is $\HH'$. Thus $\HH$
satisfies the condition of strong equicontinuity for some
$\{(Z_i,d_i)\}_{i\in I}\in{\mathfrak Q}$ and some symmetric set
$S$ of generators of $\HH$ that is closed under
compositions. By the remark~(ii) of
Definition~\ref{d:strongly equicontinuous}, we can assume that
$S$ is also closed under restrictions to open sets; so every
transformation of
$\HH$ is a combination of maps in $S$. 

\begin{claim}\label{cl:equicontinuous}
There is a subset $\Phi_0\subset\Phi$ such that:
\begin{itemize}

\item For each $\phi\in\Phi_0$, there is some $i\in I$ so that
$\dom\phi\subset Z_i$;

\item $Z'=\bigcup_{\phi\in\Phi_0}\im\phi$; and

\item $\phi^{-1}\circ\psi\in S$ for all
$\phi,\psi\in\Phi_0$.

\end{itemize}
\end{claim}

To prove Claim~\ref{cl:equicontinuous}, first note that, since $Z'$ is 
a locally compact Polish space and $\Phi$ an equivalence, there is a
sequence $\phi_1,\phi_2,\dots$ in $\Phi$ such that:
\begin{itemize}

\item The domain of each $\phi_n$ is
contained in some $Z_i$;

\item $Z'=\bigcup_n\im\phi_n$;

\item the domain of each $\phi_n$ is relatively compact in $Z$; and

\item every $\phi_n$ is a restriction of some $\tilde\phi_n\in\Phi$ with
$\overline{\dom\phi_n}\subset\dom\tilde\phi_n$.

\end{itemize}
Then an increasing sequence of finite subsets $\Phi_{0,n}\subset\Phi$
is defined by induction on $n$ so that $\phi^{-1}\circ\psi\in S$ for all
$\phi,\psi\in\Phi_{0,n}$ and
$$
\im\phi_1\cup\dots\cup\im\phi_n\subset\bigcup_{\phi\in\Phi_{0,n}}\im\phi\;.
$$
Let $\Phi_{0,0}=\emptyset$ to begin with, and assume that $\Phi_{0,n}$
is defined for some $n\ge0$ and satisfies the stated properties. To
define $\Phi_{0,n+1}$, first set
$\Phi'_{0,n}=\Phi_{0,n}\cup\left\{\tilde\phi_{n+1}\right\}$. There is
an open covering
$\UU_{n+1}$ of
$\dom\tilde\phi_{n+1}$ such that the restriction of
$\phi^{-1}\circ\tilde\phi_{n+1}$ to every $U\in\UU_{n+1}$ is in $S$ for
all $\phi\in\Phi'_{0,n}$, which follows since 
$\Phi$ is an equivalence,
$\Phi'_{0,n}$ is finite and every map in $\HH$ is a combination of
maps in $S$. Then the
compact set $\overline{\dom\phi_{n+1}}$ is covered by a finite subfamily
$\UU'_{n+1}\subset\UU_{n+1}$, and let $\Phi_{0,n+1}$ be the
union of $\Phi_{0,n}$ and the set of restrictions of
$\tilde\phi_{n+1}$ to all sets of $\UU'_{n+1}$. We get by induction
that $\phi^{-1}\circ\psi\in S$ for all
$\phi,\psi\in\Phi_{0,n+1}$ because $S$ is symmetric, and
$$
\im\phi_1\cup\dots\cup\im\phi_{n+1}\subset
\bigcup_{\phi\in\Phi_{0,n}}\im\phi\cup\im\phi_{n+1}\subset
\bigcup_{\phi\in\Phi_{0,n+1}}\im\phi\;.
$$
Therefore
Claim~\ref{cl:equicontinuous} follows with
$\Phi_0=\bigcup_n\Phi_{0,n}$.

Now, let $S'$ be the family of all possible
composites $\phi\circ h\circ\psi^{-1}$ for $h\in S$ and
$\phi,\psi\in\Phi_0$, which is symmetric and generates $\HH'$. Moreover
the following argument shows that
$S'$ is closed under compositions. Take
arbitrary elements $h'_1,h'_2\in S'$; thus $h'_1=\phi_1\circ
h_1\circ\psi_1^{-1}$ and
$h'_2=\phi_2\circ h_2\circ\psi_2^{-1}$
for $h_1,h_2\in S$ and 
$\phi_1,\psi_1,\phi_2,\psi_2\in\Phi_0$. Since
$\psi_1^{-1}\circ\phi_2\in S$ by Claim~\ref{cl:equicontinuous},
it follows that
$h_1\circ\psi_1^{-1}\circ\phi_2\circ h_2\in S$ because $S$ is closed
under compositions. So 
$$
h'_1\circ h'_2=\phi_1\circ h_1\circ\psi_1^{-1}\circ\phi_2\circ
h_2\circ\psi_2^{-1}\in S'\;.
$$ 

According to Claim~\ref{cl:equicontinuous}, take any open
covering $\{Z'_a\}_{a\in A}$ of $Z'$ such that, for each $a\in A$, there
is some $\phi_a\in\Phi_0$ and some $i_a\in I$ with
$Z'_a\subset\im\phi_a$ and $\dom\phi_a\subset Z_{i_a}$. Let $d'_a$
denote the restriction to $Z'_a$ of the metric on
$\im\phi_a$ that corresponds via $\phi_a$ to the restriction of $d_{i_a}$
to $\dom\phi_a$. Choose an assignment
$\epsilon\mapsto\delta(\epsilon)>0$ so that $S$ and
$\{(Z_i,d_i)\}_{i\in I}$ satisfy the condition of strong equicontinuity.
Let $h'$ be an arbitrary element of $S'$, which is equal to some
composite $\phi\circ h\circ\psi^{-1}$ for $h\in S$ and
$\phi,\psi\in\Phi_0$.  Let $a,b\in A$ and $x',y'\in Z'_a\cap\dom h'$
with $h'(x'),h'(y')\in Z'_b$, and let $g=\phi_b^{-1}\circ
h'\circ\phi_a$. Since $\psi^{-1}\circ\phi_a$ and
$\phi_b\circ\phi^{-1}$ are in $S$ by Claim~\ref{cl:equicontinuous}, it
follows that
$$
g=\phi_b^{-1}\circ\phi\circ
h\circ\psi^{-1}\circ\phi_a\in S
$$
because $S$ is closed under compositions.
The points $x=\phi_a^{-1}(x')$ and $y=\phi_a^{-1}(y')$ lie in
$Z_{i_a}\cap\dom g$, and
$g(x),g(y)\in Z_{i_b}$. Moreover, if $d'_a(x',y')<\delta(\epsilon)$,
then
$d_{i_a}(x,y)<\delta(\epsilon)$, yielding
$d_{i_b}(g(x),g(y))<\epsilon$ by~\eqref{e:equicontinuous}, and thus
$d'_b(h'(x'),h'(y'))<\epsilon$. It follows that $\{(Z'_a,d'_a)\}_{a\in
A}$ is a cover of $Z'$ by quasi-locally equal metric spaces (remark~(ii) of
Definition~\ref{d:weakly equicontinuous}), and $\HH'$
satisfies the condition of strong equicontinuity with $S'$ and
$\{(Z'_a,d'_a)\}_{a\in A}$. Hence $\HH'$ is strongly equicontinuity with
respect to the quasi-local metric ${\mathfrak Q}'$ represented by
$\{(Z'_a,d'_a)\}_{a\in A}$, which obviously induces the given topology
of $Z'$.
\end{proof}

\begin{problem}
Give mild conditions so that weak and strong equicontinuity
are equivalent. 
\end{problem}

\begin{problem}\label{prob:equicontinuous at every point}
  It is possible to give pseudogroup versions of weak and strong
  ``equicontinuity at every point'' as in Definitions~\ref{d:weakly
    equicontinuous} and~\ref{d:strongly equicontinuous}. Are they
  equivalent to our versions of equicontinuity \upn{(}in the uniform
  sense\upn{)} for compactly generated pseudogroups?
\end{problem}

For the purposes of this paper, the key property of strong
equicontinuity is the following.

\begin{prop}\label{p:equicontinuous}
Let $\HH$ be a compactly generated and strongly equicontinuous
pseudogroup acting on a locally compact Polish quasi-local metric space
$(Z,{\mathfrak Q})$, and let $U$ be any relatively compact open subset
of $(Z,{\mathfrak Q})$ that meets every $\HH$-orbit. Suppose that
$\{(Z_i,d_i)\}_{i\in I}\in{\mathfrak Q}$ satisfies the condition of
Definition~\ref{d:strongly equicontinuous}, $E$ is any system of compact
generation of $\HH$ on $U$, and $\bar g$ satisfies the condition of
Definition~\ref{d:compactly generated} for each $g\in E$. Let
$\{Z'_i\}_{i\in I}$ be any shrinking of
$\{Z_i\}_{i\in I}$. Then there is a finite family $\VV$ of open subsets
of $(Z,{\mathfrak Q})$ whose union contains $U$ and such that, for any
$V\in\VV$, $x\in U\cap V$, and $h\in\HH$ with $x\in\dom h$ and $h(x)\in
U$, the domain of $\tilde h=\bar g_n\circ\dots\circ\bar g_1$ contains
$V$ for any expression 
$h=g_n\circ\dots\circ g_1$ around
$x$ with $g_1,\dots,g_n\in E$, and moreover $V\subset
Z'_{i_0}$ and $\tilde h(V)\subset Z'_{i_1}$ for some
$i_0,i_1\in I$.
\end{prop}

\begin{proof}
We can assume that $\{Z_i\}_{i\in I}$ is locally finite.
Let $S$ be a symmetric set of generators of $\HH$ that is closed under
compositions and restrictions to open subsets so that the condition of strong
equicontinuity is
satisfied by
$S$ and
$\{(Z_i,d_i)\}_{i\in I}$ (Definition~\ref{d:strongly equicontinuous}). 

Observe that any system $E$ of compact generation of $\HH$ on $U$
satisfies the statement of this result if so does some other system of
compact generation of $\HH$ on $U$ whose elements are restrictions of
elements of $E$. Therefore it can be assumed that, for all $g_1,g_2\in
E$, we have
\begin{equation}\label{e:dom bar g1 cup im bar g2}
\dom g_1\cap\im g_2\neq\emptyset \Longrightarrow
\dom\bar g_1\cup\im\bar g_2\subset Z'_i
\end{equation}
for some $i,j\in I$; in particular, since $E$ is symmetric, for each
$g\in E$ there exists some
$i,j\in I$ such that
\begin{equation}\label{e:dom bar g}
\dom\bar g\subset Z'_i\;,\quad\im\bar g\subset Z'_j\;.
\end{equation}
By the same reason, we can also suppose that $\bar g\in S$ for all $g\in
E$.
 
Since $U$ is relatively compact and $\{Z_i\}_{i\in I}$ is locally
finite, $U$ meets only a finite number of the sets $Z_i$. Thus there
exists
$\epsilon>0$ such that
\begin{equation}\label{e:epsilon}
Z'_i\cap\dom g\neq\emptyset\Longrightarrow d_i\left(Z'_i\cap\dom
g,Z_i\setminus\dom\bar g\right)>\epsilon\;,
\end{equation}
for all $i\in I$ and all $g\in E$, because $Z'_i\cap\dom g$ is a
relatively compact subset of $Z_i$. Let $\delta=\delta(\epsilon)>0$
satisfy the condition of strong equicontinuity for such an
$\epsilon$; it is no restriction to assume that $\delta<\epsilon$.

Let $\VV$ be a finite family of open subsets of $Z$ whose union
contains $U$ and such that each $V\in\VV$ is contained in some $Z'_i$
and has $d_i$-diameter smaller than $\delta$. Fix any $V\in\VV$,
$x\in U\cap V$ and $h\in\HH$ with $x\in\dom h$ and $h(x)\in
U$. Since $E$ generates the restriction of $\HH$ to $U$, there exist
$g_1,\dots,g_n\in E$ so that $h=g_n\circ\dots\circ g_1$ in some
neighborhood of $x$, and let $\tilde h=\bar g_n\circ\dots\circ\bar
g_1$. If $V\subset\dom\tilde h$, then $V\subset Z'_{i_0}$ and $\tilde
h(V)\subset Z'_{i_1}$ for some
$i_0,i_1\in I$ by \eqref{e:dom bar g}. So it only remains to show that
$V\subset\dom\tilde h$, which will be done by induction on $n$.

The result is true for $n=1$. Indeed, $V\subset Z'_i$ for some $i$,
and $d_i(x,y)<\delta<\epsilon$ for all $y\in V$. Thus
$V\subset\dom\bar g_1$ by \eqref{e:epsilon}.

For $n>1$, let $f=g_{n-1}\circ\dots\circ g_1$, which is defined in some
neighborhood of $x$. By the induction hypothesis, the domain of $\tilde
f=\bar g_{n-1}\circ\dots\circ\bar g_1$ contains $V$. Then
$$
\dom g_n\cap\im g_{n-1}\neq\emptyset\;,
$$
and thus 
$$
\dom\bar g_n\cup\im\bar g_{n-1}\subset Z'_j
$$ 
for some $j\in I$ by \eqref{e:dom bar g1 cup
im bar g2}. In particular,
$$
\im\tilde f\subset\im\bar g_{n-1}\subset Z'_j\;,
$$
yielding $d_j\left(\tilde f(x),\tilde f(y)\right)<\epsilon$ for all
$y\in V$ by strong equicontinuity since $d_i(x,y)<\delta$ and $\tilde
f\in S$. Therefore $\tilde f(y)\in\dom\bar g_n$ by \eqref{e:epsilon}
because $\tilde f(x)=f(x)\in Z'_j\cap\dom g_n$; {\it i.e.}, the domain
of
$\tilde h=\bar g_n\circ\tilde f$ contains $V$ as desired.
\end{proof}

\begin{rems}
(i) With the notation of Proposition~\ref{p:equicontinuous}, given any
symmetric set $S$ of generators of $\HH$ that is closed under
compositions, we can 
choose $E$ with the extensions $\bar g$ in $S$. So $S$ contains all maps
$\tilde h$ of the statement of Proposition~\ref{p:equicontinuous}.
\newline
(ii) Note that, with the conditions of
Proposition~\ref{p:equicontinuous}, the pseudogroup $\HH$ is 
complete as defined by Haefliger in \cite{Haefliger88}. 
\end{rems}

It makes sense to consider the generalization to complete strongly
equicontinuous  pseudogroups of known results for complete pseudogroups
of local isometries of Riemannian manifolds \cite{Haefliger85},
\cite{Haefliger88}. But, for simplicity, only compactly generated
strongly equicontinuous pseudogroups will be considered in this paper.

\section{Quasi-effective pseudogroups}

Recall the following property that is
invariant by equivalences of pseudogroups; it is interesting for our
results on strongly equicontinuous pseudogroups.

\begin{defn}[{Haefliger \cite{Haefliger85}}]\label{d:quasi-analytic}
A pseudogroup $\HH$ of local transformations of a space $Z$ is called
{\em quasi-analytic\/} when the following holds for every $h\in \HH$: if
$x\in \dom h$ and $h$ is the identity on some open set whose closure
contains $x$, then $h$ is the identity on a neighborhood of $x$.
\end{defn}

\begin{example}
Pseudogroups of local isometries of
Riemannian manifolds are quasi-analytic because every such local isometry with
connected domain is determined by its differential at any point.
\end{example}

\begin{example}\label{ex:not quasi-analytic}
Let $Z$ be the compact subspace of $\R^3$ that is the union of a
horizontal $2$-dimensional euclidean disk centered at the origin with
a compact segment of the vertical axis containing the origin. Consider
the quasi-local metric on $Z$ induced by the restriction of the
euclidean metric of $\R^3$. The space
$Z$ is invariant by rotations around the vertical axis. The
pseudogroup
$\HH$ generated by any such non-trivial rotation is strongly
equicontinuous but not quasi-analytic.
\end{example}

If $\HH$ is a quasi-analytic pseudogroup on a space $Z$, then every
$h\in\HH$ with connected domain is the identity on $\dom h$ if it is
the identity on some non-trivial open set. Because of this,
quasi-analyticity is interesting for our purposes when $Z$ is locally
connected, which is a very strong condition. To remove it,
a slightly different property is defined inspired by the terminology of
group actions.

\begin{defn}\label{d:quasi-effective}
A pseudogroup $\HH$ of local transformations of a space $Z$ is said to
be \textit{quasi-effective} if it is generated by some symmetric set
$S$ that is closed under compositions, and
such that any transformation in $S$ is the identity on its domain if it
is the identity on some non-empty open subset of its domain.
\end{defn}

\begin{rems}
(i) In Definition~\ref{d:quasi-effective}, the family $S$ can be assumed
to be also closed under restrictions to open sets. So every map in
$\HH$ is a combination of maps in $S$ in this case.
\newline
(ii) If the pseudogroup $\HH$ is strongly equicontinuous and
quasi-effective, then $\HH$ is generated by a symmetric subset $S$ that
is closed under compositions and satisfies the conditions of both
Definitions~\ref{d:strongly equicontinuous} and~\ref{d:quasi-effective}.
\end{rems}

The following result can be proved with
arguments similar to those in the proof of
Lemma~\ref{l:strongly equicontinuous}.

\begin{lemma}\label{l:quasi-effective}
If $\HH,\HH'$ are equivalent pseudogroups acting on locally compact
Polish spaces $Z,Z'$, then $\HH$ is quasi-effective if and only if so is
$\HH'$.
\end{lemma}

\begin{lemma}\label{l:quasi-effective implies quasi-analytic}
Any quasi-effective pseudogroup is quasi-analytic.
\end{lemma}

\begin{proof}
Let $\HH$ be a quasi-effective pseudogroup of local transformations of a
space $Z$. So $\HH$ satisfies the condition of
Definition~\ref{d:quasi-effective} with some symmetric set $S$ that
generates $\HH$ and is closed under compositions and
restrictions to open sets (remark~(i) of
Definition~\ref{d:quasi-effective}). Then
$\HH$ is obviously quasi-analytic because any $h\in \HH$ is a
combination of elements of $S$.
\end{proof}

\begin{example}
Let $r_n,s_n$ be two sequences of real numbers satisfying $0<r_n<s_n$
and $r_n,s_n\downarrow0$. For each $n\in\Z_+$, let $U_n$ denote the
(multiplicative) group of $n$th roots of $1$ in $\C$, and fix a
generator $\alpha_n$ of each $U_n$. Then let $Z$ be the compact
subspace of $\R\times\C$ that is the union of the origin and the
subspaces $\{s_n\}\times r_nU_n$, $n\in\Z_+$. Let $\HH$ be the
pseudogroup on $Z$ generated by the homeomorphism $h:Z\to Z$ that
fixes the origin and satisfies $h(s_n,z)=(s_n,\alpha_nz)$ for any
$z\in r_nU_n$ and $n\in\Z_+$. Note that $h$ is an isometry with
respect to the restriction of the metric on $\R\times\C$ induced by
the norm defined by $\|(t,z)\|=\max\{|t|,|z|\}$. So $\HH$ is strongly
equicontinuous with respect to the corresponding quasi-local metric on $Z$.
Moreover
$\HH$ is quasi-analytic because, on the one hand, the origin is the only
non-isolated point of
$Z$ and, on the other hand, if some power $h^m$ is the identity on some
open set whose closure contains the origin, then $m=0$. But $\HH$ is not
quasi-effective, as follows easily by using that, for any
neighborhood $U$ of the origin in $Z$, there is some $m\in\Z_+$ such
that $h^m$ fixes some point in $U$ different from the origin, which is
an open subset.
\end{example}

The above example shows that being quasi-effective is a
strictly stronger property than quasi-analyticity, even for equicontinuous
pseudogroups. Nevertheless, the following result shows that both
properties are equivalent when quasi-analyticity fits our needs. 

\begin{lemma}\label{l:quasi-effective iff quasi-analytic}
Let $\HH$ be a compactly generated strongly equicontinuous
pseudogroup on a locally connected and locally compact Polish space
$Z$. Then
$\HH$ is quasi-effective if and only if it is quasi-analytic.
\end{lemma}

\begin{proof}
The ``only if'' part was shown in Lemma~\ref{l:quasi-effective implies
quasi-analytic}. 

Now assume that $\HH$ is quasi-analytic.
Let $U$ be
any relatively compact open subset of $Z$ that meets every
$\HH$-orbit, and let $\GG$ denote the restriction of $\HH$ to $U$. By
Lemma~\ref{l:quasi-effective}, it is enough to show that $\GG$ is
quasi-effective. Let
$E$ be any system of compact generation of $\HH$ on $U$, and let $\bar
g$ be an extension of each $g\in E$ satisfying the condition of
Definition~\ref{d:compactly generated}.  Take a family $\VV$ of
open  subsets of $Z$ satisfying the statement of
Proposition~\ref{p:equicontinuous} for the above $U$, $E$ and
extensions $\bar g$.  Since $Z$ is locally connected, we can assume
that all sets in $\VV$ are connected. 
Let $S$ be the set of maps $h\in\GG$ such that:
\begin{itemize}

\item $h$ is a restriction of some composite of elements of $E$; and

\item the domain and range of $h$ are contained in elements of $\VV$. 

\end{itemize}
Such an $S$ generates $\GG$, is symmetric, and is closed under
compositions and restrictions to open sets. Suppose that some $h\in S$
is the identity on some non-trivial open subset of its domain. 
We have
$$
h=g_n\circ\dots\circ g_1:O\to P\;,
$$
where $g_1,\dots,g_n\in E$ and $O,P$ are open subsets of $U$ that are
contained in elements of $\VV$; say $O\subset V\in\VV$. Then the domain
of $\tilde h=\bar g_n\circ\dots\circ\bar g_1$ contains $V$ by
Proposition~\ref{p:equicontinuous}. Since $h$ is the identity on some
non-trivial open subset of $O$, the germ of $\tilde h$ at some point of
$V$ is equal to the germ of the identity. So $\tilde h$ is the identity
on $V$ because $\HH$ is quasi-analytic and $V$ is connected. Thus $h$
is the identity on $O$ and the result follows.
\end{proof}

The following result combines strong equicontinuity and
quasi-effectiveness.

\begin{prop}\label{p:A,B}
Let $\HH$ be a compactly generated, strongly equicontinuous and
quasi-effective pseudogroup of local homeomorphisms of a locally
compact Polish space $Z$. Suppose that the
conditions of strong equicontinuity and quasi-effectiveness are satisfied with a
symmetric set
$S$ of generators of
$\HH$ that is closed under compositions \upn{(}Definitions~\ref{d:strongly
equicontinuous} and~\ref{d:quasi-effective}\upn{)}. Let $A,B$ be open subsets of
$Z$ such that $\overline{A}$ is compact and contained in $B$. If $x$
and $y$ are close
enough points in $Z$, then
$$
f(x)\in A\Longrightarrow f(y)\in B
$$
for all $f\in S$ whose domain contains $x$ and $y$.
\end{prop}

\begin{proof}
Suppose that the condition of strong equicontinuity is satisfied with $S$,
some quasi-local metric $\mathfrak Q$, some $\{(Z_i,d_i)\}_{i\in
I}\in{\mathfrak Q}$ such that $\{Z_i\}_{i\in I}$ is locally finite, and some
assignment $\epsilon\mapsto\delta(\epsilon)$.  Let
$\{Z'_i\}_{i\in I}$ be a shrinking of the open covering
$\{Z_i\}_{i\in I}$. Since $A$ is relatively compact and $\{Z_i\}_{i\in I}$
locally finite, $\overline{A}$ only meets finitely many of the sets
$Z_i$. Thus, because each $Z'_i\cap A$ is relatively compact in $Z_i$,
there exists some $\epsilon>0$ such that
\begin{equation}\label{e:A,B}
Z'_i\cap A\neq\emptyset\neq Z_i\setminus B\Longrightarrow
d_i(Z'_i\cap A,Z_i\setminus B)>\epsilon
\end{equation}
for all $i\in I$. 

Fix $x,y\in Z$. Since $\HH$ is compactly
generated, $A$ and $x$ are contained in some relatively compact open $U$ that
meets all orbits. Let $\VV$ be a finite family of open sets that covers
$\overline{U}$ and satisfies the conditions of
Proposition~\ref{p:equicontinuous}. If $x$ and $y$ are close enough, then both of these
points lie in some set $V\in\VV$. Furthermore $V\subset Z_i$ for some
$i\in I$, and we have $d_i(x,y)<\delta(\epsilon)$ if $x$ and $y$ are close enough.

Now take any $f\in S$ with $x,y\in\dom f$ and $f(x)\in A$. According
to Proposition~\ref{p:equicontinuous}, there is some $f'\in S$ whose domain
contains $V$ and so that $f,f'$ have the same germ at $x$; furthermore there is
some $j\in I$ such that $f'(V)\subset Z_j$. In particular,
$f(x)=f'(x)\in Z_j$. Since $y\in\dom f\cap\dom f'$, we also get
$f(y)=f'(y)\in Z_j$ by quasi-effectiveness. Therefore $d_j(f(x),f(y))<\epsilon$ by
strong equicontinuity, and the result follows from \eqref{e:A,B}.
\end{proof}

\section{Coarse quasi-isometry type of orbits with trivial groups of germs}

To compare different orbits of pseudogroups, some connection between
them is needed; so the following terminology will be used. A
pseudogroup $\HH$ acting on a space $Z$ is called transitive when some
orbit is dense in $Z$, and it is called minimal if every orbit is
dense. A non-trivial subset $Y\subset Z$ is called minimal if it is
closed, $\HH$-invariant, and every orbit in $Y$ is dense in $Y$;
equivalently, if $Y$ is a minimal element of the family of all
non-trivial $\HH$-invariant closed subsets of $Z$.

\begin{theorem}\label{t:equicontinuous pseudogroup} 
  Let $\HH$ be a compactly generated, strongly equicontinuous and
  quasi-effective pseudogroup of local transformations of a locally
  compact Polish space $Z$. Assume the space of orbits $Z/\HH$ is
  connected \upn{(}for example, if $\HH$ is transitive\upn{)}.  Let
  $\GG$ denote the restriction of $\HH$ to some relatively compact
  open subset $U\subset Z$ that meets every orbit. Then, with respect
  to any recurrent system of compact generation of $\HH$ on $U$, all
  $\GG$-orbits with trivial group of germs are uniformly coarsely
  quasi-isometric to each other.
\end{theorem}

\begin{proof}
  Let $E$ be a recurrent system of compact generation of $\HH$ on $U$,
  and for each $g\in E$ let $\bar g$ denote its extension satisfying
  the conditions of Definition~\ref{d:compactly generated}.  According
  to Corollary~\ref{c:recurrent systems of compact generation} and
  Theorem~\ref{t:coarsely quasi-isometric orbits}, it may be assumed
  that $\overline{E}=\{\bar g\ |\ g\in E\}$ is also a recurrent system
  of compact generation on some relatively compact open subset
  $U'\subset Z$ with $\overline{U}\subset U'$.  Let $\GG'$ denote the
  restriction of $\HH$ to $U'$.  By considering restrictions of
  elements of $\overline E$ to open subsets of their domains, we can
  assume that $\overline E\subset S$ for some subset $S\subset\HH$
  satisfying the conditions of Definition~\ref{d:quasi-effective}.
  Take a family $\VV$ of open subsets of $Z$ satisfying the statement
  of Proposition~\ref{p:equicontinuous} for the above $U$, $E$ and
  extensions $\bar g$.

  Since $Z$ is a Polish space, the union of orbits with trivial group
  of germs is a dense subset of $Z$. Hence, because $Z/\HH$ is
  connected, it is enough to establish coarse quasi-isometries between
  the $\GG$-orbits of points $x,y\in U$ that are close enough to each
  other and have trivial group of germs; moreover the corresponding
  coarse distortions must be independent of $x$ and $y$. Thus it can
  be assumed that $x$ and $y$ are in the same element $V\in\VV$.
  Consider the map $\phi_{x,y}:\GG(x)\to\GG'(y)$ given by
  $h(x)\mapsto\tilde h(y)$, where $h\in\GG$, $\tilde h\in S$,
  $x\in\dom h$, $V\subset\dom\tilde h$, and both $h,\tilde h$ have the
  same germ at $x$. Here, the germ of $h$ at $x$ is determined by the
  value $h(x)$ because the group of germs at $x$ is trivial. There
  exists such an $\tilde h$ for any $h$ by
  Proposition~\ref{p:equicontinuous} and since $\overline{E}\subset
  S$. Moreover $\tilde h$ is unique on $V$ because $\HH$ satisfies the
  condition of Definition~\ref{d:quasi-effective} with $S$. Note also
  that $\phi_{x,y}$ takes values in $\GG'(y)$ by
  Proposition~\ref{p:equicontinuous}. Therefore $\phi_{x,y}$ is well
  defined.

\begin{claim}\label{cl:phix,y is injective}
    $\phi_{x,y}:\GG(x)\to\GG'(y)$ is injective.
\end{claim}

To prove this claim, take $f_1,f_2\in S$ whose domains contain $V$. So
$\phi_{x,y}(f_1(x))=f_1(y)$ and $\phi_{x,y}(f_2(x))=f_2(y)$. If
$f_1(y)=f_2(y)$, then $f_1,f_2$ have the same germ at $y$ because the
group of germs of $\HH$ at $y$ is trivial. It follows that $f_1=f_2$
on $V$ because both of these maps are in $S$ and their domains contain
$V$. Hence $h_1(x)=h_2(y)$ as desired.

\begin{claim}\label{cl:phix,y le}
We have 
$$
d_{\overline{E}}(\phi_{x,y}(z_1),\phi_{x,y}(z_2))\le d_E(z_1,z_2)
$$
for all $z_1,z_2\in\GG(x)$.
\end{claim}

We now show this assertion. We have $z_1=f_1(x)$, $z_2=f_2(x)$, 
$\phi_{x,y}(z_1)=f_1(y)$ and
$\phi_{x,y}(z_2)=f_2(y)$ for
some $f_1,f_2\in S$ whose domains contain $V$. Suppose
$d_E(z_1,z_2)=k\ge0$ in the orbit $\GG(x)$. This means that there
is a minimal decomposition $h_2\circ h_1^{-1}=g_k\circ\dots\circ g_1$
about $z_1$ with $g_1,\dots,g_k\in E$. Hence $f_2\circ f_1^{-1}$ and
$\bar g_k\circ\dots\circ \bar g_1$ are equal on
$f_1(V)$ because both of these maps are in $S$ and their domains
contain $f_1(V)$. This yields
$d_{\overline{E}}\left(f_1(y),f_2(y)\right)\leq k$ in $\GG'(y)$,
which finishes the proof of Claim~\ref{cl:phix,y le}.

Let $A$ be an open subset of $Z$ intersecting every $\HH$-orbit and such
that $\overline{A}\subset U$.

\begin{claim}\label{cl:le C phix,y}
There is some $C>0$, independent of $x,y$, such that 
$$
d_E(z_1,z_2)\le
C\,d_{\overline{E}}\left(\phi_{x,y}(z_1),\phi_{x,y}(z_2)\right)
$$
for all $z_1,z_2\in\GG(x)\cap A$.
\end{claim}

To prove this estimate, take any $z_1,z_2\in\GG(x)$. Again, there are
$f_1,f_2\in S$, whose domains contain
$V$, such that $z_1=f_1(x)$, $z_2=f_2(x)$,
$\phi_{x,y}(z_1)=f_1(y)$ and
$\phi_{x,y}(z_2)=f_2(y)$. Suppose
$d_{\overline{E}}(f_1(y),f_2(y))=k$. Then there is a
decomposition $f_2\circ f_1^{-1}=\bar
g_k\circ\dots\circ\bar g_1$ around
$f_1(y)$ for some $g_1,\dots,g_k\in E$. So $f_2\circ f_1^{-1}=
\bar g_k\circ\dots\circ\bar g_1$ on
$f_1(V)$ because both of these maps are in $S$ and their domains contain
$f_1(V)$. It follows that $d_{\overline{E}}(z_1,z_2)\le k$. Hence
$d_E(z_1,z_2)\le Ck$ for some $C>0$ independent of $x$ and $y$ by
Lemma~\ref{l:coarsely quasi-isometric orbits} since $z_1,z_2\in A$ and
$E,\overline{E}$ are recurrent systems of compact generation on $U,U'$.

\begin{claim}\label{cl:GG'y cap A subset phix,y(GGx)}
If $x,y$  are close enough, then $\GG'(y)\cap
A\subset\phi_{x,y}(\GG(x))$.
\end{claim}

By Proposition~\ref{p:A,B}, if $x,y$ are close enough in $V$, then 
\begin{equation}\label{e:A,U}
f(y)\in A\Longrightarrow f(x)\in U
\end{equation}
for all $f\in S$ whose domain contains $V$. 
Then $\GG'(y)\cap
A=\GG(y)\cap A$ and thus every point in $\GG'(y)\cap A$ can be written
as $f(y)$, for some
$f\in\GG$ whose domain contains $y$.  By
Proposition~\ref{p:equicontinuous}, there exists $\tilde f\in\HH$
whose domain contains $V$, and such that $f$ and $\tilde f$ have the
same germ at $y$. By~\eqref{e:A,U}, from $\tilde f(y)=f(y)\in A$,
it follows that $\tilde f(x)\in U$. Thus the restriction $h$ of
$\tilde f$ to some neighborhood of $x$ is in $\GG$, and so
$$
f(y)=\tilde f(y)=\phi_{x,y}(h(x))\in\phi_{x,y}(\GG(x))\;,
$$
which shows
Claim~\ref{cl:GG'y cap A subset phix,y(GGx)}.

Since $\overline{E}$ is recurrent, Lemma~\ref{l:recurrent finite
symmetric family of generators} implies that there exist $R,R'>0$ such
that every $d_E$-ball of radius $R$ in any $\GG$-orbit
meets $A$, as well as every $d_{\overline{E}}$-ball of radius $R'$ in
any $\GG'$-orbit. Therefore $\phi_{x,y}(\GG(x)\cap A)$ is an
$(R+R')$-net in $\left(\GG'(y),d_{\overline{E}}\right)$ by
Claims~\ref{cl:phix,y le} and~\ref{cl:GG'y cap A subset phix,y(GGx)}.
Moreover 
$$
\phi_{x,y}:(\GG(x)\cap A,d_E)\to\left(\phi_{x,y}(\GG(x)\cap
A),d_{\overline{E}}\right)
$$
is a bi-Lipschitz map whose distortion is independent of $x,y$ by
Claims~\ref{cl:phix,y is injective},~\ref{cl:phix,y le} and~\ref{cl:le
C phix,y}. Hence $(\GG(x),d_E)$ is
coarsely quasi-isometric to
$\left(\GG'(y),d_{\overline{E}}\right)$, where the coarse distortions
are independent of the choices of $x,y$. The result now follows from
Theorem~\ref{t:coarsely quasi-isometric orbits} since $E$ and
$\overline{E}$ are recurrent systems of compact generation of $\HH$ on
$U,U'$.
\end{proof}

\begin{example}
  In the pseudogroup $\HH$ on $Z$ of Example~\ref{ex:not
    quasi-analytic}, all orbits have trivial groups of germs, except
  the origin. Moreover $Z$ is locally compact and locally connected,
  and $Z/\HH$ is connected because so is $Z$. But, if $\HH$ is
  generated by an irrational rotation, the statement of
  Theorem~\ref{t:equicontinuous pseudogroup} does not hold because
  $\HH$ is not quasi-analytic.  Indeed, there are two coarse
  quasi-isometry types of orbits with trivial group of germs: the
  orbits of the points in the vertical segment are trivial, and all
  other orbits are quasi-isometric to the integers.
\end{example}

An action of a group $\Gamma$ on a space will be called {\em
quasi-effective\/} when it generates a quasi-effective pseudogroup. A
quasi-effective action may not be effective, as shown by the following
example.

\begin{example}
Let $Z$ be a finite discrete space with more than two elements, and let
$\Gamma$ be the group of all permutations of $Z$. Then the canonical
action of $\Gamma$ on $Z$ is not effective but it is quasi-effective:
the condition of Definition~\ref{d:quasi-effective} is satisfied with the
set $S$ of maps
$\{x\}\to\{y\}$ with $x,y\in Z$.
\end{example}

\begin{cor}\label{c:equicontinuous action} Let $\Gamma$ be a finitely
  generated group acting quasi-effectively and equicontinuously on a
  compact separable metric space $Z$ with connected space of orbits
  \upn{(}for example, if some orbit is dense\upn{)}. Then all orbits
  with trivial group of germs are uniformly coarsely quasi-isometric
  to each other.
\end{cor}

\section{Minimality of the orbit closures}

For compactly generated pseudogroups of local isometries of a
Riemannian manifold, the closures of the orbits are manifolds, and the
restriction of the pseudogroup to each orbit closure is a minimal pseudogroup
\cite[Appendix D]{Mol88}---this is just a pseudogroup version of Molino's 
theory for Riemannian foliations. In the more general situation
considered here, at least the minimality of the orbit closures holds true.

\begin{theorem}\label{t:minimal orbit closures}
Let $\HH$ be a compactly generated strongly equicontinuous
pseudogroup of local transformations of a locally compact Polish space
$Z$. Then 
the closure of each orbit is a minimal set. In particular, such a
pseudogroup is 
minimal if it is transitive.
\end{theorem}

\begin{proof}
It is required to show that if the orbit of a point $x\in Z$ approaches
another point $y$, then the orbit of $y$ also approaches $x$. If $U$ is
a relatively compact open subset of $Z$ that meets every
$\HH$-orbit, then it can be assumed that $x,y\in U$.

Let $\VV$ be a finite family of open subsets of $Z$ whose union covers
$U$ as in Proposition~\ref{p:equicontinuous}. Let $V,W\in\VV$ be such
that $x\in V$ and $y\in W$. If $h_n\in\HH$ is a sequence such that
$h_n(x)\to y$, then it can be assumed that $\dom h_n=V$ and $h_n(x)\in
W$ for all $n$. Moreover, there are maps $f_n\in\HH$ with $\dom
f_n=W$, and such that $f_n$ and $h_n^{-1}$ have the same germ at
$h_n(x)$ for all $n$. Strong equicontinuity of $\HH$ then implies that
$f_n(y)\to x$ as follows.  Suppose that the condition of strong
equicontinuity of $\HH$ is satisfied for a locally finite 
covering $\{(Z_i,d_i)\}_{i\in I}$ of $Z$ by quasi-locally equal metric spaces,
some symmetric set $S$ of generators of
$\HH$ that is closed under compositions, and some assignment
$\epsilon\mapsto\delta(\epsilon)$
(Definition~\ref{d:strongly equicontinuous}). It may be assumed that
$V\subset Z_i$ and $W\subset Z_j$ for some $i,j\in I$, and that $h_n,f_n\in S$ for
all $n$. Given $\epsilon>0$,  there exists an integer
$N>0$ such that $d_j(h_n(x),y)<\delta(\epsilon)$ for all $n\ge N$.
Hence 
$$
d_i(x,f_n(y))=d_i(f_n\circ h_n(x),f_n(y))<\epsilon
$$ 
for all $n\ge N$ by strong equicontinuity.
\end{proof}

\begin{cor}\label{c:partition}
  Let $\HH$ be a compactly generated and strongly equicontinuous
  pseudogroup of local transformations of a locally compact Polish
  space $Z$. Then the orbit closures of $\HH$ define a partition of
  $Z$.
\end{cor}

\section{The closure of a strongly equicontinuous pseudogroup}

In the study of pseudogroups of local isometries of Riemannian
manifolds, an important role is played by the closure of such a
pseudogroup \cite{Haefliger88}.  It is defined by using the space
of $1$-jets, which is not available in our more general setting. But
the closure of our type of pseudogroups can be also defined by using
the compact-open topology on the spaces of local transformations
defined on small enough open subsets.

As usual, for spaces $Y,Z$, let $C(Y,Z)$ denote the set of continuous
maps $Y\to Z$, which will be denoted by $C_{\text{\rm c-o}}(Y,Z)$ when
it is endowed with the compact-open topology. For open subspaces $O,P$
of a space $Z$, the space $C_{\text{\rm c-o}}(O,P)$ will be considered
as an open subspace of $C_{\text{\rm c-o}}(O,Z)$ in the canonical way.

\begin{theorem}\label{t:closure}
  Let $\HH$ be a quasi-effective, compactly generated and strongly
  equicontinuous pseudogroup of local transformations of a locally
  compact Polish space $Z$. Let $S$ be a symmetric set of generators
  of $\HH$ that is closed under compositions and restrictions to open
  subsets, and satisfies the conditions of strong equicontinuity and
  quasi-effectiveness \upn{(}Definitions~\ref{d:strongly
    equicontinuous} and~\ref{d:quasi-effective}\upn{)}. Let
  $\widetilde{\HH}$ be the set of maps $h$ between open subsets of $Z$
  that satisfy the following property: for every $x\in\dom h$, there
  exists a neighborhood $O_x$ of $x$ in $\dom h$ so that $h|_{O_x}$
 is in the
  closure of $C(O_x,Z)\cap S$ in $C_{\text{\rm c-o}}(O_x,Z)$. Then:
\begin{itemize}
  
\item $\widetilde{\HH}$ is closed under composition, combination and
  restriction to open sets;
  
\item every map in $\widetilde{\HH}$ is a homeomorphism around every
  point of its domain;
  
\item the maps of $\widetilde{\HH}$ that are homeomorphisms form a
  pseudogroup $\overline{\HH}$ that contains $\HH$;

\item $\overline{\HH}$ is strongly equicontinuous;
  
\item the orbits of $\overline{\HH}$ are equal to the closures of the
  orbits of $\HH$; and
  
\item $\widetilde{\HH}$ and $\overline{\HH}$ are independent of the
  choice of $S$.

\end{itemize}
\end{theorem}

\begin{proof}
  The family $\widetilde{\HH}$ is obviously closed under combination
  of maps and restrictions to open sets. Moreover $\widetilde{\HH}$ is
  closed under composition of maps because $Z$ is locally compact
  Hausdorff (see {\it e.g.} \cite[p.~289, Exercise~4]{munkres}).
  
  Given any relatively compact open subset $U$ that meets all
  $\HH$-orbits, by Proposition~\ref{p:equicontinuous}, its remark~(i)
  and Proposition~\ref{p:A,B}, there is some finite family $\VV$ of
  open subsets of $Z$ and another relatively compact open set $U_0$
  such that:
\begin{itemize}

\item $\overline{U}$ is covered by the family $\VV$;
  
\item any germ of any map in the restriction of $\HH$ to $U$ is a germ
  of some map in $S$ whose domain belongs to $\VV$; and
  
\item $f(V)\subset U_0$ for any $V\in\VV$ and $f\in S$ with
  $V\subset\dom f$ and $f(V)\cap U\ne\emptyset$. In particular, any
  $V\in\VV$ is contained in $U_0$.

\end{itemize}

For any map $h:V\to W$ in $\widetilde{\HH}$ and any $x\in V$, we will
show that there is some neighborhood $O$ of $x$ in $V$ such that the
restriction $h:O\to h(O)$ is a homeomorphism whose inverse is also in
$\widetilde{\HH}$.  It can be assumed that there are some $U$, $U_0$
and $\VV$ as above so that $V,W\in\VV$. Furthermore we can suppose
that $h$ is the limit in $C_{\text{\rm c-o}}(V,Z)$ of some sequence of
maps $h_n\in C(V,Z)\cap S$. Take any open neighborhood $V'$ of $x$
with $\overline{V'}\subset V$. Since $\overline{V'}$ is compact, it
follows that $h_n\left(\overline{V'}\right)\subset W$ for $n$ large
enough.

The germ of each $h_n^{-1}$ at $h_n(x)$ is equal to the germ of some
$f_n\in S$ whose domain is $W$. By quasi-effectiveness, each $f_n$ is
equal to $h_n^{-1}$ on $h_n(V)\cap W$, which contains $h_n(V')$. Hence
$V'\subset\im f_n$, and $f_n^{-1}$ is equal to $h_n$ on $V'$.

By strong equicontinuity, the set $C(W,U_0)\cap S$ is equicontinuous
in the usual sense. So, by the Ascoli theorem and the compactness of
$\overline{U_0}$, we can assume that $f_n$ is convergent to some map
$f$ in $C_{\text{\rm c-o}}(W,Z)$, which is in $\widetilde{\HH}$.

We have that $h_n(x)\to y=h(x)$, yielding $f_n(y)\to x$ by strong
equicontinuity as in the proof of Theorem~\ref{t:minimal orbit
  closures}. Therefore $f(y)=x\in V'$, and there is some open
neighborhood $W'$ of $y$ with $\overline{W'}\subset W$ and
$f\left(\overline{W'}\right)\subset V'$.  Since $\overline{W'}$ is
compact, we get $f_n\left(\overline{W'}\right)\subset V'$ for $n$
large enough. So $f_n^{-1}$ is equal to $h_n$ on $f_n(W')$, yielding
that the composite $h_n\circ f_n$ is the identity on $W'$ for $n$
large enough. It follows that $h\circ f$ is the identity on $W'$
because $Z$ is locally compact Hausdorff. Similarly, for any open
neighborhood $O$ of $x$ with $\overline{O}\subset V'$ and
$h\left(\overline{O}\right)\subset W'$, we get that $f\circ h$ is the
identity on $O$. So $h:O\to h(O)$ is a homeomorphism whose inverse is
$f:h(O)\to O$, which is in $\widetilde{\HH}$ as desired.

Now, from what was proved for $\widetilde{\HH}$, it follows directly
that $\overline{\HH}$ is a pseudogroup. Moreover $\overline{\HH}$
contains $\HH$ because $\widetilde{\HH}$ contains $S$ by definition.

We now show that $\overline{\HH}$ is strongly equicontinuous. Suppose that
$\HH$ satisfies the condition of strong equicontinuity
(Definition~\ref{d:strongly equicontinuous}) with the above $S$, some
quasi-local metric $\mathfrak Q$, some $\{(Z_i,d_i)\}_{i\in
  I}\in{\mathfrak Q}$ with $\{Z_i\}_{i\in I}$ locally finite, and some
assignment $\epsilon\mapsto\delta(\epsilon)$. Let $\overline{S}$ be
the set of homeomorphisms that are in the union of the closures of
$C(O,Z)\cap S$ in $C_{\text{\rm c-o}}(O,Z)$ with $O$ running on the
open sets of $Z$. By definition, every element of $\overline{\HH}$ is
a combination of maps in $\overline{S}$. Since $S$ is closed under
restrictions to open sets, it easily follows that so is
$\overline{S}$. The set $\overline{S}$ is also closed under
compositions because so is $S$ and $Z$ is locally compact Hausdorff.
Moreover $\overline{S}$ is symmetric since it is closed under
restrictions to open sets and because $\overline{\HH}$ is a
pseudogroup whose elements are combinations of maps in $\overline{S}$.
A typical ``$\epsilon/3$-argument'' will show that $\overline{\HH}$
satisfies the strong equicontinuity condition with $\overline{S}$ and
the above family $\{(Z_i,d_i)\}_{i\in I}$. Take any $h:O\to P$ in
$\overline{S}$, $i,j\in I$ and $x,y\in Z_i\cap h^{-1}(Z_j\cap\im h)$.
Suppose that $d_i(x,y)<\delta(\epsilon/3)$ for some $\epsilon>0$. Such
an $h$ is the limit of some sequence of maps $h_n\in C(O,Z)\cap S$ in
$C_{\text{\rm c-o}}(O,Z)$. On the one hand, since the compact-open
topology is equal to the topology of uniform convergence on compact
sets, it follows that $d_j(h_n(x),h(x))<\epsilon/3$ and
$d_j(h_n(x),h(x))<\epsilon/3$ for $n$ large enough. On the other hand,
we have $d_j(h_n(x),h_n(y))<\epsilon/3$ for all $n$ since $h_n\in S$.
Therefore $d_j(h(x),h(y))<\epsilon$ as desired by the triangle
inequality.

We now show that the orbits of $\overline{\HH}$ are equal to the orbit
closures of $\HH$. Given two points $x,y$ in the same orbit closure of
$\HH$, it has to be shown that $x,y$ are in the same orbit of
$\overline{\HH}$.
There is a sequence $h_n\in\HH$ with $h_n(x)\to y$.  It can be assumed
that $x,y\in U$ for some relatively compact open set $U$ that meets
all $\HH$-orbits. As above, by Proposition~\ref{p:equicontinuous}, its
remark~(i) and Proposition~\ref{p:A,B}, we can suppose that $h_n\in
S$, $\dom h_n=V$ and $h_n(V)\subset U_0$ for some fixed open set $V$
and some relatively compact open set $U_0$. Thus $h_n$ is a sequence
in $C(V,U_0)\cap S$, which is an equicontinuous family of maps.
Therefore we can assume that $h_n$ is convergent in $C_{\text{\rm
    c-o}}(V,Z)$ by the Ascoli theorem, and let $h$ be its limit.  Then
$h(x)=y$ and $h\in\widetilde{\HH}$ by definition. Thus $x,y$ are in
the same orbit of $\overline{\HH}$ because the restriction of $h$ to
some open neighborhood of $x$ is in $\overline{\HH}$.

Finally $\overline{\HH}$ is independent of the choice of $S$ because
it is the pseudogroup generated by the local transformations of $Z$
lying in the union of closures of $C(O,Z)\cap\HH$ in $C_{\text{\rm
    c-o}}(O,Z)$ with $O$ running on the open sets of $Z$. Obviously,
$\overline{\HH}$ is independent of $S$ if and only if 
$\widetilde{\HH}$ is also.
\end{proof}

\begin{defn}\label{d;closure}
  Let $\HH$ be a quasi-effective, compactly generated and strongly
  equicontinuous pseudogroup of local transformations of a locally
  compact Polish space $Z$. With the notation of
  Theorem~\ref{t:closure}, the pseudogroup $\overline{\HH}$ is called
  the {\em closure\/} of $\HH$.
\end{defn}

As a direct consequence of Theorem~\ref{t:closure}, in the present
general setting, the orbit closures satisfy the following property of
manifolds.

\begin{defn}\label{homog}
  A topological space is \textit{homogeneous} if the pseudogroup of
  all local homeomorphisms has exactly one orbit.
\end{defn}

\begin{cor}\label{c:homogeneous orbit closures}
  Let $\HH$ be a quasi-effective, compactly generated and strongly
  equicontinuous pseudogroup of local transformations of a locally
  compact Polish space $Z$. Then the closure of each orbit is
  homogeneous.
\end{cor}

\section{Local metric spaces}

Pseudogroups of local isometries make sense on metric spaces but, with
more generality, this type of pseudogroup can be defined on local
metric spaces, which are introduced as follows.

\begin{defn}\label{d:local metric}
  Two metrics on the same set are said to be {\em locally equal\/}
  when they induce the same topology and each point has a neighborhood
  where both metrics are equal.  Let $\{(Z_i,d_i)\}_{i\in I}$ be a
  family of metric spaces such that $\{Z_i\}_{i\in I}$ is a covering
  of a set $Z$, each intersection $Z_i\cap Z_j$ is open in $(Z_i,d_i)$
  and $(Z_j,d_j)$, and the metrics $d_i,d_j$ are locally equal on
  $Z_i\cap Z_j$ whenever this is a non-empty set. Such a family will
  be called a {\em cover of $Z$ by locally equal metric spaces\/}. Two
  such families are called {\em locally equal\/} when their union also
  is a cover of $Z$ by locally equal metric spaces. This is an
  equivalence relation whose equivalence classes are called {\em local
    metrics\/} on $Z$. For each local metric ${\mathfrak D}$ on $Z$,
  the pair $(Z,{\mathfrak D})$ is called a {\em local metric space\/}.
\end{defn}

\begin{rems}
(i) Observe the analogy between the definitions of local metrics and
quasi-local metrics: for every local metric
$\mathfrak D$, there is a unique quasi-local metric $\mathfrak Q$ so that
${\mathfrak D}\subset{\mathfrak Q}$. In particular, all topological
properties of quasi-local metric spaces hold for local metric spaces.
\newline
(ii) In contrast with quasi-local metrics, local metrics can be also
characterized as maximal covers of $Z$ by locally equal metric spaces; 
there always exist such maximal families.
\newline
(iii) The concept of local metric has the following sheaf theoretic
description, which shows its naturality. Suppose that the set $Z$ is
endowed with a topology {\it a priori}, even though this topology will
be later determined by the local metric. Then, for each open subset
$U\subset Z$, let ${\mathcal M}(U)$ denote the set of all metrics on $U$
that induce its topology. Such an $\mathcal M$ is a presheaf on $Z$
with the usual restriction of metrics, and a local metric on $Z$ is
just a global section of the sheaf $\widetilde{\mathcal M}$ determined
by $\mathcal M$. By Example~\ref{ex:x,y} below, the presheaf
$\mathcal M$ is a sheaf only in the uninteresting case where the only
metrizable open sets contain just one point. 
\end{rems}

\begin{example}\label{ex:x,y}
  If $Z$ is metrizable and contains at least two points $x,y$, then
  there are infinitely many metrics that are locally equal to any
  given metric $d$ inducing the topology of $Z$; for instance, all the
  metrics $d_r$, $0<r<d(x,y)$, given by $d_r(z,z')=\min\{d(z,z'),r\}$
  for $z,z'\in U$.
\end{example}

\begin{example}\label{ex:Z}
  Let $B$ be any open disc in $\R^2$, $Z=\R^2\setminus\overline B$.
  Let $d$ denote the restriction of the euclidean distance of $\R^2$
  to $Z$, and $d'$ the distance map on $Z$ induced by the restriction
  of the Riemannian metric of $\R^2$. Also, let $[x,y]$ denote the
  segment that joins each pair of points $x,y\in Z$. We have
  $d(x,y)=d'(x,y)$ if $[x,y]\cap B=\emptyset$, and $d(x,y)<d'(x,y)$
  otherwise. So both metrics $d,d'$ are locally equal, and thus define
  the same local metric space $(Z,{\mathfrak D})$.
\end{example}

The proof of the following result is analogous to the proof of
Lemma~\ref{l:locally finite quasi-locally equal covers}.

\begin{lemma}\label{l:locally finite local metrics} Let
  $(Z,{\mathfrak D})$ be a local metric space. If $\{Z_i\}_{i\in I}$
  is locally finite for some $\{(Z_i,d_i)\}_{i\in I}\in{\mathfrak D}$,
  then there is some open neighborhood $U_z$ of each $z\in Z$ such
  that the metrics $d_i,d_j$ are equal on $U_z\cap Z_i\cap Z_j$ for
  all $i,j\in I$.
\end{lemma}

Each metric $d$ on a set $Z$ induces a unique local metric
$\mathfrak D$ so that $\{(Z,d)\}\in{\mathfrak
D}$. The following shows that the reciprocal holds when $(Z,{\mathfrak
D})$ is Hausdorff and paracompact. 

\begin{theorem}\label{t:local metric}
  A local metric space $(Z,{\mathfrak D})$ is induced by some metric
  on $Z$ if and only if $(Z,{\mathfrak D})$ is Hausdorff and
  paracompact.
\end{theorem}

\begin{proof}
The ``only if'' part holds by the Stone theorem (see {\it e.g.}
\cite[Theorem~20.9]{Willard}). Now suppose that
$(Z,{\mathfrak D})$ is Hausdorff and paracompact.

\begin{claim}\label{cl:Ua,Da}
There is some $\{(U_a,D_a)\}_{a\in A}\in{\mathfrak D}$ such that
$\{U_a\}_{a\in A}$ is locally finite, and
$D_a,D_b$ are equal on $U_a\cap U_b$ for all $a,b\in A$ with
$U_a\cap U_b\neq\emptyset$.
\end{claim}

Indeed, since
$(Z,{\mathfrak D})$ is paracompact, there is some
$\{(Z_i,d_i)\}_{i\in I}\in{\mathfrak D}$ such that $\{Z_i\}_{i\in I}$
is locally finite. Let $\{Z'_i\}_{i\in I}$ be a shrinking of
$\{Z_i\}_{i\in I}$. For each $i\in I$ and $x\in Z'_i$, let $V_{i,x}$ be
an open neighborhood of $x$ that is contained in $Z'_i$ and meets only
a finite number of sets $Z_j$ for $j\in I$. Therefore, for any $y\in
V_{i,x}$, there is some open neighborhood $W_{i,x,y}$ of $y$ in
$V_{i,x}$ such that:
\begin{itemize}

\item If $y\not\in\overline{Z'_j}$ for some $j\in I$, then
$W_{i,x,y}\cap Z'_j=\emptyset$;

\item if $y\in\overline{Z'_j}$ for some $j\in I$, then
$W_{i,x,y}\subset Z_j$ and the metrics $d_i,d_j$ are equal on
$W_{i,x,y}$.

\end{itemize}
Again, because $(Z,{\mathfrak D})$ is paracompact, there is an open
locally finite refinement $\{U_a\}_{a\in A}$ of the open cover given by
all possible sets $W_{i,x,y}$ as above. For each $a\in A$, choose any
$W_{i,x,y}$ containing $U_a$, and let $D_a$ denote the restriction of
$d_i$ to $U_a$. Then Claim~\ref{cl:Ua,Da} follows easily with such a
family $\{(U_a,D_a)\}_{a\in A}$.

A metric $D$ on $Z$ is now defined as follows. With the notation
of Claim~\ref{cl:Ua,Da}, let $\{U'_a\}_{a\in A}$ be a shrinking of the
open covering $\{U_a\}_{a\in A}$. A pair $(z_1,z_2)\in
Z\times Z$ will be said to be admissible if there is some $a\in A$ such
that
$z_1,z_2\in U'_a$, and moreover 
$$
\{z_1,z_2\}\cap U'_b\neq\emptyset\Longrightarrow \{z_1,z_2\}\subset U_b
$$
for all $b\in A$. For each 
$(x,y)\in Z\times Z$, let $S_{x,y}$ denote the set of all finite
sequences $(z_0,\dots,z_n)$ in $Z$, with arbitrary length $n\in\Z_+$,
such that $z_0=x$, $z_n=y$, and $(z_{k-1},z_k)$ is an admissible pair
for every $k=1,\dots,n$. Then set
$D(x,y)=1$ if $S_{x,y}=\emptyset$, and let
$$
D(x,y)=
\inf_{(z_0,\dots,z_n)\in S_{x,y}}\sum_{k=1}^nD_{a_k}(z_{k-1},z_k)
$$
if $S_{x,y}\neq\emptyset$, where $z_{k-1},z_k\in U'_{a_k}$ with
$a_k\in A$ for each
$k=1,\dots,n$. This definition is
independent of the choices of the indices $a_k$ by
Claim~\ref{cl:Ua,Da}.

\begin{claim}\label{cl:D(x,y)}
Let $a\in A$, $x\in U'_a$ and $y\in Z$ with $S_{x,y}\neq\emptyset$. Then
$$
D(x,y)\ge
\begin{cases}
\min\{D_a(x,y),D_a(x,U_a\setminus U'_a)\}&\text{if $y\in U'_a$,}\\
D_a(x,U_a\setminus U'_a)&\text{if $y\not\in U'_a$.}
\end{cases}
$$
\end{claim}

To prove this assertion, let
$(z_0,\dots,z_n)\in S_{x,y}$ and $a_1,\dots,a_k\in A$ with
$z_{k-1},z_k\in U'_{a_k}$ for $k=1,\dots,n$.
On the one hand, if
$z_0,\dots,z_n\in U'_a$, we have
$$
\sum_{k=1}^nD_{a_k}(z_{k-1},z_k)=\sum_{k=1}^nD_a(z_{k-1},z_k)
\ge D_a(z_0,z_n)=D_a(x,y)
$$
 by Claim~\ref{cl:Ua,Da}. On the other hand, suppose
$\{z_0,\dots,z_n\}\not\subset U'_a$. Then
$n\ge1$, and let
$$
n_0=\min\{k\in\{1,\dots,n\}\ |\ z_k\not\in U'_a\}\;.
$$
Since $z_{n_0-1}\in U'_a$, we get
$z_{n_0}\in U_a$ because $(z_{n_0-1},z_{n_0})$ is an admissible pair.
So
$$
\sum_{k=1}^nD_{a_k}(z_{k-1},z_k)
\ge\sum_{k=1}^{n_0}D_{a_k}(z_{k-1},z_k)
\ge D_a(z_0,z_{n_0})\ge D_a(x,U_a\setminus U'_a)
$$
by Claim~\ref{cl:Ua,Da}, which completes the proof of
Claim~\ref{cl:D(x,y)}.

The above $D$ is a pseudometric on $Z$ because the following holds for
all $x,y,z\in Z$:
\begin{gather*}
(x,x)\in S_{x,x}\;,\\
(z_0,\dots,z_n)\in S_{x,y}\Longrightarrow(z_n,\dots,z_0)\in S_{y,x}\;,\\
\begin{array}.{r}\}
(z_0,\dots,z_m)\in S_{x,y}\\
\quad(z_m,\dots,z_{m+n})\in S_{y,z}
\end{array}
\Longrightarrow(z_0,\dots,z_{m+n})\in S_{x,z}\;.
\end{gather*}
To show that $D$ is indeed a metric, suppose
$D(x,y)=0$ for some $x,y\in Z$; thus $S_{x,y}\neq\emptyset$.
Take any $a\in A$ with $x\in U'_a$. Since $D_a(x,U_a\setminus U'_a)>0$,
it follows from Claim~\ref{cl:D(x,y)} that $y\in U'_a$ and $D_a(x,y)\le
D(x,y)=0$. So $x=y$ as desired because $D_a$ is a metric. 

It remains to check that $\{(Z,D)\}\in{\mathfrak D}$. 
Fix any $z\in Z$ and any $a_0\in A$ with $z\in U'_{a_0}$. The following
assertion follows easily because $\{U_a\}_{a\in A}$ is locally finite
and $\{U'_a\}_{a\in A}$ is a shrinking of $\{U_a\}_{a\in A}$.

\begin{claim}\label{cl:Pz}
There is some open neighborhood $P_z$ of $z$ in
$U'_{a_0}$ such that
$$
P_z\cap U'_a\neq\emptyset\Longrightarrow P_z\subset U_a
$$
for all $a,b\in A$.
\end{claim}

Since
$\{U_a\}_{a\in A}$ is locally finite and $(Z,{\mathfrak D})$ is
Hausdorff, the set
$$
O_x=\bigcap_{a\in A,\ x\in U'_a}U_a
\setminus\bigcup_{b\in A,\ x\not\in U_b}\overline{U'_b}
$$
is an open neighborhood of every $x$ in $Z$. If $x\in U'_{a_0}$, it is
easy to see that $(x,y)\in S_{x,y}$ for any $y\in U'_{a_0}\cap O_x$,
and thus $D(x,y)\le D_{a_0}(x,y)$. Since 
$$
x\in P_z\Longrightarrow P_z\subset O_x
$$
for all $x\in Z$ by Claim~\ref{cl:Pz}, it follows that $D(x,y)\le
D_{a_0}(x,y)$ for all $x,y\in P_z$.

On the other hand, we get from Claim~\ref{cl:D(x,y)} that
$D(x,y)\ge D_{a_0}(x,y)$ for all $x\in U'_{a_0}$ and all $y$ in
the open ball in $(U_{a_0},D_{a_0})$ of center $x$ and
radius $\rho(x)=D_{a_0}(x,U_{a_0}\setminus U'_{a_0})$. So $D(x,y)\ge
D_{a_0}(x,y)$ for all $x,y$ in
the open ball in $(U_{a_0},D_{a_0})$ of center $z$ and
radius $\frac{1}{2}\,\rho(z)$. Therefore the metrics
$D,D_{a_0}$ are equal on some neighborhood of $x$, and the
result follows.
\end{proof}

\begin{rems}
  (i) Theorem~\ref{t:local metric} is very similar to the Smirnov
  metrization theorem \cite{Smirnov}, \cite[pp.~260--261]{munkres}
  (see also J.~Nagata~\cite[Chapter VI.3]{nagata} for a
  stronger result), which shows that a topological space is metrizable
  if and only if it is Hausdorff, paracompact and locally metrizable:
  in Theorem~\ref{t:local metric}, the existence of a local metric is
  slightly stronger than local metrizability, and the existence of a
  metric that induces a given local metric is slightly stronger than
  metrizability.  
%
%
%
  \newline 
(ii) By the proof of Theorem~\ref{t:local metric}, any
  paracompact Hausdorff local metric $\mathfrak D$ can be considered
  as the germ of some metric on $Z$ around the diagonal of $Z\times
  Z$. But even in this case, the introduction of local metrics makes
  sense to emphasize the fact that we are only considering distances
  between ``very close'' points.  
\newline 
(iii) With the sheaf theoretic point of view given in the remark~(iii) of
  Definition~\ref{d:local metric}, even though $\MM$ is never a sheaf
  for interesting spaces, it is closer to be so for Hausdorff
  paracompact spaces: in this case, Theorem~\ref{t:local metric}
  asserts that the canonical homomorphism of presheaves, ${\mathcal
    M}\to\widetilde{\mathcal M}$, is surjective on all open sets.
\end{rems}

\begin{example}\label{ex:half-disk topology}
Let $P$ be the open upper half-plane
$\{(x,y)\in\R^2\ |\ y>0\}$, and $L$ the real axis $\{(x,0)\ |\
x\in\R\}$. Consider the half-disk topology on $Z=P\cup L$
\cite[pp.~96--97]{Counterexamples}, which has a base given
by the euclidean open sets in $P$ and the sets of the form
$\{z\}\cup(P\cap U)$, where $z\in L$ and $U$ is any euclidean open
neighborhood of $z$ in $\R^2$. This space is not metrizable
because it is not paracompact. But this topology is induced by a
local metric $\mathfrak D$ on $Z$, which is determined by the family
$$
\{(P,d_P)\}\cup\{(U_z,d_z)\ |\ z\in L\}\;,
$$
where $d_P$ is the restriction of the euclidean
metric to $P$, $U_z=\{z\}\cup P$, and $d_z$ is the restriction of the
euclidean metric to $U_z$. 
\end{example}

\begin{example}
With more generality, let $(Z,d)$ be a metric
space, let $\{Z_i\}_{i\in I}$ be a covering of $Z$, and let $d_i$ be the
restriction of $d$ to $Z_i$ for each $i\in I$. Then the metrics
$d_i,d_j$ have equal restriction to the overlap $Z_i\cap Z_j$ for all
$i,j\in I$, and thus the family $\{(Z_i,d_i)\}_{i\in I}$ defines a
local metric $\mathfrak D$ on $Z$. If the sets $Z_i$ are open in
$(Z,d)$, then $\mathfrak D$ is induced by the metric $d$, otherwise the
topology induced by $\mathfrak D$ is strictly finer than the topology
induced by $d$, and $\mathfrak D$ may not be induced by any metric, as
in Example~\ref{ex:half-disk topology}.
\end{example}

Even though we are only interested on paracompact Hausdorff spaces,
the following problem is interesting.

\begin{problem}
Is any locally metrizable topology induced by some local metric? In
particular, is there a compatible local metric on every non-paracompact
manifold? For instance, is there a compatible local metric on the Long
Line \cite[pp.~71--72]{Counterexamples}?
\end{problem}

\section{Pseudogroups of local isometries}\label{sec:local isometries}

The idea of a local metric as measuring distances between ``very
close'' points is specially appropriate to define local isometries.

\begin{defn}\label{d:local isometry}
Let $(Z,{\mathfrak D})$ be a local metric space, and
let $h$ be a homeomorphism between open subsets of $(Z,{\mathfrak D})$.
Then $h$ is called a {\em local isometry\/} of $(Z,{\mathfrak D})$ if
there is some $\{(Z_i,d_i)\}_{i\in I}\in{\mathfrak D}$ such that, for
$i,j\in I$ and $z\in Z_i\cap h^{-1}(Z_j\cap\im h)$, there is some
neighborhood $U_{h,i,j,z}$ of $z$ in $Z_i\cap h^{-1}(Z_j\cap\im h)$
so that $d_i(x,y)=d_j(h(x),h(y))$ for all $x,y\in U_{h,i,j,z}$.
\end{defn}

\begin{rems}
(i) For a map $h$ between open subsets of a local metric space
$(Z,{\mathfrak D})$, the property of being a local
isometry is completely local, and $h$ may not be isometric
for a given metric inducing $\mathfrak D$ (Examples~\ref{ex:S neq HH}
and~\ref{ex:Z, HH}).
\newline
(ii) About the condition that the
metrics $d_i,d_j$ are locally equal on $Z_i\cap Z_j$ for any
$\{(Z_i,d_i)\}_{i\in I}\in{\mathfrak D}$, it just means that the
identity map on any open subset of $(Z,{\mathfrak D})$ is a local
isometry.
\newline
(iii) A homeomorphism $h$ between open subsets of a local metric space
$(Z,{\mathfrak D})$ is a local isometry when it preserves the local
metric in the obvious sense: $h^*({\mathfrak D}_{|\im h})={\mathfrak
D}_{|\dom h}$, where the restrictions and pull-backs of
local metrics are defined in an obvious way. With the sheaf theoretic 
description of local metrics
(remark~(iii) of Definition~\ref{d:local metric}), this means that $h$
induces an isomorphism between the restrictions of $\widetilde\MM$ to
its domain and image.
\newline
(iv) The definition of local isometry is completely independent of the
choice of the family $\{(Z_i,d_i)\}_{i\in I}\in\mathfrak D$. So the same
$\{(Z_i,d_i)\}_{i\in I}$ can be chosen to verify Definition~\ref{d:local
isometry} for any family of local isometries. Therefore the concept of
pseudogroup of local isometries is completely
analogous to the concept of weakly equicontinuous pseudogroup.
\end{rems}

\begin{example}\label{ex:Z, HH}
On the local metric space $(Z,{\mathfrak D})$ of Example~\ref{ex:Z},
let $\HH$ be the restriction of the pseudogroup generated by all
translations of $\R^2$. Then
$\HH$ is a pseudogroup of local isometries of $(Z,{\mathfrak D})$. The
maps in $\HH$ with connected domain are isometries with respect to
$d$, but many of them are not isometries with respect to $d'$. For
instance, let $U$ be any relatively compact and connected open subset
of $Z$ containing points $x,y$ with $[x,y]\cap
B\neq\emptyset$. Then there is a translation $h$ of
$\R^2$ such that $h(U)\subset Z$ and $[h(x),h(y)]\cap
B=\emptyset$. So
$$
d'(h(x),h(y))=d(h(x),h(y))=d(x,y)<d'(x,y)\;,
$$
and thus the restriction $h:U\to h(U)$ is an element of $\HH$ with
connected domain that does not preserve $d'$. 
\end{example}

Arguments similar to those used in the proof of Lemma~\ref{l:weakly
  equicontinuous} prove the following lemma.

\begin{lemma}\label{l:pseudogroup of local isometries-1}
Let $\HH,\HH'$ be equivalent pseudogroups on 
spaces $Z,Z'$. Then $\HH$ is a pseudogroup of local isometries for some
local metric inducing the topology of $Z$ if and only if $\HH'$ is a
pseudogroup of local isometries for some local metric inducing the
topology of $Z'$.
\end{lemma}

Unlike the concept of equicontinuity, it is not necessary to
introduce weak and strong versions of the concept of pseudogroup
of local isometries by the following result.
%
%

\begin{lemma}\label{l:pseudogroup of local isometries-2}
Let $\HH$ be a pseudogroup of local transformations of a paracompact
local metric space $(Z,{\mathfrak D})$. Then $\HH$ is a pseudogroup of
local isometries of $(Z,{\mathfrak D})$ if and only if there is some
$\{(Z_i,d_i)\}_{i\in I}\in{\mathfrak D}$ and some symmetric set $S$ of
generators of $\HH$ that is closed under compositions and such that
$d_i(x,y)=d_j(h(x),h(y))$
for all $h\in S$, $i,j\in I$ and $x,y\in Z_i\cap h^{-1}(Z_j\cap\im h)$.
\end{lemma}

%

\begin{proof}
Take any $\{(Z_i,d_i)\}_{i\in I}\in{\mathfrak D}$ such that
$\{Z_i\}_{i\in I}$ is locally finite. With the notation of
Lemma~\ref{l:locally finite local metrics} and
Definition~\ref{d:local isometry}, for each $h\in\HH$ and $z\in\dom h$,
let 
$$
U_{h,z}=U_z\cap\bigcap_{i,j\in I,\ z\in Z_i\cap Z_j}U_{h,i,j,z}\;,
$$
which is an open neighborhood of $z$. Then the result holds with $S$
equal to the set of compositions of all restrictions of the form
$h:U_{h,z}\to h(U_{h,z})$ and their inverses. We have used that
composition of isometries is an isometry, which fails for the
equicontinuous condition~\eqref{e:equicontinuous} with a fixed
assignment $\epsilon\mapsto\delta(\epsilon)$.
\end{proof}

\section{Isometrization of strongly equicontinuous pseudogroups}

On the type of spaces we are considering, it will be
shown that compactly generated quasi-effective strongly equicontinuous
pseudogroups are pseudogroups of local isometries for some local
metric. We begin 
with the following version of Theorem~\ref{t:local metric} for quasi-local
metric spaces. Most of its proof is also similar to the proof of
Theorem~\ref{t:local metric}, but there are some new difficulties.

\begin{theorem}\label{t:isometrization of quasi-local metrics}
A quasi-local metric space $(Z,{\mathfrak Q})$ is induced by some metric on
$Z$ if and only if $(Z,{\mathfrak Q})$ is Hausdorff and paracompact.
\end{theorem}

\begin{proof}
As in the proof of Theorem~\ref{t:local metric}, the ``only if'' part
holds by the Stone theorem. Now suppose that
$(Z,{\mathfrak Q})$ is Hausdorff and paracompact. The following
assertion can be proved in the same way as Claim~\ref{cl:Ua,Da}.

\begin{claim}\label{cl:Ua,Da; equicontinuous}
There is some $\{(U_a,D_a)\}_{a\in A}\in{\mathfrak Q}$ and some
$\delta(\epsilon)>0$ for each $\epsilon>0$ such that
$\{U_a\}_{a\in A}$ is locally finite, and
$$
D_a(x,y)<\delta(\epsilon)\Longrightarrow D_b(x,y)<\epsilon
$$
for all $\epsilon>0$, $a,b\in A$ and $x,y\in U_a\cap U_b$.
\end{claim}

We can also assume that the family
$\{(U_a,D_a)\}_{a\in A}$ given by Claim~\ref{cl:Ua,Da} satisfies that
the $D_a$-diameter of each $U_a$ is smaller than
$1$. If $x,y\in U_{a_0}$ for some $a_0\in A$,
let
$$
\overline{D}(x,y)=\sup_{a\in A,\ x,y\in U_a}D_a(x,y)\;.
$$
Note that $\overline{D}(x,y)\le1$ by the condition on the
$D_a$-diameter of each $U_a$, and that $\overline{D}(x,y)$ is independent
of $a_0$. Moreover
$\overline{D}$ is obviously symmetric, we have
$\overline{D}(x,y)=0$ if and only if
$x=y$, and the following assertion
follows directly from
Claim~\ref{cl:Ua,Da; equicontinuous}.

\begin{claim}\label{cl:overline D}
We have
$$
D_a(x,y)<\delta(\epsilon)\Longrightarrow
\overline{D}(x,y)<\epsilon
$$
for all $a\in A$ and $x,y\in U_a$.
\end{claim}

But $\overline{D}$ may not satisfy the triangle inequality
on any open set because there may be points $x,y,z\in U_{a_0}$ so that
$x,y\in U_a$ and $z\not\in U_a$ for some $a\in A$. So $\overline{D}$ may 
not be a metric on the sets of any open covering of
$Z$; otherwise, Theorem~\ref{t:local metric} could be used to conclude.
Yet $\overline{D}$ is used to define a metric on $Z$ with the idea of the
proof of Theorem~\ref{t:local metric}. 

Let $\{U'_a\}_{a\in A}$ be a shrinking of
$\{U_a\}_{a\in A}$. A pair
$(z_1,z_2)\in Z\times Z$ will be said to be admissible if there is some
$a\in A$ such that $z_1,z_2\in U'_a$, and moreover 
$$
\{z_1,z_2\}\cap U'_b\neq\emptyset\Longrightarrow \{z_1,z_2\}\subset U_b
$$
for any $b\in A$. 
For each 
$(x,y)\in Z\times Z$, let $S_{x,y}$ denote the set of all finite
sequences $(z_0,\dots,z_n)$ in $Z$, with arbitrary length $n\in\Z_+$,
such that $z_0=x$, $z_n=y$, and $(z_{k-1},z_k)$ is an admissible pair
for every $k=1,\dots,n$.
Now set
$D(x,y)=1$ if $S_{x,y}=\emptyset$, and 
$$
D(x,y)=
\inf_{(z_0,\dots,z_n)\in S_{x,y}}\sum_{k=1}^n\overline{D}(z_{k-1},z_k)
$$
if $S_{x,y}\neq\emptyset$.

\begin{claim}\label{cl:D(x,y); equicontinuous}
Let $a\in A$, $x\in U'_a$ and $y\in Z$ with $S_{x,y}\neq\emptyset$. Then
$$
D(x,y)\ge
\begin{cases}
\min\{D_a(x,y),D_a(x,U_a\setminus U'_a)\}&\text{if $y\in U'_a$,}\\
D_a(x,U_a\setminus U'_a)&\text{if $y\not\in U'_a$.}
\end{cases}
$$
\end{claim}

To prove this assertion, let
$(z_0,\dots,z_n)\in S_{x,y}$ and $a_1,\dots,a_k\in A$ with
$z_{k-1},z_k\in U'_{a_k}$ for $k=1,\dots,n$.
On the one hand, if
$z_0,\dots,z_n\in U'_a$, we have
$$
\sum_{k=1}^n\overline{D}(z_{k-1},z_k)\ge
\sum_{k=1}^nD_a(z_{k-1},z_k)
\ge D_a(z_0,z_n)=D_a(x,y)\;.
$$
On the other hand, suppose
$\{z_0,\dots,z_n\}\not\subset U'_a$. Then
$n\ge1$, and let
$$
n_0=\min\{k\in\{1,\dots,n\}\ |\ z_k\not\in U'_a\}\;.
$$
Since $z_{n_0-1}\in U'_a$, we get
$z_{n_0}\in U_a$ because $(z_{n_0-1},z_{n_0})$ is an admissible pair.
So
$$
\sum_{k=1}^n\overline{D}(z_{k-1},z_k)
\ge\sum_{k=1}^{n_0}D_a(z_{k-1},z_k)
\ge D_a(z_0,z_{n_0})\ge D_a(x,U_a\setminus U'_a)\;,
$$
which completes the proof of
Claim~\ref{cl:D(x,y)}.

With the same arguments as in the proof of Theorem~\ref{t:local
metric}, it follows that $D$ is a metric on $Z$ by using
Claim~\ref{cl:D(x,y); equicontinuous}.

It remains to check that $\{(Z,D)\}\in{\mathfrak Q}$. Fix any $z\in Z$
and any $a_0\in A$ with $z\in U'_{a_0}$. We get the following assertion
as in the proof of Theorem~\ref{t:local metric}.

\begin{claim}\label{cl:Pz; equicontinuous}
There is some open neighborhood $P_z$ of $z$ in
$U'_{a_0}$ such that
$$
P_z\cap U'_a\neq\emptyset\Longrightarrow P_z\subset U_a
$$
for all $a,b\in A$.
\end{claim}

Also, as in the proof of Theorem~\ref{t:local metric}, the set
$$
O_x=\bigcap_{x\in U'_a,\ a\in A}U_a
\setminus\bigcup_{x\not\in U_b,\ b\in A}\overline{U'_b}
$$
is an open neighborhood of every $x$ in $Z$, and we have $(x,y)\in
S_{x,y}$ for any $x\in U'_{a_0}$ and $y\in
U'_{a_0}\cap O_x$. So
$D(x,y)\le\overline{D}(x,y)$ for all $y\in U'_{a_0}\cap O_x$, yielding
\begin{equation}\label{e:Da0}
D_{a_0}(x,y)<\delta(\epsilon)\Longrightarrow D(x,y)<\epsilon
\end{equation}
by Claim~\ref{cl:overline D}. Since 
$$
x\in P_z\Longrightarrow P_z\subset O_x
$$
for all $x\in Z$ by Claim~\ref{cl:Pz}, it follows that~\eqref{e:Da0}
holds for all $x,y\in P_z$.

On the other hand, as in the proof of Theorem~\ref{t:local metric}, we
get from Claim~\ref{cl:D(x,y); equicontinuous} that $D(x,y)\ge
D_{a_0}(x,y)$ for all $x,y$ in the open ball in $(U_{a_0},D_{a_0})$ of
center $z$ and radius $\frac{1}{2}\,D_{a_0}(z,U_{a_0}\setminus
U'_{a_0})$. Therefore the families of metric spaces $\{(Z,D)\}$ and
$\{(U_a,D_a)\}_{a\in A}$ are quasi-locally equal;  {\it i.e.}, $\mathfrak
Q$ is induced by $D$.
\end{proof}

\begin{rems}
(i) Theorem~\ref{t:isometrization of quasi-local metrics} can be also
compared with 
the Smirnov metrization theorem.
\newline
(ii) By Theorem~\ref{t:isometrization of quasi-local metrics}, in the 
paracompact Hausdorff case, a quasi-local metric is almost the same concept
as a local metric; the only different being that different local metrics
may induce the same quasi-local metric (Example~\ref{ex:quasi-local metric
on R2}).
\end{rems}

Our ``isometrization'' result for pseudogroups can be stated as follows.

\begin{theorem}\label{t:isometrization}
Let $\HH$ be a compactly generated, quasi-effective and strongly
equicontinuous pseudogroup of local transformations of a locally compact
Polish space $Z$. Then $\HH$ is a pseudogroup of local isometries with
respect to some local metric inducing the topology of $Z$.
\end{theorem}

\begin{proof}
The pseudogroup $\HH$ is strongly equicontinuous with respect to some
quasi-local metric $\mathfrak Q$ that induces the topology of $Z$.
Such a $\mathfrak Q$ is induced by some metric $d$ on $Z$ according to
Theorem~\ref{t:isometrization of quasi-local metrics}. So, by
remark~(iv) of Definition~\ref{d:strongly equicontinuous}, the
condition of strong equicontinuity is satisfied by the family
$\{(Z,d)\}$ with some assignment $\epsilon\mapsto\delta(\epsilon)$ and
some symmetric set $S$ of generators of $\HH$ that is closed under
compositions. We can also suppose that $S$ is closed under
restrictions to open sets by remark~(iii) of
Definition~\ref{d:strongly equicontinuous}. Furthermore we can assume
that the condition of quasi-effectiveness is also satisfied with $S$
(remarks of Definition~\ref{d:quasi-effective}). This means that any
element of $S$ is equal to the identity on its domain if it is equal
to the identity on some non-trivial open subset; so two elements of
$S$ are equal on the intersection of their domains if they have the
same germ at some point.

Let $U$ be any relatively compact open subset of $Z$ that meets every
$\HH$-orbit, and $E$ any symmetric system of compact generation of $\HH$
on
$U$. For each $g\in E$, let
$\bar g$ be its extension satisfying the conditions of
Definition~\ref{d:compactly generated}, and let
$\overline{E}=\{\bar g\ |\ g\in E\}$. We can choose $S$, $E$
and the extensions $\bar g$ so that 
$\overline{E}\subset S$. 

Let $\VV$ be a finite family of open subsets of $Z$ given by
Proposition~\ref{p:equicontinuous} for the above $d$, $S$, $U$, $E$ and
extensions $\bar g$. We can suppose that the $d$-diameter of every
$V\in V$ is smaller than $\delta(1)$. Let $R\subset\HH$ be
the set of all compositions of elements in $E$, and
$\overline{R}\subset\HH$ the set of all compositions of elements in
$\overline{E}$; so $R,\overline{R}\subset S$. For each
$V\in\VV$, and
$x,y\in V$, let
$$
d_V(x,y)=\sup_{h\in\overline{R},\ V\subset\dom h} d(h(x),h(y))\;,
$$
Such $d_V$ is well defined by
Proposition~\ref{p:equicontinuous}, and we have
$d_V(x,y)\le1$ by the condition on the diameter of $V$ and because
$\overline{R}\subset S$. It is easy to check that $d_V$ is a
metric on $V$. Moreover we have the
following fact.

\begin{claim}\label{cl:dV}
The metrics $d_V,d_W$ are equal on $V\cap W$ for all $V,W\in\VV$.
\end{claim}

Take sets $V,W\in\WW$ with $V\cap W\neq\emptyset$ to verify this
assertion. It suffices to show that, for all $h\in\overline{R}$ whose
domain contains $V$, there is some
$h'\in\overline{R}$ whose domain contains $W$ and so that $h,h'$ are
equal on $V\cap W$: this clearly yields $d_V(x,y)\le d_W(x,y)$ for all
$x,y\in V\cap W$, and the reverse inequality is similarly obtained.
Thus let $h\in\overline{R}$ with $V\subset\dom h$. The germ of $h$ at
any $x\in V\cap W$ is equal to the germ of some $f\in R$ at $x$. By
Proposition~\ref{p:equicontinuous}, there is some $h'\in\overline{R}$
whose domain contains $W$ and equal to $f$ around $x$. Since $h,h'\in S$
and have the same germ at $x$, these maps are equal on $V\cap W$
by quasi-effectiveness, and the claim follows.

Therefore the collection $\{(V,d_V)\ |\ V\in\VV\}$
defines a local metric ${\mathfrak D}_0$ on the union $U_0$ of the sets
$V\in\VV$. Moreover, on the one hand, we obviously have $d_V(x,y)\ge
d(x,y)$ for all $V\in\VV$ and $x,y\in V$. On the other hand,
$$
d(x,y)<\delta(\epsilon)\Longrightarrow d_V(x,y)<\epsilon
$$
for all $\epsilon>0$, $V\in\VV$ and $x,y\in V$ by strong
equicontinuity since $\overline{R}\subset S$. Thus ${\mathfrak D}_0$
induces the restriction ${\mathfrak Q}_0$ of $\mathfrak Q$ to
$U_0$. 

\begin{claim}\label{cl:f}
We have 
$$
d_W(f(x),f(y))=d_V(x,y)
$$ 
for all $V,W\in\VV$,
$f\in R$ and $x,y\in V\cap f^{-1}(W\cap\im h)$.
\end{claim}

To prove this equality, let $V,W,f,x,y$ be as in the statement of this
claim. Then we have $f=g_m\circ\dots\circ g_1$ for $g_1,\dots,g_m\in E$. Let
$\tilde f=\bar g_m\circ\dots\circ\bar g_1\in\overline{R}$. Then
$V\subset\dom\tilde f$ by Proposition~\ref{p:equicontinuous}. For any
$h\in\overline{R}$ with
$W\subset\dom h$, the germ of $h$ at any fixed point 
$z\in W\cap\im f$ is equal to the germ at $z$ of some element of $R$;
say $g_{m+n}\circ\dots\circ g_{m+1}$ for some 
$g_{m+n},\dots,g_{m+1}\in E$. Hence $\bar g_{m+n}\circ\dots\circ\bar
g_{m+1}\in\overline{R}$ has the same germ at $z$ as $h$ and its domain
contains $W$ again by Proposition~\ref{p:equicontinuous}. It follows
that $h=\bar g_{m+n}\circ\dots\circ\bar g_{m+1}$ on $W$ by
quasi-effectiveness. Since $g_{m+n}\circ\dots\circ g_1$is
defined around $f^{-1}(z)\in V$, the domain of $\bar
g_{m+n}\circ\dots\circ\bar g_1$ contains $V$ by
Proposition~\ref{p:equicontinuous} once more, and we have
$h\circ\tilde f=\bar g_{m+n}\circ\dots\circ\bar g_1$ on $V\cap\tilde
f^{-1}(W\cap\im\tilde f)$ by quasi-effectiveness. So $h\circ f$ is
equal to some element of $\overline{R}$ on $V\cap f^{-1}(W\cap\im f)$,
which yields
\begin{align*}
d_W(f(x),f(y))
&=\sup_{h\in\overline{R},\ W\subset\dom h} 
d(h\circ f(x),h\circ f(y))\\
&\le\sup_{h'\in\overline{R},\ V\subset\dom h'} 
d(h'(x),h'(y))\\
&=d_V(x,y)\;.
\end{align*}
We also get $d_W(f(x),f(y))\ge d_V(x,y)$ by applying the above argument
to $f^{-1}$ instead of $f$, and Claim~\ref{cl:f} follows.

Claim~\ref{cl:f} shows that the restriction of $\HH$ to $U$ is a
pseudogroup of local isometries with respect to the restriction of
${\mathfrak D}_0$ to $U$, and therefore the theorem follows by
Lemma~\ref{l:pseudogroup of local isometries-1}.
\end{proof}

According to Theorem~\ref{t:isometrization}, the following problem may be
difficult and interesting.

\begin{problem}
Find an example of a strongly equicontinuous pseudogroup that is not a
pseudogroup of local isometries for any local metric.
\end{problem}

\section{A non-standard description of weak equicontinuity}

The following simple non-standard description of weak equicontinuity
shows the naturality of this condition, even though strong
equicontinuity is what is mainly used in our study. 
 The reference for
non-standard analysis is Robinson~\cite{robinson}; we do not use any
technique particular to non-standard analysis, only the concept of
monad, which is now defined.

%
%

Fix a non-principal ultrafilter $\FF$ on the set $\N$ of positive
integers; {\it i.e.}, $\FF$ defines a point in the corona of the
Stone-\v{C}ech compactification of $\N$. Let $(Z,d)$ be any metric
space. For any $x\in Z$, the {\em monad\/} of $x$ in $(Z,d)$, denoted
by $M(x,Z,d)$ or simply $M(x)$, is the quotient set of the set of
sequences $y_n$ in $Z$ such that
$$
\{n\in\N\ |\ d(x,y_n)<r\}\in\FF
$$ 
for all $r>0$, where two such sequences
$y_n,z_n$ are identified when 
$$
\{n\in\N\ |\ y_n=z_n\}\in\FF\;.
$$ 
If $(Z',d')$ is another metric space, any continuous map $f:(Z,d)\to(Z',d')$
induces a map $f_*:M(x,Z,d)\to M(f(x),Z',d')$ for every $x\in Z$, which is
defined as follows: if ${\mathbf y}\in M(x,Z,d)$ is represented by the sequence
$y_n$, then $f_*({\mathbf y})$ is represented by the sequence $f(y_n)$.

The monad of $0$ in $\R$ with the euclidean metric is the set $\mathbf I$
of infinitesimal numbers.  The infinitesimal number represented by the zero
constant sequence will be denoted by $\boldsymbol0$.
For ${\boldsymbol\epsilon},{\boldsymbol\delta}\in{\mathbf I}$, represented by
sequences $\epsilon_n,\delta_n$, the inequality
${\boldsymbol\epsilon}<{\boldsymbol\delta}$ means that
$$
\{n\in\N\ |\ \epsilon_n<\delta_n\}\in\FF\;.
$$
Moreover the metric $d$ on $Z$ defines a map
$d_*:M(x)\to{\mathbf I}$ for every $x\in Z$ in the following way: if ${\mathbf
y}\in M(x)$ is represented by the sequence
$y_n$, then $d_*({\mathbf y})$ is represented by the
sequence $d(x,y_n)$.

Now suppose that $\mathfrak Q$ is a quasi-local metric on $Z$ and
$\{(Z_i,d_i)\}_{i\in I}\in{\mathfrak Q}$. If $x\in Z_i\cap Z_j$ for $i,j\in I$,
then $M(x,Z_i,d_i)\equiv M(x,Z_j,d_j)$ canonically. Thus
$M(x,Z_i,d_i)$ can be called the monad of $x$ in $(Z,{\mathfrak Q})$, and
denoted by $M(x,Z,{\mathfrak Q})$ or simply $M(x)$. It also follows that any
continuous map between quasi-local metric spaces, $f:(Z,{\mathfrak
Q})\to(Z',{\mathfrak Q}')$, induces a map $f_*:M(x,Z,{\mathfrak Q})\to
M(f(x),Z',{\mathfrak Q}')$ for each
$x\in Z$.

\begin{theorem}\label{t:monada}
Let $\HH$ be a pseudogroup of local homeomorphisms of a quasi-local
metric space $(Z,{\mathfrak Q})$, and let $\{(Z_i,d_i)\}_{i\in
I}\in{\mathfrak Q}$. Then $\HH$ is weakly equicontinuous if and only
if for every ${\boldsymbol\epsilon}\in{\mathbf I}$,
${\boldsymbol\epsilon}>{\boldsymbol0}$, there is some
${\boldsymbol\delta}({\boldsymbol\epsilon})\in{\mathbf I}$,
${\boldsymbol\delta}({\boldsymbol\epsilon})>{\boldsymbol0}$, such that
\begin{equation}\label{e:monada}
d_{i*}({\mathbf y})<{\boldsymbol\delta}({\boldsymbol\epsilon})
\Longrightarrow 
d_{j*}(h_*({\mathbf y}))<{\boldsymbol\epsilon}
\end{equation}
for all $h\in\HH$, $i,j\in I$, $x\in Z_i\cap h^{-1}(Z_j\cap\im h)$ and 
${\mathbf y}\in M(x)$.
\end{theorem}

\begin{proof}
Suppose first that $\HH$ is weakly equicontinuous. So the condition of weak
equicontinuity is satisfied with $\{(Z_i,d_i)\}_{i\in I}$, some
assignment $\epsilon\mapsto\delta(\epsilon)$ and neighborhoods $U_{h,i,j,z}$
(Definition~\ref{d:weakly equicontinuous}). We can assume that
$\delta(\epsilon)<\epsilon$ for all $\epsilon>0$. Given any
${\boldsymbol\epsilon}\in{\mathbf I}$,
${\boldsymbol\epsilon}>{\boldsymbol0}$, take some sequence
$\epsilon_n$ representing ${\boldsymbol\epsilon}$. We can assume that
$\epsilon_n>0$ for all $n$. Then the sequence
$\delta(\epsilon_n)$ also represents some infinitesimal number, which is
denoted by ${\boldsymbol\delta}({\boldsymbol\epsilon})$. Now take $h\in\HH$, 
$i,j\in I$, $x\in Z_i\cap h^{-1}(Z_j\cap\im h)$ and 
${\mathbf y}\in M(x)$ with $d_{i*}({\mathbf
y})<{\boldsymbol\delta}({\boldsymbol\epsilon})$. So
$$
\{n\in\N\ |\ d_i(x,y_n)<\delta(\epsilon_n)\}\in\FF\;.
$$
Moreover
$$
\{n\in\N\ |\ y_n\in U_{h,i,j,x}\}\in\FF
$$
because ${\mathbf y}\in M(x)$. Therefore
$$
\{n\in\N\ |\ y_n\in Z_i\cap h^{-1}(Z_j\cap\im h),\
d_j(h(x),h(y_n))<\delta(\epsilon_n)\}\in\FF
$$
by weak equicontinuity, yielding 
$d_{j*}(h_*({\mathbf y}))<{\boldsymbol\epsilon}$,
and~\eqref{e:monada} follows.

Now suppose that~\eqref{e:monada} holds for some assignment
${\boldsymbol\epsilon}\mapsto{\boldsymbol\delta}({\boldsymbol\epsilon})$
and all $h,i,j,x,{\mathbf y}$ as in the statement. According to
Definition~\ref{d:weakly equicontinuous}, if $\HH$ is not weakly
equicontinuous, then there exists some $\epsilon>0$, $h\in\HH$,
$i,j\in I$ and $z\in Z_i\cap h^{-1}(Z_j\cap\im h)$ so that, in every
neighborhood $U$ of $z$ in $Z_i\cap h^{-1}(Z_j\cap\im h)$, there are
points $x_U,y_U$ with $d_j(h(x_U),h(y_U))\ge\epsilon$.  So, for every
$n\in\NN$, there are points $x_n,y_n\in Z_i\cap h^{-1}(Z_j\cap\im h)$
with
$$
d_i(x_n,z)\,,\,d_i(y_n,z)<1/n\;,\quad d_j(x_n,y_n)\ge\epsilon\;.
$$
On the other hand, there is some $N\in\N$ so that
$d_j(h(z),h(x_n))<\epsilon/2$ for $n\ge N$. Thus
\begin{equation}\label{e:>epsilon/2}
d_j(h(z),h(y_n))\ge d_j(h(x_n),h(y_n))-d_j(h(z),h(x_n))>\epsilon/2
\end{equation}
for $n\ge N$. Given any ${\boldsymbol\epsilon}\in{\mathbf I}$,
${\boldsymbol\epsilon}>{\boldsymbol0}$, let
${\boldsymbol\delta}({\boldsymbol\epsilon})$ be represented by a sequence
$\delta_n$. We can suppose that $\delta_n>0$ for all $n$. Then there
is some $k_n\ge N$ with $1/k_n<\delta_n$ for every $n$, the sequence
$y'_n=y_{k_n}$ represents an element ${\mathbf y'}\in M(z)$, and we have
$d_{i*}({\mathbf y'})<{\boldsymbol\delta}({\boldsymbol\epsilon})$. So
$d_{j*}(h_*({\mathbf y'}))<{\boldsymbol\epsilon}$, which
contradicts~\eqref{e:>epsilon/2}.
\end{proof}

\section{Strongly equicontinuous foliated spaces}

Let $(X,\FF)$ be a compact foliated space.  Compact generation and
recurrence are properties satisfied by its holonomy pseudogroup with
the generators given by a finite defining cocycle. To see this, fix a
finite defining cocycle $(U_i,p_i,h_{i,j})$ of $\FF$ with $p_i:U_i\to
Z_i$ and $h_{i,j}:Z_{i,j}\to Z_{j,i}$, where $Z_{i,j}=p_i(U_i\cap
U_j)$. Let $\HH$ be the representative of the holonomy pseudogroup of
$\FF$ induced by $(U_i,p_i,h_{i,j})$ on $Z=\bigsqcup_i Z_i$.  Suppose
that $(U_i,p_i,h_{i,j})$ is a shrinking of another defining cocycle
$\left(\widetilde U_i,\tilde p_i,\tilde h_{i,j}\right)$ with $\tilde
p_i:\widetilde U_i\to\widetilde Z_i$ and $\tilde h_{i,j}:\widetilde
Z_{i,j}\to\widetilde Z_{j,i}$, where $\widetilde Z_{i,j}=\tilde
p_i\left(\widetilde U_i\cap\widetilde U_j\right)$. This means that,
for each $i$, $\overline{U_i}\subset\widetilde U_i$, $Z_i=\tilde
p_i(U_i)$, and $p_i$ is the restriction of $\tilde p_i$. Thus
$\overline{Z_i}\subset\widetilde Z_i$,
$\overline{Z_{i,j}}\subset\widetilde Z_{i,j}$, and $h_{i,j}$ is the
restriction of $\tilde h_{i,j}$. Let $\widetilde\HH$ be the
representative of the holonomy pseudogroup of $\FF$ induced by
$\left(\widetilde U_i,\tilde p_i,\tilde h_{i,j}\right)$ on $\widetilde
Z=\bigsqcup_i\widetilde Z_i$. It easily follows that $Z$ is relatively
compact in $\widetilde Z$, $\HH$ is the restriction of $\widetilde\HH$
to $Z$, $Z$ meets all $\widetilde\HH$-orbits, and the transformations
$h_{i,j}$ form a system of compact generation of $\widetilde\HH$ on
$Z$. This system is proved to be recurrent as follows. Fix any point
$x$ in the closure of some $Z_i$ in $\widetilde Z_i$. Then $\tilde
p_i^{-1}(x)$ cuts some $\widetilde U_j$. Thus $V_{i,j}=\tilde
p_i\left(\widetilde U_i\cap U_j\right)$ is a neighborhood of $x$ in
$\widetilde Z_i$ such that $V_{i,j}\subset\widetilde Z_{i,j}$,
$V_{i,j}\cap Z_i=Z_{i,j}$ and $\tilde h_{i,j}(V_{i,j})\subset Z_j$.
Hence the transformations $h_{i,j}$ form a recurrent system of compact
generation by Lemma~\ref{l:recurrent systems of compact generation}.
Therefore, according to Theorem~\ref{t:coarsely quasi-isometric
  orbits}, the coarse quasi-isometry type of the $\HH$-orbits with the
metric induced by the generators $h_{i,j}$ is kept uniformly fixed
when varying the defining cocycle. So each leaf of $\FF$ determines a
coarse quasi-isometry type of metric spaces.  See {\it e.g.\/}
\cite{Hector-Hirsch} for a more detailed description of the coarse
quasi-isometry between the leaves and the orbits of the holonomy
pseudogroup, which is already implicit in \cite{Plante}. The bornotopy
type of leaves is studied in \cite{Hurder94}, and most of its
discussion also applies to the coarse quasi-isometry type.

If $\FF$ is at least of class $C^1$, then there exists a
metric tensor on the leaves that is continuous on $X$. In this case, it
is well known that the already explained embedding of $Z$ into $X$, as a
complete transversal of $\FF$, defines a uniform collection of
coarse quasi-isometries between the $\HH$-orbits and the corresponding leaves.
This is just a consequence of having a uniform upper bound of the
diameter of the plaques for the given finite defining cocycle.
Therefore the coarse quasi-isometry type determined by each leaf is
given by just itself with such a  metric tensor.

The foliated space $(X,\FF)$ will be called {\em weakly
  equicontinuous\/}, {\em strongly equicontinuous\/}, {\em
  quasi-analytic\/}, or {\em quasi-effective\/}, respectively, if any
representative of its holonomy pseudogroup is such. This is well
defined because all of these properties on pseudogroups are invariant
by equivalences (Lemmas~\ref{l:weakly equicontinuous},~\ref{l:strongly
  equicontinuous} and~\ref{l:quasi-effective}). Then
Theorem~\ref{t:equicontinuous pseudogroup} has the following
consequence.

\begin{theorem}\label{t:equicontinuous foliated space}
  Let $(X,\FF)$ be an equicontinuous, compact and quasi-effective
  foliated space. Assume that the space of leaves is connected
  \upn{(}for example, if $\FF$ is transitive, or if $X$ is
  connected\upn{)}. Then all leaves with trivial holonomy group
  determine the same coarse quasi-isometry type.
\end{theorem}

When $(X,\FF)$ is at least of class $C^2$, the leaves can be endowed
with a metric tensor which is continuous on $X$, and the above result
can be improved to obtain quasi-isometries via diffeomorphisms between
covers of leaves. 

This will be done with the help of the normal
quasi-foliated bundles of Section~\ref{fol.space}. As the previous
results make evident, the problem with the holonomy and topological
structure of the transversal makes it difficult to effectively use the
equicontinuity to push whole leaves onto others. In any case, the
following weaker solution to this problem can be provided.

\begin{theorem}\label{t:equicontinuous foliated spaces}
  Let $(X,\FF)$ be a strongly equicontinuous, compact foliated space
  of class $C^2$ with connected space of leaves \upn{(}for example, if
  $\FF$ is transitive, or if $X$ is connected\upn{)}. Then the
  universal covers of all the leaves are uniformly quasi-isometric via
  diffeomorphisms.
\end{theorem}

\begin{proof}
Let $L'$ be the universal cover of a leaf $L$, and let $N(L')$ be the
normal bundle described in Section~\ref{fol.space}. With respect to
some metric on $N(L')$, which is boundedly distorted via the
embedding, there is a product neighborhood $N(L',\epsilon_0)$ of the
zero section in $N(L')$ carrying a lamination $Y$, as described in
Theorem~\ref{quasibundle}. The projection of the leaves in this
neighborhood onto the zero section $L'$ is locally a diffeomorphism
with bounded distortion, by Proposition~\ref{nor.1}. On the other
hand, the strong equicontinuity of the pseudogroup readily implies that these
projections are actually covering maps.  More precisely, with the
notation of Section~\ref{fol.space}, strong equicontinuity implies that given
$\epsilon>0$ there exists $\delta >0$ such that if $S$ is a leaf of
$Y$ and meets the fiber $p^{-1}(x)$ of some point $x\in L'$ at a point
at distance $<\delta$ from $x$, then it meets every fiber $p^{-1}(y)$
at points at distance $<\epsilon$ from the base $y$. This implies that
the projection $p:S\to L'$ is a covering map, for it is a local
homeomorphism which has the path covering property, the only
obstruction to lifting a path being that such leaf $S$ runs off the
neighborhood $N(L',\epsilon_0)$.  Since $L'$ is simply connected, it
follows that $p$ is a diffeomorphism. That $p$ has bounded distortion
was already discussed in Proposition~\ref{nor.1}.

The above paragraph shows that universal covers of pairs of leaves are
uniformly quasi-isometric if both leaves are close enough. Then the
result follows since the space of leaves is compact and connected.
\end{proof}

\begin{theorem}\label{t:equicontinuous quasi-effective foliated spaces}
  Let $(X,\FF)$ be a strongly equicontinuous, compact and
  quasi-effective
  foliated space of class $C^2$ with connected space of leaves
  \upn{(}for example, if $\FF$ is transitive, or if $X$ is
  connected\upn{)}. Then the holonomy covers of all the leaves are
  uniformly quasi-isometric via diffeomorphisms.
\end{theorem}

\begin{proof}
The proof is similar to the proof of Theorem~\ref{t:equicontinuous
foliated spaces}, but using the holonomy cover $L''$ of a given leaf
$L$ instead of using the universal cover $L'$. As $N(L',\epsilon_0)$
in the above proof, the neighborhood $N(L'',\epsilon_0)$ of the zero
section in the normal bundle $p:N(L'')\to L''$, carrying a lamination
$Y$, satisfies the following property: given $0<\epsilon<\epsilon_0$,
there exists $\delta>0$ such that, if $S$ is a leaf of $Y$ meeting
some fiber $p^{-1}(x)$ of a point $x\in L''$ at a point at distance
$<\delta$ from $x$, then $S$ meets every fiber $p^{-1}(y)$ at a
distance $<\epsilon$ from the base point $y\in L''$, and it follows
that $p:S\to L''$ is a covering map, whose triviality has to be
proved. This would finish the proof because the distortion of $p:S\to
L''$ is uniformly bounded for the leaves $S$ of $Y$, as in the proof of
Theorem~\ref{t:equicontinuous foliated spaces}.

Observe that $Z(x,r)=Y\cap p^{-1}(x)\cap N(L'',\delta)$ is a
transversal of $Y$ through $x$ for any $r\le\epsilon_0$. Then the key
property of the above statement can be stated as follows: if $h$ is
any holonomy map of $Y$ defined on some neighborhood of $x$ in
$Z(x,\delta)$, and whose image is contained in $Z(x,\epsilon)$, then
$h$ can be extended to a holonomy transformation with domain
$Z(x,\delta)$ and image contained in $Z(x,\epsilon)$. Moreover, under
the present hypothesis, this extension can be assumed to be unique.
The only such holonomy transformation is the identity on $Z(x,\delta)$
because $L''$ has trivial holonomy group in $Y$ since it is the
holonomy cover of $L$.  Therefore any such a leaf $S$ meets every
fiber of $p$ at just one point; {\it i.e.}, $p:S\to L''$ is a
diffeomorphism, as desired.
\end{proof}

\begin{rem}
There are versions of Theorems~\ref{t:equicontinuous foliated spaces}
and~\ref{t:equicontinuous quasi-effective foliated spaces} for the
coarse quasi-isometry type of the universal coverings or the holonomy
covers of all leaves when the foliated space is not of class $C^2$; in
particular, this generalizes Theorem~\ref{t:equicontinuous foliated
space}. The coarse quasi-isometry types of these covers can be defined
again via the generators of a representative of the holonomy
pseudogroup induced by a finite defining cocycle: the orbits can be
thought as graphs in an obvious way, and thus the corresponding covers
can be constructed. The coarse quasi-isometry types of such covers can
be proved to be invariant by equivalences when the metrics are induced
by recurrent systems of compact generation. Hence, versions of
Theorems~\ref{t:coarsely quasi-isometric orbits}
and~\ref{t:equicontinuous pseudogroup} for covers of the orbits need
to be proved first. These generalizations are easy to make, but the
required notation becomes complicated; thus they are left to the reader.
\end{rem}

One of the fundamental results of Molino's theory of Riemannian
foliations is the following. The closure of the leaves partition the
manifold into the leaves of a larger singular foliation
\cite{Mol88}. In the general situation considered here, the following
weaker results are available; they follow directly by applying
Theorem~\ref{t:minimal orbit closures}, Corollary~\ref{c:partition}
and Corollary~\ref{c:homogeneous orbit closures} to the holonomy
pseudogroup.

\begin{theorem}\label{t:minimal leaf closures}
Let $(X,\FF)$ be a strongly equicontinuous, compact foliated space. Then the
closure of each leaf is a minimal set. In particular, $\FF$ is minimal
if it is transitive.
\end{theorem}

\begin{cor}\label{c:minimal leaf closures}
  The leaf closures define a partition of any strongly equicontinuous,
  compact foliated space.
\end{cor}

\begin{theorem}\label{t:homogeneus leaf closures}
  Let $(X,\FF)$ be a strongly equicontinuous, compact and
  quasi-effective
  foliated space.  Then the closure of each leaf is a homogeneous
  space
\end{theorem}

The next result shows another geometric aspect of the structure of a strongly
equicontinuous foliated space.  It was shown by
H. Winkelnkemper~\cite{win1} for the holonomy groupoid or graph, and
by F. Alcalde Cuesta \cite{alc} for the homotopy groupoid of a
Riemannian foliation.

\begin{cor}
Let $(X,\FF)$ be a strongly equicontinuous, compact foliated space of class
$C^2$ with connected leaf space. Then the homotopy groupoid of
$(X,\FF)$ is, canonically, a fiber bundle whose structural group can
be reduced to the group of differentiable quasi-isometries of the
typical fiber, this being the universal cover of a leaf. If $\FF$ is
quasi-effective, then the same is true for the holonomy groupoid, the
fiber now being the holonomy cover of a leaf.
\end{cor}

\begin{proof}
The homotopy groupoid $G$ of $(X, \FF)$ consists of equivalence
classes of paths on leaves, two paths $\alpha$ and $\beta$ being
equivalent if they have the same endpoints and the closed loop
$\alpha\beta^{-1}$ is homotopically trivial in the leaf which contains
it. If $[\alpha]$ is a point of $G$, then let $s[\alpha]=\alpha(0)$
denote the source map $s:G\to X$.  The fiber of $s$ over a point $x\in
X$ is $G_x$, and is canonically identified with the universal cover of
the leaf $L_x$. The foliated space being of class $C^1$ means that its
leaves can be endowed with a continuous metric tensor. Let $U$ be a
flow box for $X$ with leaf space $Z$ and such that its plaques are
convex subsets of the leaves with respect to the chosen metric
tensor. Let $x\in U$, and let $P$ be the plaque containing $x$, so
that $U$ is of the form $P\times Z$. Then, if $G_U$ denotes the
restriction of $G$ to $U$, $G_U=s^{-1}U$, there is a map $G_U\to
Z\times G_P$ which sends a point $[\alpha]\in U$ to $(q(\alpha(0)),
[p\o\alpha(0)])$, where $q:U\to Z$ is the projection into the space of
leaves of $U$. Since the plaque $P$ is convex, there is a unique
geodesic path in $P$ joining $x$ to any given point of $P$, so there
is a well defined map $G_P\to P\times G_x$ obtaining by precomposing a
path starting at some $y\in P$ with the unique geodesic in $P$ from
$x$ to $y$. Let $\phi:G_U\to U\times G_x$ denote the composition of
this two maps.  This map is a homeomorphism, and from the previous
work (Theorem~\ref{t:equicontinuous foliated spaces}), it follows that
this map is a quasi-isometry on the fibers, that is, it sends the
fiber $s^{-1}(y)=G_y$ quasi-isometrically onto $G_x$. The
quasi-isometry distortion is bounded, and there is a commutative
diagram $$
\begin{CD}
G_U @>{\phi}>> U\times G_x \\
@VVV @VVV \\
U @= U
\end{CD}
$$
from which the result readily follows.

If the foliated space is quasi-effective, then the same property for the
holonomy groupoid is proved similarly by using Theorem~\ref{t:equicontinuous
quasi-effective foliated spaces}.
\end{proof}

To conclude, it is quite reasonable to expect that the theory
presented in this paper can be extended to include larger classes of
foliated spaces which have some sort of transverse uniformity,
paralleling certain well-known structures of classical topological
dynamics \cite{vee}, {\it e.g.}, distal actions.


\begin{thebibliography}{00}

\bibitem{alc}
F. Alcalde Cuesta, {\em
Groupo\"\i de d'homotopie d'un feuilletage riemannien et r\'ealisation
symplectique de certaines vari\'et\'es de Poisson},
Publ. Mat. {\bf 33} (1989), 395--410.

\bibitem{Alv-Candel:Riemannian}
J.A.~\'Alvarez L\'opez and A. Candel, {\em Topological description of
Riemannian 
foliations with dense leaves}, preprint.

\bibitem{jes} J. A. \'Alvarez L\'opez and A. Candel, {\em Generic
geometry of leaves}, (Forthcomming preprint).

\bibitem{Block-Weinberger}
J. Block and S. Weinberger, {\em Aperiodic tilings, positive scalar
  curvature, and amenability of spaces\/}, J. Amer. Math. Soc. {\bf 5} (1992), 907--918.

\bibitem{CC2K}
A. Candel and L. Conlon, {\em Foliations I\/}, Graduate Studies in
Mathematics, American Mathematical Society, Providence, 2000.

\bibitem{Goodman-Plante} S.E.~Goodman and J.F.~Plante, {\em Holonomy
    and averaging in foliated sets\/}, J.  Diff. Geom. {\bf 14},
  401--407.

\bibitem{Gromov93}
M. Gromov, {\em Asymptotic invariants of infinite groups\/},
Geometric Group Theory, Volume 2, G.A. Niblo, M.A. Roller, eds., 
Cambridge University Press, Cambridge, 1993.

\bibitem{Haefliger85} 
A. Haefliger, {\em Pseudogroups of local
    isometries\/}, Differential Geometry (Santiago de Compostela 1984),
  L.A. Cordero, ed., Research Notes in Math. {\bf 131}, Pitman
  Advanced Pub. Program, Boston, 1985, pp.~174--197.

\bibitem{Haefliger88}
A. Haefliger, {\em Leaf closures in Riemannian foliations}, 
A F\^{e}te on Topology, Y.~Matsumoto, T.~Mizutani and S.~Morita, eds.,
   Academic Press, New York,
1988, pp.~3--32.

\bibitem{Haefliger00}
A. Haefliger, {\em Foliations and compactly generated pseudogroups\/}, preprint, 2001.


\bibitem{Hector-Hirsch}
G. Hector and U. Hirsch, {\em Introduction to the Geometry of
Foliations, Parts~{A} and~{B}\/}, Aspects of Mathematics, vols. {\bf
E1} and {\bf E3},  Friedr. Vieweg and Sohn, Braunschweig, 1981 and 1983.

\bibitem{hirsch} 
M. Hirsch, {\em Differential Topology\/}, Graduate
  Texts in Mathematics {\bf 33}, Springer-Verlag, New York, 1976.

\bibitem{Hurder94}
S. Hurder, {\em Coarse geometry of foliations},
Geometric Study of Foliations (Tokyo 1993), T.~Mizutani, K.~Masuda,
S.~Matsumoto, T.~Inaba, T.~Tsuboi, Y.~Mitsumatsu, eds., 
World scientific Publishing Co. Pte. Ltd., Singapore, 1994, pp.~35--96.

\bibitem{HK1}
S. Hurder and A. Katok, {\em Ergodic theory and Weil measures for
foliations}, Annals of Math. {\bf 126} (1987), 221--275.

\bibitem{Kanai85}
M. Kanai, {\em Rough isometries, and combinatorial approximations of
  geometries of non-compact manifolds\/}, J. Math. Soc. Japan {\bf 37} (1985), 391--413.

\bibitem{kellum} 
M. Kellum,
{\em Uniformly quasi-isometric foliations},
Ergodic Theory Dynamical Systems {\bf 13} (1993), 101--122.

\bibitem{Mol88}
P. Molino, {\em Riemannian Foliations} (with appendices by G.~Cairns,
Y.~Carri\`ere, E.~Ghys, E.~Salem and V.~Sergiescu), Progress in
Mathematics, vol. {\bf 73}, Birkh\"auser, Boston, 1988.

\bibitem{MS}
C.C. Moore and C. Schochet, {\em Global Analysis on Foliated Spaces},
MSRI Publications, vol.~{\bf 9},  Springer-Verlag, New York, 1988.

\bibitem{munkres}
J.R. Munkres, {\em Topology: a First Course\/}, Prentice-Hall, Inc., Englewood
Cliffs, N.J., 1975.



\bibitem{nagata} 
J. Nagata, {\em Modern General Topology}, (Second,
revised edition), Noth-Holland Publishing Company, Amsterdam, 1974.

\bibitem{robinson}
A. Robinson, {\em Non-standard Analysis}, (Reprint of the 1974
Edition), Princeton University Press, Princeton, N.J., 1996.

\bibitem{Plante}
J.F.~Plante, {\em Foliations with measure preserving holonomy\/},
Ann. of Math. {\bf 
102} (1975), 327--361.


\bibitem{Roe93}
J. Roe, {\em Coarse Cohomology and Index Theory on Complete Riemannian
  Manifolds\/}, 
Memoirs of the Amer. Math. Soc., Number 497, 1993.

\bibitem{Sak} R. Sacksteder, {\em Foliations and pseudogroups},
American J. of Math. {\bf 87} (1965), 79--102.

\bibitem{Smirnov}
Y.M. Smirnov, {\em On metrization of topological spaces\/}, Amer. Math. Soc.
Translations. Series One {\bf 8} (1953), 62--77.

\bibitem{Counterexamples}
L.A. Steen and J.A. Seebach, Jr., {\em Counterexamples in Topology\/}, (Second
Edition), Springer-Verlag, New York-Heidelberg, 1978. 

\bibitem{Tarquini} C.~Tarquini, {\em Feuilletages conformes\/}, Ann.
  Inst. Fourier {\bf 54} (2004),  453--480.


\bibitem{vee} 
W. Veech, {\em Topological dynamics},
Bull. Amer. Math. Soc. {\bf 83} (1977), 775--830.

\bibitem{weil} 
A. Weil, {\em L'Integration dans les Groupes
Topologiques et ses Applications\/} (Second Edition),  Actualit\'{e}s
Scientifiques et Industrielles, no. 1145. Publications de l'Institut de
Mathematique de l'Universite de  Strasbourg  4, Hermann, Paris,  1951.

\bibitem{Willard}
S. Willard, {\em General Topology\/}, Addison-Wesley Publishing Co.,
Reading, Mass. 
1970.


\bibitem{win1}
H. E.  Winkelnkemper, {\em The graph of a
 foliation}, Ann. Global Anal. Geom. {\bf 1} (1983), 51--75.



\end{thebibliography}
\end{document}